\documentclass[11pt, reqno]{amsart}

\newcommand{\chapter}{\part}
\usepackage[usenames,dvipsnames]{xcolor}

\usepackage[OT2, T1]{fontenc}
\usepackage{url}
\usepackage{amsmath}
\usepackage{array}
\usepackage{graphicx}
\usepackage{amssymb}
\usepackage{amsthm}
\definecolor{link}{RGB}{11,0,128}
\usepackage[colorlinks=true,citecolor=NavyBlue,linkcolor=OliveGreen,urlcolor=link]{hyperref}
\usepackage{hyperref}
\usepackage{paralist}
\usepackage{colonequals}			
\usepackage{color}
\usepackage{mathabx}			
\usepackage{amscd}
\usepackage[all,cmtip]{xy}			
\usepackage{sseq}				
\usepackage{verbatim}
\usepackage{parskip}
\usepackage{microtype}
\usepackage[in]{fullpage}
\usepackage{mathrsfs} 			
\usepackage{cleveref}
\usepackage{etoolbox}
\usepackage{enumitem}
\usepackage{moreenum}			
\usepackage[alphabetic,lite]{amsrefs} 	
\usepackage{stmaryrd}			
\usepackage[foot]{amsaddr}
\usepackage{subfiles}
\usepackage{ragged2e}
\usepackage{tikz}
\usepackage{cleveref}
\usetikzlibrary{graphs}
\usetikzlibrary{cd}
\usepackage{stackengine}
\usepackage{lipsum}

\calclayout
\def\cond{\mathrm{cond}}

\def\iHom{\underline{\mathrm{Hom}}}

\def\sol{\blacksquare}

\def\soten{\tensor^{\sol}}

\newcommand{\bA}{\mathbb{A}}

\newcommand{\bC}{\mathbb{C}}
\newcommand{\bD}{\mathbb{D}}

\newcommand{\bN}{\mathbb{N}}

\newcommand{\bZ}{\mathbb{Z}}

\newcommand{\bperf}{\pmb{\mathrm{Perf}}}
\newcommand{\bnuc}{\pmb{Nuc}}

\newcommand{\cA}{\mathcal{A}}
\newcommand{\cB}{\mathcal{B}}
\newcommand{\cC}{\mathcal{C}}
\newcommand{\cD}{\mathcal{D}}
\newcommand{\cE}{\mathcal{E}}
\newcommand{\cF}{\mathcal{F}}
\newcommand{\cG}{\mathcal{G}}
\newcommand{\cH}{\mathcal{H}}

\newcommand{\cK}{\mathcal{K}}

\newcommand{\cM}{\mathcal{M}}

\newcommand{\cO}{\mathcal{O}}
\newcommand{\cP}{\mathcal{P}}

\newcommand{\cR}{\mathcal{R}}
\newcommand{\cS}{\mathcal{S}}
\newcommand{\cT}{\mathcal{T}}

\newcommand{\cW}{\mathcal{W}}
\newcommand{\cX}{\mathcal{X}}
\newcommand{\cY}{\mathcal{Y}}
\newcommand{\cZ}{\mathcal{Z}}

\usepackage{relsize}
\usepackage[bbgreekl]{mathbbol}
\usepackage{amsfonts}

\DeclareSymbolFontAlphabet{\mathbb}{AMSb}
\DeclareSymbolFontAlphabet{\mathbbl}{bbold}

\DeclareMathOperator{\Affine}{\mathrm{Affine}}
\DeclareMathOperator{\Ani}{Ani}		
\DeclareMathOperator{\Hom}{Hom}		
\DeclareMathOperator{\rqcoh}{\mathrm{Qcoh}}
\DeclareMathOperator{\AnRing}{AnRing}
\DeclareMathOperator{\AffRing}{AffRing}
\DeclareMathOperator{\AniAlg}{AniAlg}
\DeclareMathOperator{\extdis}{Extdis}

\DeclareMathOperator{\Aff}{Aff}
\DeclareMathOperator{\Anspec}{AnSpec}

\DeclareSymbolFont{cyrletters}{OT2}{wncyr}{m}{n}
\DeclareMathSymbol{\Sha}{\mathalpha}{cyrletters}{"58}	
\DeclareMathOperator{\Sp}{Sp}		
\DeclareMathOperator{\Spa}{Spa}		
\DeclareMathOperator{\Spec}{Spec}
\DeclareMathOperator{\colim}{\mathop{colim}}
\DeclareMathOperator{\overdisc}{\mathbb{D}^{\dagger}}
\DeclareMathOperator{\closedisc}{\Bar{\mathbb{D}}}
\DeclareMathOperator{\disc}{\mathbb{D}}
\DeclareMathOperator{\rind}{\mathrm{Ind}}
\DeclareMathOperator{\rmod}{\mathrm{Mod}}
\DeclareMathOperator{\rfun}{\mathrm{Fun}}
\DeclareMathOperator{\rperf}{\mathrm{Perf}}

\newcommand{\alg}{\mathrm{alg}}		

\newcommand{\et}{\mathrm{\acute{e}t}}	

\newcommand{\isomto}{\overset{\sim}{\longrightarrow}}

\newcommand{\op}{{\mathrm{op}}}			

\newcommand{\rmap}{\mathrm{Map}}    
\newcommand{\rsh}{\mathrm{Sh}}
\newcommand{\rpsh}{\mathrm{PSh}}

\newcommand{\surjects}{\twoheadrightarrow}
\newcommand{\tensor}{\otimes} 			

\newcommand{\csub}[1]{\centering\subsection{\text{#1} \nopunct} \hfill \justify}

\providecommand{\up}[1]{{\upshape(}#1{\upshape)}}

\renewcommand{\b}{\textbf}

\newcommand{\brems}{\begin{rems} \hfill \begin{enumerate}[label=\b{\thenumberingbase.},ref=\thenumberingbase]}

\newcommand{\erems}{\end{enumerate} \end{rems}}
\newcommand{\begs}{\begin{egs} \hfill \begin{enumerate}[label=\b{\thenumberingbase.},ref=\thenumberingbase]}

\newcommand{\eegs}{\end{enumerate} \end{egs}}

\newcommand{\bsm}{\begin{smallmatrix}}
\newcommand{\esm}{\end{smallmatrix}}
\newcommand{\blem}{\begin{lemma}}
\newcommand{\elem}{\end{lemma}}

\newcommand{\bconj}{\begin{conj}}
\newcommand{\econj}{\end{conj}}
\newcommand{\bprob}{\begin{Problem}}
\newcommand{\eprob}{\end{Problem}}

\newcommand{\bq}{\begin{Q}}
\newcommand{\eq}{\end{Q}}
\newcommand{\benum}{\begin{enumerate}[label={{\upshape(\alph*)}}]}
\newcommand{\benuma}{\begin{enumerate}[label={{\upshape(\arabic*)}}]}
\newcommand{\benumb}{\begin{enumerate}[label={{\upshape\b{\arabic*.}}}]}
\newcommand{\benumr}{\begin{enumerate}[label={{\upshape(\roman*)}}]}
\newcommand{\eenum}{\end{enumerate}}
\newcommand{\bitem}{\begin{itemize}}
\newcommand{\eitem}{\end{itemize}}
\newcommand{\bc}{}
\newcommand{\bd}{\begin{defn}}
\newcommand{\ed}{\end{defn}}

\newcommand{\beg}{\begin{eg}}
\newcommand{\eeg}{\end{eg}}

\newcommand{\bcl}{\begin{claim}}
\newcommand{\ecl}{\end{claim}}

\newcommand{\ba}{\begin{aligned}}
\newcommand{\ea}{\end{aligned}}
\newcommand{\be}{\begin{equation}}
\newcommand{\ee}{\end{equation}}
\newcommand{\bpf}{\begin{proof}}
\newcommand{\epf}{\end{proof}}
\newcommand{\bthm}{\begin{thm}}
\newcommand{\ethm}{\end{thm}}
\newcommand{\bprop}{\begin{prop}}
\newcommand{\eprop}{\end{prop}}
\newcommand{\bcor}{\begin{cor}}
\newcommand{\ecor}{\end{cor}}
\newcommand{\brem}{\begin{rem}}
\newcommand{\erem}{\end{rem}}

\usepackage{aliascnt}
\newaliascnt{numberingbase}{subsubsection}
\numberwithin{equation}{numberingbase}

\newtheoremstyle{thms}{0.5em}{0.5em}{\itshape}{}{\bfseries}{.}{ }{}
\theoremstyle{thms}
\newtheorem{conj}[numberingbase]{Conjecture}

\newtheorem{cor}[numberingbase]{Corollary}

\newtheorem{lemma}[numberingbase]{Lemma}

\newtheorem{prop}[numberingbase]{Proposition}

\newtheorem{Q}[numberingbase]{Question}
\newtheorem{thm}[numberingbase]{Theorem}
\newtheorem{theorem}[numberingbase]{Theorem}

\newtheorem{defprop}[numberingbase]{Definition/Proposition}
\newtheoremstyle{claims}{0.5em}{0.5em}{}{}{\itshape}{.}{ }{}
\theoremstyle{claims}
\newtheorem{claim}[equation]{Claim}

\newtheoremstyle{defs}{0.5em}{0.5em}{}{}{\bfseries}{.}{ }{}
\theoremstyle{defs}
\newtheorem{defn}[numberingbase]{Definition}
\newtheorem{definition}[numberingbase]{Definition}
\newtheorem{Construction}[numberingbase]{Construction}
\newtheorem{propcons}[numberingbase]{Proposition/Construction}
\newtheorem{defcons}[numberingbase]{Definition/Construction}

\newtheorem{eg}[numberingbase]{Example}
\newtheorem*{egs}{Examples}
\newtheorem{rem}[numberingbase]{Remark}

\newtheorem*{rems}{Remarks}
\newtheorem*{remw}{Remark}
\Crefname{claim}{Claim}{Claims}
\Crefname{bclaim}{Claim}{Claims}
\Crefname{sublemma}{Lemma}{Lemmas}
\Crefname{conj}{Conjecture}{Conjectures}
\Crefname{cor}{Corollary}{Corollaries}
\Crefname{defn}{Definition}{Definitions}
\Crefname{eg}{Example}{Examples}
\Crefname{prop}{Proposition}{Propositions} 
\Crefname{Q}{Question}{Questions}
\Crefname{rem}{Remark}{Remarks}
\Crefname{thm}{Theorem}{Theorems}
\Crefname{Theorem}{Theorem}{Theorems}
\Crefname{variant}{Variant}{Variants}
\Crefname{caution}{Caution}{Cautions}
\Crefname{thmenumi}{Theorem}{Theorems}
\AtBeginEnvironment{thm}{%
\crefalias{enumi}{thmenumi}%
\setlist[enumerate,1]{label={\textit{(\arabic*)}},ref={\thethm.(\arabic*)}}}
\Crefname{propenumi}{Proposition}{Proposition}
\AtBeginEnvironment{prop}{%
\crefalias{enumi}{propenumi}%
\setlist[enumerate,1]{label={\textit{(\arabic*)}},ref={\theprop.(\arabic*)}}}

\theoremstyle{thms}
\newtheorem{thm-tweak}[subsection]{Theorem}
\Crefname{thm-tweak}{Theorem}{Theorems}
\newtheorem{lemma-tweak}[subsection]{Lemma}
\Crefname{lemma-tweak}{Lemma}{Lemmas}
\newtheorem{cor-tweak}[subsection]{Corollary}
\Crefname{cor-tweak}{Corollary}{Corollaries}
\newtheorem{prop-tweak}[subsection]{Proposition}
\Crefname{prop-tweak}{Proposition}{Propositions} 
\newtheorem{conj-tweak}[subsection]{Conjecture}
\Crefname{conj-tweak}{Conjecture}{Conjectures} 
\newtheorem{q-tweak}[subsection]{Question}
\Crefname{q-tweak}{Question}{Questions} 

\theoremstyle{defs}
\newtheorem{defn-tweak}[subsection]{Definition}
\Crefname{defn-tweak}{Definition}{Definitions}
\newtheorem{eg-tweak}[subsection]{Example}
\Crefname{eg-tweak}{Example}{Examples}
\newtheorem*{rems-tweak}{Remarks}
\newtheorem{rem-tweak}[subsection]{Remark}
\Crefname{rem-tweak}{Remark}{Remarks}

\newtheoremstyle{subsection-tweak}
   {2pt}
   {3pt}%
   {}
   {}%
   {\bfseries}
   {}%
   {.5em}
   {\thmnumber{\@{#1}{}\@{#2}.}%
    \thmnote{~{\bfseries#3.}}}    
    
\theoremstyle{subsection-tweak}
\newtheorem{pp}[numberingbase]{}
\newcommand{\bpp}{\begin{pp}}
\newcommand{\epp}{\end{pp}}

\theoremstyle{subsection-tweak}
\newtheorem{pp-tweak}[subsection]{}

\makeatletter
\def\@tocline#1#2#3#4#5#6#7{
    \begingroup 
    \@ifempty{#4}{}{}

    \parindent\z@ \leftskip#3\relax \advance\leftskip\@tempdima\relax
    #5\hskip-\@tempdima
      \ifcase #1
       \or\or \hskip 2em \or \hskip 1em \else \hskip 3em \fi%
      #6\nobreak\relax
    \dotfill\hbox to\@pnumwidth{\@tocpagenum{#7}}\par
    \nobreak
    \endgroup
 }
 \def\l@section{\@tocline{1}{0pt}{1pc}{}{}}

\renewcommand{\tocsection}[3]{%
  \indentlabel{\@ifnotempty{#2}{\makebox[1.3em][l]{%
    \ignorespaces#1 \bfseries{#2}.\hfill}}}\bfseries{#3}
    \vspace{-5pt}}

\renewcommand{\tocsubsection}[3]{%
  \indentlabel{\@ifnotempty{#2}{\hspace*{-0.5em}\makebox[2.1em][l]{%
    \ignorespaces#1#2.\hfill}}}#3
    \vspace{-5pt}}
   
\makeatother 

\makeatletter

\newcounter{hiddenss}[subsection]
\renewcommand{\thehiddenss}{\thesubsection.\arabic{hiddenss}}

\newcommand{\hiddensubsubsection}[1]{%
  \refstepcounter{hiddenss}
  \par\vspace{1.5ex}
  \noindent\textbf{\thehiddenss\quad #1}\par
  \vspace{0.5ex}
}

\makeatother

\makeatletter 
\newcommand\appendix@section[1]{%
  \refstepcounter{section}%
  \orig@section*{Appendix \@Alph\c@section. #1}%
}
\let\orig@section\section
\g@addto@macro\appendix{\let\section\appendix@section}
\makeatother
\makeatletter
\def\l@subsubsection{\@tocline{3}{1.8em}{3.2em}{\small}{}}
\makeatother

\title{The relative GAGA Theorem and an application to the analytic mapping stacks}
\author{Qixiang Wang}
\address{Universit\'{e} Paris-Saclay,   Laboratoire de math\'{e}matiques d'Orsay, F-91405, Orsay, France}
\email{qixiang.wang@universite-paris-saclay.fr}
\date{September 6, 2025}

\begin{document}

\maketitle
\hypersetup{
    linktoc=page,     
}

\renewcommand*\contentsname{}
\begin{abstract}
We prove a relative GAGA theorem for perfect and pseudo-coherent complexes in non-archimedean analytic geometry, allowing bases given by Fredholm analytic rings, including those associated from affinoid perfectoid spaces. This answers a question raised in \cite{heuer2024padicnonabelianhodgetheory}. As an application, we show that for a proper scheme \(X\) and an Artin stack \(Y\) with suitable conditions, the analytification of the algebraic mapping stack \(\mathrm{Map}(X,Y)\) agrees with the intrinsic analytic mapping stack \(\mathrm{Map}(X^{\mathrm{an}},Y^{\mathrm{an}})\).
\end{abstract}

\tableofcontents
\newpage

\section{Introduction}
\csub{Analytification of mapping stacks}

In algebraic geometry, given two stacks $X$, $Y$, one may define the \textit{mapping stack} $\underline{\mathrm{Map}}(X,Y)$ as $\underline{\mathrm{Map}}(X,Y) (A)= \mathrm{Map}(X\times \Spec A, Y)$ for any affine scheme $\Spec A$. This simple construction in fact encodes many moduli problems (e.g., Picard stack and moduli stack of vector bundles). 

In rigid analytic geometry, one attempts to do the same, where the relevant notion of ‘‘analytic stacks'' in the rigid analytic setting was introduced in \cite{camargo2024analyticrhamstackrigid} and \cite{anschütz2025analyticrhamstacksfarguesfontaine}. More precisely, fixing a non-archimedean field $K$, given two \textit{Gelfand stacks} (see precise definition in \Cref{Gelfand stack}) $\cX$, $\cY$ over $K$, one can form the \textit{mapping stack} $\underline{\mathrm{Map}}(\cX,\cY)$ simply as $\underline{\mathrm{Map}}(\cX,\cY)(\cA)=\mathrm{Map}(X\times_{\mathrm{GSpec} K} \mathrm{GSpec}\cA,\cY)$ for any affinoid test object $\mathrm{GSpec} \cA$ over $\Spa K$. However when $\cX$ and $\cY$ are the \textit{analytifications} \up{see \Cref{analytification def}} of algebraic stacks $X$ and $Y$ over $K$, there is a potential alternative natural definition of mapping stacks between them, that is, $\underline{\mathrm{Map}}(X,Y)^{\mathrm{an}}$ the \textit{analytification} of $\underline{\mathrm{Map}}(X,Y)$. In many cases, the latter is easier to control, by applying representability of algebraic stacks.

In this paper, we show that in many cases (e.g., $X$ is proper and $Y$ is a quotient stack), these two definitions are the same:
\begin{thm}[\Cref{analytification}]\label{intro main}
    Let $X$ be a proper scheme over $K$, and let $Y$ be a \textit{perfectly Tannakian} \textit{Artin stack} of \textit{geometric nature} \up{\textup{\Cref{perfectly tannakian}}}. Suppose that $\underline{\rmap} (X, Y)$ is an \textit{Artin stack}. Then we have\upshape:
    \[
    \underline{\rmap} (X^{\mathrm{an}}, Y^{\mathrm{an}}) \simeq \underline{\rmap} (X, Y)^{\mathrm{an}}.
    \]
\end{thm}

    One current unsatisfying point of the state of the art in the theory of analytic stacks is that it lacks the notion of a ‘‘geometric stack'' (i.e., an Artin stack), More precisely, even though abstractly one can define it, it lacks an abstract criterion to show that some stack $Z$ admit a ``smooth affinoid chart'': a $!$-surjection $\pi: \underset{i}{\bigcup} Z_{i}=\Spa \cA_i \surjects Z$ such that $\pi$ is representable (or $!$-locally representable) and ``\textit{solid smooth}'' (\cite{camargo2024analyticrhamstackrigid}*{Definition 3.5.5}).
    
    The main motivation of this paper is to prove ‘‘geometricity'' of certain very basic analytic stacks, e.g., those mapping stacks stated in the \Cref{intro main}. Indeed, in algebraic geometry, any mapping stack of proper schemes $X$ of finite Tor-dimension (sometimes stacks even) towards an Artin stack $Y$ is a geometric stack (by the presence of cotangent complex, see \cite{Adeel}*{Theorem 10.6}). 
    As an easy consequence of \Cref{intro main}, the (analytic) moduli stack of $G$-bundles on smooth proper rigid analytic curves are “Artin'' in the sense above (posteriorly, the analytic cotangent complex descends along with this cover \cite{camargo2024analyticrhamstackrigid}*{Theorem 3.5.6}).

The strategy we are going to use to prove \Cref{intro main} is to apply Tannaka duality (originated from \cite{Dag3}). Given any test affine object $\Anspec (\cA) $, by definition, on the right hand side we have that 
$$ \underline{\rmap}(X^{\mathrm{an}}, Y^{\mathrm{an}})(\cA)\simeq\mathrm{Map}(X^{\mathrm{an}}\times \Anspec \cA, Y^{\mathrm{an}}).$$
And on the left hand side (under some conditions), it is very close to 
$$\mathrm{Map}(X\times \Spec \underline{\cA}(*),Y).$$
 Via the pullback functor on ‘‘quasi-coherent sheaves'', both sides admit maps (in $\Ani$) as follows
 $$\mathrm{Map}(X^{\mathrm{an}}\times \Anspec \cA, Y^{\mathrm{an}})\rightarrow \mathrm{Fun}^{L}(\rind(\rperf(Y)), \rind(\bperf(X^{\mathrm{an}}\times \Anspec \cA)))$$
 $$\mathrm{Map}(X\times \Spec \underline{\cA}(*),Y)\rightarrow \mathrm{Fun}^{L}(\rind(\rperf(Y),\rind(\rperf(X\times \Spec \underline{\cA}(*)))).$$
 Here, $\bperf(-)$ stand for the $\infty$-(sub)category consisting of dualizable objects.
 Later on we will show that both functors are fully faithful. So it is then natural that, in order to compare two mapping space, we first identify $\bperf(X^{\mathrm{an}}\times\Anspec \cA)$ and $\rperf(X\times \Spec \underline{\cA}(*))$, which requires a generalized version of relative GAGA result. For this reason, we will restrict our test objects to those analytic ring $\cA$ are \textit{Fredholm bounded affinoid algebras}. 
 
 Here the notion of bounded affinoid algebra was defined in \cite{camargo2024analyticrhamstackrigid}*{Definition 2.6.10}, and we will briefly review the relevant theory (\Cref{non-arch ana stack}), one should think it as the associated norm function $|f|:\Anspec \cA\rightarrow [0,\infty]$ of every element in $\underline{\cA}(*)$ is a bounded function  (with image in $[0,\infty)$). Typical examples are analytic rings associated to complete analytic Huber pairs.
 
\csub{Relative GAGA theorem}
    In \cite{heuer2024padicnonabelianhodgetheory}*{Remark 8.1.2} there is a question posed\upshape: for a proper scheme $X$ over $K$ a complete $p$-adic field, if it is true that for any $S=\Spa (A,A^+)$ perfectoid space over $K$, we have that every vector bundle on $X^{\mathrm{an}}\times_{\Spa K} S$ is \textit{algebraic}, that is, coming from a vector bundle on $X\times_{\Spec K} \Spec A$ via pullback functor. 
    
    This will be a special case of the relative GAGA theorem \cite{MR422671} if $S$ is a \textit{Tate} affinoid space in the sense of classical rigid geometry instead of a perfectoid space. Indeed, relative GAGA theorem relative to such (with non-finite type nature) objects are missing in the literature.
\begin{definition}
    An analytic ring $\cA$ is called $Fredholm$ if all the dualizable objects in $D(\cA)$ are discrete. Equivalently, $\cA$ is \textit{Fredholm} if $\bperf(\cA)\simeq\mathrm{Perf}(\underline{\cA}(*))$. 
\end{definition}

    A ``non-classical'' example of \textit{Fredholm bounded affinoid algebra} is the analytic ring associated to a perfectoid space $S=\Spa (A,A^+)$.

 The following theorem answers (a vast generalization of) the question in \cite{heuer2024padicnonabelianhodgetheory}*{Remark 8.1.2}.
\begin{thm}[\Cref{main}] \label{thm1}
    Let $K$ be a nonarchimedean field, $X/K$ a proper scheme, and $\mathcal{A}/K$ a Fredholm bounded affinoid algebra. Denote by $X^{\mathrm{an}}$ the analytification \up{as \textit{Tate stack} , see \Cref{analytification def}} of $X$ with respect to $K$. Then the pullback functor induces an equivalence of categories\upshape:
    \[
    \rperf(X \times_{\Spec{K}} \Spec{\underline{\cA}(*)}) \isomto \bperf(X^{\mathrm{an}} \times_{\Spa (K,K^+)} \Spa\mathcal{A}).
    \]
    Analogous statement for Gelfand rings also holds. 
\end{thm} 

Notice that here the condition of $\cA$ being Fredholm is necessary, since when $X=\Spec K$, the above statement would force $\bperf(\cA)\simeq \mathrm{Perf}(\underline{\cA}(*))$, which implies that $\cA$ is Fredholm.

It turns out that with the notion of ``quasi-coherent sheaves'' in analytic geometry introduced by Clausen-Scholze \cite{analytic}, proving such a theorem—along with an enhanced version of the classical GAGA theorem—can be accomplished purely categorically.

The key ingredient in our proof (of \Cref{thm1}) is the categorical Künneth formula for solid modules, combined with an enhanced version of the GAGA theorem (\Cref{newGAGA}). The categorical Künneth formula is proved (as a special case) in \cite{kesting2025categoricalkunnethformulasanalytic}:

\begin{thm}[\cite{kesting2025categoricalkunnethformulasanalytic}] \label{kunnethnuc}
    In the setting above, we have the following equivalence of symmetric monoidal $\infty$-categories\upshape:
    \[
    D_{\sol}(X^{\mathrm{an}}) \tensor_{D_{\sol}(K, K^+)} D_{\sol}(\cA) \isomto D_{\sol}(X^{\mathrm{an}} \times_{\Spa(K, K^+)} \Spa\mathcal{A}).
    \]
\end{thm}

The relevant formalism of Lurie's tensor product of presentable $\infty$-categories (\cite{SAG}*{section 4.8}) will be reviewed in \Cref{sect2}.

\begin{rem} \label{remark1} 
    By \cite{benzvi2010integraltransformsdrinfeldcenters}*{Proposition 4.6}, the right-hand side in \Cref{thm1} is equivalent to the tensor product of categories $\rperf(X) \tensor_{\rperf(K)} \rperf(\Spec \underline{\cA}(*))$. Moreover, by the classical GAGA theorem combined with  $Fredholm$ property of $\cA$, this is also equivalent to the tensor product $\bperf(X^{\mathrm{an}}) \tensor_{\bperf(k)} \bperf(\cA)$. Thus, \Cref{thm1} essentially establishes a categorical Künneth formula for perfect complexes on $X^{\mathrm{an}} \times_{\Spa(K, K^+)} \Spa\mathcal{A}$.
\end{rem}

A naive approach to proving \textit{Künneth} formula for categories of dualizable objects would be to argue that dualizable objects commute with Lurie's tensor product, and then apply \Cref{kunnethnuc}. Unfortunately, this is false in general. 

\begin{eg}
Let $\cA=(\mathbb{Q}_p\langle T \rangle,\mathbb{Z}_p\langle T \rangle)_{\sol}$ to be the closed unit disk over $\mathbb{Q}_p$. Consider the category of dualizable objects in $D_{\sol}(\cA)\tensor_{D_{\sol}(\mathbb{Q}_p)}D_{\sol}(\cA)(\simeq D_{\sol}(\cA\tensor_{\mathbb{Q}_{p\sol}}
\cA)$). 

Then we know (and we will prove later) that $\cA\tensor_{\mathbb{Q}_{p\sol}}\cA\simeq (\mathbb{Q}_p\langle T_1,T_2 \rangle, \mathbb{Z}_p\langle T_1,T_2 \rangle)_{\sol}$, and it is a Fredholm anlalytic ring. Thus we have that $$\bperf(\cA\tensor_{\mathbb{Q}_{p\sol}}\cA)\simeq \rperf(\mathbb{Q}_p\langle T_1,T_2\rangle).$$ But on the other hand, by Künneth formula of perfect complexes on schemes (\cite{benzvi2010integraltransformsdrinfeldcenters}*{Proposition 4.6}), we have that $$\bperf(\cA)\tensor_{\bperf(\mathbb{Q}_p)}\bperf (\cA)\simeq \rperf(\underline{\cA}(*)\tensor_{\mathbb{Q}_p}\underline{\cA}(*))\simeq \rperf(\mathbb{Q}_p\langle T_1\rangle \tensor_{\mathbb{Q}_p} \mathbb{Q}_p\langle T_2\rangle)$$ which is not equivalent to $\rperf (\mathbb{Q}_p \langle T_1, T_2 \rangle) $ since $\mathbb{Q}_p\langle T_1\rangle \tensor_{\mathbb{Q}_p} \mathbb{Q}_p\langle T_2\rangle \neq \mathbb{Q}_p\langle T_1, T_2 \rangle$ (here it is algebraic tensor product!).  
  
\end{eg}

However, with the following enhancement of the GAGA theorem, this naive approach essentially works in our specific setting.

\begin{thm}[Enhanced GAGA] \label{newGAGA}
    In the setting above, we have\upshape:
    \[
    D_{\sol}(X^{\mathrm{an}}) \isomto \rqcoh(X) \tensor_{\rqcoh(\Spec K)} D_{\sol}(K, K^+).
    \]
\end{thm}

Given this result, the proof of \Cref{thm1} crucially relies on the fact that $\rqcoh(X)$ is compactly generated and rigid, and we have an equivalence of categories \cite{Thomason1990HigherAK}\upshape: 
\begin{equation}\label{indperf}
   \rqcoh(X) \isomto \rind(\rperf(X)). 
\end{equation}

In fact, In order to check the functor $\bperf(X^{\mathrm{an}})\tensor_{\bperf(K)}\bperf(\cA)\rightarrow \bperf(X^{\mathrm{an}}\times_{\Spa(K, K^+)}\Spa \cA)$ is essential surjective, by formal resaon, it is equivalent to check that the \textit{Fourier-Mukai} functor:
$$\begin{aligned}
    \pmb{\mathrm{FM}}\colon \bperf(X^{\mathrm{an}}\times_{\Spa(K, K^+)}\Spa \cA)&\rightarrow \mathrm{Fun}_{D_{\sol}(K, K^+)}^{L}(D_{\sol}(X^{\mathrm{an}}), D_{\sol}(\cA))\\
              \quad E &\mapsto (M\mapsto\pi_{2 *}(\pi_{1}^{*}M\tensor E))
\end{aligned}$$
 is conservative. But by the presense of \Cref{indperf} one can see that $\pmb{\mathrm{FM}}$ actually factors through $\mathrm{Fun}_{\rperf(K)}^{ex}(\rperf(X),\rperf(\underline{\cA}(*)))$. The conservativity then becomes easier to check.  \\    \\  
In classical rigid analytic geometry, relative GAGA theorem states the following:

\begin{thm}[\cite{MR422671}, \cite{AIF_2006__56_4_1049_0}, \cite{kedlaya2015relativepadichodgetheory}*{Theorem 2.7.7}, \cite{andreychev2021pseudocoherentperfectcomplexesvector}*{Theorem 1.4}]\label{gaga of coh}
   Let $K$ be a nonarchimedean field. Let $X/K$ be a proper scheme and $(A,A^+)$ be a sheafy analytic Huber pair over $\Spa(K, K^+)$. Then we have the following equivalence of categories\upshape: 
    $$\mathrm{Coh}(X^{\mathrm{an}}\times_{\Spa(K, K^+)} \Spa(A,A^+))\simeq \mathrm{Coh}(X\times_{\Spec K} \Spec A).$$
\end{thm}
In fact \cite{MR422671} proved a more general version of GAGA theorem. Namely, for any $\cK$ Tate affinoid algebra, denote its underlying ring by $K$, one has a relative analytification functor which sending a scheme $X$ over $K$ to a rigid analytic space $X^{\mathrm{an},\cK}$ over $\cK$. Then one has $\mathrm{Coh}(X^{\mathrm{an},\cK})\simeq \mathrm{Coh}(X).$ And it is well known to experts that there is also a derived enhancement, namely $\mathrm{Pcoh}(X)\simeq \mathrm{Pcoh}(X^{\mathrm{an},\cK}).$ Here $\mathrm{Pcoh}$ denote the $\infty$-category of pseudo coherent objects (see \Cref{pseudo-coherent}, also know as almost perfect complexes). 

One may argue that \Cref{thm1} is not a full version of (generalization) of relative GAGA theorem, indeed, since \Cref{thm1} only dealt with perfect complexes while the classical version (implicitly) dealt with almost perfect complexes. Due to certain technicality of the very definition of nuclear modules for general analytic stack, we will only be able to push relative GAGA theorem to the generality of almost perfect complexes in the case where $\cA$ is a analytic ring over $\cK$ whose underlying condensed ring $\underline{\cA}$ is a nuclear $
\cK$-module. Which is the following:
\begin{theorem}[\Cref{GAGA of pseudo}]\label{thm2}
    Let $X/K$ be a proper scheme, then the pullback functor $j_X^{*}\colon\mathrm{Qcoh}(X)\rightarrow D_{\sol}(X^{an,\cK})$ induces
    an equivalence of categories $$\mathrm{Pcoh}(X^{\mathrm{an},\cK})\simeq \mathrm{Pcoh}(X).$$
    Here the notion of $\mathrm{Pcoh}(X^{\mathrm{an},\cK})$  stands for ``pseuodo-coherent complexes'' on $X^{\mathrm{an},\cK}$ for which will be defined in \upshape\Cref{def of pseudo}.
\end{theorem}
\begin{remw}
  Although at this generality (assuming $\underline{\cK}$ being nuclear), the proof of \Cref{thm2} is essentially a repetition  of arguments in \cite{complex}. For the completeness we will present the proof in details but we claim no originality in this work.  
\end{remw}

As a simple application of \Cref{intro main}, we deduce the following result of cotangent complex on analytic stacks.
\begin{thm}
     Let $X$ be a proper scheme over $K$, and let $Y$ be a \textit{perfeclty Tannakian} \textit{Artin stack} of \textit{geometric nature}. Suppose that $\underline{\rmap}(X, Y)$ is an \textit{Artin stack}. Then we have that $\underline{\rmap}(X^{\mathrm{an}}, Y^{\mathrm{an}})$ admits a relative cotangent complex \textup{(see \upshape\Cref{def of cotangent})} $\mathbb{L}_{\underline{\rmap}(X^{\mathrm{an}}, Y^{\mathrm{an}})/\Spa{K,K^{+}}}$ and the canonical map $$j_{\underline{\rmap}(X^{\mathrm{an}}, Y^{\mathrm{an}})}^{*}\mathbb{L}_{\underline{\mathrm{Map}}(X,Y)/\Spec K}\rightarrow \mathbb{L}_{\underline{\rmap}(X^{\mathrm{an}}, Y^{\mathrm{an}})/\Spa{K,K^{+}}}$$ is an isomorphism. Here, $j_{\underline{\rmap}(X^{\mathrm{an}}, Y^{\mathrm{an}})}\colon\underline{\rmap}(X, Y)^{\mathrm{an}}\rightarrow \underline{\mathrm{Map}}(X,Y)$ denote the counit map of \textit{analytification} functor.
\end{thm}
In particular, this justifies the following\upshape: for $X$ smooth proper rigid analytic curve over $\Spa(K, K^+)$, $G/K$ a reductive group. The moduli stack $\mathrm{Higgs}_{G}(X)$ of $G$-Higgs bundles on $X$, is the (total space of) cotangent complex of $\mathrm{Bun}_{G}(X)$.

As a final remark, the fact that we are working with Gelfand stack is partially due to artificial reason: in fact we only need to work with the category of certain version of ``analytic stacks'' that forms a $\infty$-topos with the test objects being Fredholm (in particular, there is no reason to require the whole category lives over $\mathrm{AnSpec}({\mathbb{Q}_p}_{\sol})$), however it is unclear to the author that the (opposite) category of all Fredholm analytic rings are stable under finite limits. On the other hand, this paper motivated by a joint work in progress with Ben Heuer on a $p$-adic version of geometric Langlands correspondence, where we need to construct certain $\mathbb{G}_m^{\mathrm{an}}$-gerbe (Simpson gerbe \cite{Simpsongerbe}) over the analytic moduli stack of $G$-Higgs bundle on smooth proper rigid curves via descending along smooth charts, which a priori certain representability result of analytic stacks are needed. Since the construction of Simpson gerbe indeed have a $p$-adic nature, for our application its enough to work with Gelfand stacks.
\begin{pp-tweak}[\!\!Acknowledgements]
The author would like to thank Kęstutis Česnavičius for his guidance and support. The author is also grateful to Nick Rozenblyum, Johannes Anschütz, Matteo Montagnani, Yuanyang Jiang, Ayman Toufik, Hari Sudarsan, Hyungseop Kim, and Ruichuan Zhang for many helpful discussions.
\end{pp-tweak}

\section{Recollection on \(\mathrm{Pr}^L\)}\label{sect2}
\numberwithin{equation}{subsection}
In this section, we briefly review the theory of presentable \(\infty\)-categories and state a few basic properties that will be used in this note. For details, we refer to \cite{lurie2008highertopostheory} and \cite{SAG}.

Given a small \(\infty\)-category \(\cB\) and a regular cardinal \(\kappa\), we denote by \(\rind_{\kappa}(\cB)\) the \(\infty\)-category freely generated by \(\cB\) via \(\kappa\)-filtered colimits (equivalently, the full subcategory of the category of presheaves \(\rfun(\cB^{\mathrm{op}}, \Ani\footnote{\text{We denote $\Ani$ by the $\infty$-category of anima.}})\), consisting of \(\kappa\)-filtered colimits of representable presheaves \cite{lurie2008highertopostheory}*{Definition 5.3.5.1}). For \(\kappa = \omega\), we will simply write \(\rind(\cB)\) instead of \(\rind_{\omega}(\cB)\).

An \(\infty\)-category \(\cC\) is called \(\kappa\)-accessible for some regular cardinal \(\kappa\) if \(\cC\simeq \rind_{\kappa}(\cC_{0})\), where \(\cC_0\) is a small \(\infty\)-category. 
An \(\infty\)-category \(\cC\) is \(\kappa\)-presentable if \(\cC\) is \(\kappa\)-accessible and cocomplete (\cite{lurie2008highertopostheory}*{Definition 5.5.0.1}). Equivalently, this means that \(\cC\simeq \rind_{\kappa}(\cC_{0})\), where the category \(\cC_{0}\) has \(\kappa\)-small colimits (\cite{lurie2008highertopostheory}*{Theorem 5.5.1.1}). Furthermore, \(\cC\) is presentable if it is \(\kappa\)-presentable for some \(\kappa\). Following \cite{lurie2008highertopostheory}, we denote by \(\mathrm{Pr}^L\) the \((\infty, 1)\)-category of all presentable categories and colimit-preserving functors (i.e., left adjoint functors).

If \(\cC\) and \(\cD\) are presentable, we will denote by \(\rfun^{L}(\cC, \cD)\) the $\infty$-category of colimit-preserving functors. The category \(\mathrm{Pr}^L\) has a natural symmetric monoidal structure, for which the internal \(\underline{\mathrm{Hom}}\) from \(\cC\) to \(\cD\) is given by \(\rfun^{L}(\cC, \cD)\). The tensor product of presentable categories is usually called the \textit{Lurie tensor product} and can be defined as \(\cC\tensor\cD\simeq \rfun^{L}(\cC, \cD^{\mathrm{op}})^{\mathrm{op}}\) for \(\cC, \cD \in \mathrm{Pr}^L\) (\cite{HA}*{Proposition 4.8.1.17}). The unit object is given by \(\Ani\).

The full subcategory \(\mathrm{Pr}_{\mathrm{st}}^L\) of \(\mathrm{Pr}^L\), spanned by presentable stable \(\infty\)-categories, admits a standard symmetric monoidal structure whose tensor product is given by the Lurie tensor product, with unit object \(Sp\), the \(\infty\)-category of spectra (\cite{SAG}*{Variant D.2.3.3}). For each \(\cC \in \mathrm{CAlg}(\mathrm{Pr}_{\mathrm{st}}^L)\), we denote \(\rmod_{\cC}(\mathrm{Pr}_{\mathrm{st}}^L)\) as the symmetric monoidal \(\infty\)-category of \(\cC\)-module objects in \(\mathrm{Pr}_{\mathrm{st}}^L\). A \(\cC\)-module \(\cD\) is called \textit{dualizable} over \(\cC\) if it is a dualizable object in \(\rmod_{\cC}(\mathrm{Pr}_{\mathrm{st}}^L)\) (\cite{HA}*{Section 4.6.1}). There are many equivalent useful descriptions of dualizable categories in general, but in this note, we will only crucially use the fact that the Lurie tensor product with a dualizable category commutes with limits, since tensoring with a dualizable object always commutes with limits.

We recall a few properties that will be used later, from which one could also illustrate the concrete meaning of the Lurie tensor product.
\subsection{Lurie tensor product}
\begin{prop}[Universal property of the Lurie tensor product \cite{lurie2008highertopostheory}*{Proposition 4.8.1.17}]\label{universal_lurie}
    Let \(\cA \in \mathrm{CAlg}(\mathrm{Pr}_{\mathrm{st}}^L)\). For \(\cA\)-linear stable presentable \(\infty\)-categories \(\cC, \cD\), and \(\cE\), we have
    \[
    \rfun_{\cA}^{L,L}(\cC\times\cD, \cE) \isomto \rfun_{\cA}^L(\cC\tensor_{\cA}\cD, \cE),
    \]
    Here \(\rfun^{L,L}(\cC\times\cD, \cE)\) denotes the \(\infty\)-category of functors \(\cC\times\cD\to\cE\) that preserve colimits separately in \(\cC\) and \(\cD\). We denote by \(c\boxtimes d\) in \(\cC\tensor_{\cA}\cD\) the image of \((c,d)\in \cC\times\cD\) via the functor \(p\colon\ \cC\times\cD\rightarrow\cC\tensor_{\cA}\cD\).
\end{prop}

\begin{prop}[Objects in the Lurie tensor product]\label{obj}
    In the above setting, let \(x\in \cC\tensor_{\cA}\cD\). Then there exists a \textup{(small)} diagram \(f\colon\ I\rightarrow \cC\times\cD\) such that
    \[
    x\simeq \colim_{I } p\circ f.
    \]
    Here \(p\) is the functor \(p\colon\ \cC\times\cD\rightarrow\cC\tensor_{\cA}\cD\).
\end{prop}

\begin{proof}
    This essentially follows from \cite{dualizability}*{Lemma 0.1}. For the \(\cA\)-linear version of his proof, one only needs to observe that any \(\cA\)-module \(\cC\in \mathrm{Pr}^{L}\) is an (\(\cA\)-linear) localization of some \(\cA\)-valued presheaf category.
\end{proof}

\begin{prop}[Lurie tensor product of module categories]\label{mod}
    Let \(\cC\) be a stable presentable symmetric monoidal \(\infty\)-category, and let \(A \in \cC\) be an associative algebra object.

    \benuma
        \item For any \(\cC\)-module \(\cM\), there is an equivalence of \(\infty\)-categories\upshape:
        $$
        \rmod_{A}(\cC)\tensor \cM \simeq \rmod_{A}(\cM).
        $$
        \item For \(B\) a second associative algebra, there is an equivalence of \(\infty\)-categories\upshape:
        $$
        \rmod_{A}(\cC)\tensor_{\cC}\rmod_{B}(\cC)\simeq \rmod_{A\tensor B}(\cC).
        $$
    \eenum
\end{prop}

\begin{proof}
    See \cite{HA}*{Theorem 4.8.4.6}.
\end{proof}
For later application towards \Cref{intro main}, we will need a alternative description of Lurie tensor product of two presentable categories, in the case of one of categories is \textit{rigid}, in terms of \textit{Fourier-Mukai transform}. 
\begin{definition}
    Let $\mathrm{L}_{\cC} \colon \cC\rightarrow\cD$ be a morphism in $\mathrm{CAlg}(\mathrm{Pr}_{\mathrm{st}}^{\mathrm{L}})$, then we call $\cD$ is \textit{rigid} over $\cC$ if 
    \benuma
    \item The morphism $m_{\cD}\colon \cD\tensor_{\cC}\cD\rightarrow \cD$ admits a right adjoint in $\mathrm{Mod}_{\cD\tensor_{\cC}\cD}(\mathrm{Pr}_{\mathrm{st}}^{\mathrm{L}})$.
    \item $\mathrm{L}_{\cC}$ admits a right adjoint $\mathrm{R}_{\cD}$ in $\mathrm{Mod}_{\cC}(\mathrm{Pr}_{\mathrm{st}}^{\mathrm{L}}).$
    \eenum
\end{definition}

\begin{definition}
   Let $\cA \in \mathrm{CAlg}(\mathrm{Pr}_{\mathrm{st}}^{L})$ and $ \cB, \cC \in \mathrm{CAlg}_{\cA}(\mathrm{Pr}_{\mathrm{st}}^{L}).$ Moreover, we assume that the unit map $\mathrm{L}_\cB\colon \cA\rightarrow \cB$ is \textit{cocontinuous}, i.e., admits a right adjoint $\mathrm{R}_{\cB}$ that is also colimit preserving. Then we define the \textit{Fourier-Mukai transform} functor 
   $$\pmb{\mathrm{FM}}\colon \cB\tensor_{\cA}\cC\rightarrow \mathrm{Fun}_{\cA}^{\mathrm{L}}(\cB,\cC)$$
   as follows$\colon$ it corresponds to, by adjunction, the composition 
   $$\cB\tensor_{\cA}(\cB\tensor \cC)\underset{(\mathrm{id}_{\cB}\tensor L_{\cC})\tensor \mathrm{id}_{\cB\tensor_{\cA}\cC}}{\rightarrow}(\cB\tensor_{\cA}\cC)\tensor_{\cA}(\cB\tensor_{\cA}\cC)\underset{m_{\cB\tensor_{\cA}\cC}}{\rightarrow} \cB\tensor_{\cA}\cC \underset{\mathrm{R}_{\cB}\tensor \mathrm{id}_{\cC}}{\rightarrow} \cC.$$
\end{definition}

\begin{prop}\label{fm is conservative}
 In the above settings, if $\cB$ is over $\cA$, then we have that $\pmb{\mathrm{FM}}$ is an equivalence.
\end{prop}
\begin{proof}
Since $\cB$ is relatively rigid, we have that $\cB\simeq \cB^{\vee}$ and by dualizability of $\cB$, we have $f\colon \cB\tensor_{\cA}\cC \simeq \mathrm{Fun}_{\cA}^{L}(\cB^{\vee},\cC)\simeq \mathrm{Fun}_{\cA}^{L}(\cB,\cC),$ thus it suffices to show that this natural equivalence is induced by $\pmb{\mathrm{FM}}.$ Consider the following commutative diagram
$$\begin{tikzcd}
\cB\tensor_{\cA}\cB\tensor_{\cA}\cC \arrow[r, "\mathrm{id}_{\cB}\tensor\mathrm{L}_{\cC}\tensor \mathrm{id}"'] \arrow[d, "\mathrm{id}"'] \arrow[ddd, "\mathrm{counit}\tensor \mathrm{id}_{\cC}"', bend right=67] & \cB\tensor_{\cA}\cC\tensor_{\cA}\cB\tensor_{\cA}\cC \arrow[d, "\mathrm{id}\tensor m_{\cC}"] \arrow[dd, "m_{\cB\tensor_{\cA}\cC}", bend left=74] \\
\cB\tensor_{\cA}\cB\tensor_{\cA}\cC \arrow[r, "\mathrm{id}"] \arrow[d, "m_{\cB}\tensor \mathrm{id}_{\cC}"']                                                                                                     & \cB\tensor_{\cA}\cB\tensor_{\cA}\cC \arrow[d, "m_{\cB}\tensor \mathrm{id}_{\cC}"]                                                                     \\
\cB\tensor_{\cA}\cC \arrow[d, "\mathrm{R}_{\cB}\tensor \mathrm{id}_{\cC}"'] \arrow[r, "\mathrm{id}"]                                                                                                                  & \cB\tensor_{\cA}\cC \arrow[d, "\mathrm{R}_{\cB}\tensor \mathrm{id}_{\cC}"]                                                                            \\
\cC \arrow[r, "\mathrm{id}"]                                                                                                                                                                                    & \cC                                                                                                                                            
\end{tikzcd}$$
 Here we used that there is a natural equivalence between the counit functor $\mathrm{counit}\colon \cB\tensor_{\cA} \cB^{\vee}\rightarrow \cA$ and the composition $\cB\tensor_{\cA}\cB\underset{m_{\cB}}{\rightarrow} \cB\underset{\mathrm{R}_{\cB}}{\rightarrow} \cA.$ Since by adjunction the functor $f$ corresponds to $\mathrm{counit}\tensor \mathrm{id}_{\cC}\in \mathrm{Fun}(\cB\tensor_{\cA}\cB\tensor_{\cA}\cC,\cC)$, the above commutative diagram shows that there is a natural equivalence from $f$ to $\pmb{\mathrm{FM}}.$ This finishes the proof. 
\end{proof}
    It is unclear how to describe the \(\Hom\) space inside the Lurie tensor product of categories. Still, there is an explicit description for compactly generated categories, we will see later (\Cref{hom}).

\csub{Small stable categories}
For the purpose of our note we also need the tenor product of small stable idempotent complete $\infty$-categories. Indeed, as explained in \Cref{remark1}, the statement of \Cref{thm1} is actually about taking tensor product of small categories.

Let $\mathrm{Cat}_{\infty}^{\mathrm{st},\mathrm{perf}}$ be the full $\infty$-subcategory of the stable $\infty$-categories (with morphisms exact functors) consiting of those $\infty$-categories that are idempotent complete. Recall that an $\infty$-category $\cC$ is idempotent complete if the essential image of the Yoneda embedding $\cC\rightarrow \cP(\cC)$ is closed under retracts. 

\begin{definition}
 For $\cC_{1}, \cC_{2}\in \mathrm{Cat}_{\infty}^{\mathrm{st},\mathrm{perf}}$, we define their tensor product by $\cC_{1}\tensor{\cC_2}=(\rind(\cC_{1})\tensor \rind(\cC_{2}))^c$, where $(-)^c$ denotes the subcategory consists of compact objects. This is well defined since, $\rind(\cC_1)\tensor \rind(\cC_2)$ is idempotent complete and retract of compact object is also compact. Similarly, for $\cA\in  \mathrm{CAlg}(\mathrm{Cat}_{\infty}^{\mathrm{st},\mathrm{perf}})$, we denote by ${\mathrm{Cat}_{\infty,\cA}^{\mathrm{st},\mathrm{perf}}}$ the $\infty$-category of $\cA$-linear idempotent complete stable $\infty$-categories.

\end{definition}

This definition satisfies the universal properties one should expect for tensor product\upshape: 

\begin{prop}\label{Karoubian}
The $\infty$-category $\mathrm{Cat}_{\infty}^{\mathrm{st},\mathrm{perf}}$ carries a symmetric monoidal structure characterized by the property
that for $\cC_1,\cC_2,\cD\in \mathrm{Cat}_{\infty}^{\mathrm{st},\mathrm{perf}}$, the $\infty$-category of exact functors
 $\rfun_{ex}(\cC_1\otimes \cC_2, \cD)$ is equivalent to the full $\infty$-subcategory 
of all functors $\cC_1\times \cC_2 \to \cD$ that preserve finite colimits in 
$\cC_1$ and $\cC_2$ separately. Furthermore, passing to the corresponding 
stable presentable $\infty$-categories of $\rind$-objects 
is a symmetric monoidal functor.
\end{prop}

\begin{proof}
    \cite{benzvi2010integraltransformsdrinfeldcenters}*{proposition 4.4}.
\end{proof}
As a corollary this tells us the following:
\bcor\label{dual}
In the setting above, under this symmetric monoidal structure on $\mathrm{Cat}_{\infty}^{\mathrm{st},\mathrm{perf}}$, we have that any $\cC\in \mathrm{Cat}_{\infty}^{\mathrm{st},\mathrm{perf}} $ is dualizable, and $\cC^{\vee}\simeq \cC^{\mathrm{op}}$. 
\ecor
From now on all the $\infty$-categories appears are symmetric monoidal. We give some useful criterion to check whether a symmetric monoidal functor  $\cC_1\tensor\cC_2\rightarrow \cD$ is an equivalence of categories.
\bprop[Computation of Hom space in Lurie tensor product]\label{hom}
    Assume that $\cC_1,\cC_2,\cA\in \mathrm{Cat}_{\infty}^{\mathrm{st},\mathrm{perf}}$, where $\cC_1$ and $\cC_2$ are $\cA$-linear and $\cA$ is \textit{self dual} \up{$\cA^{\mathrm{op}}\simeq \cA$}. For $c\boxtimes d, c^{\prime}\boxtimes d^{\prime} \in \cC_1\tensor_{\cA}\cC_2$, we have
$$\mathrm{Hom}(c\boxtimes d, c^{\prime}\boxtimes d^{\prime})=\mathrm{Hom}(c,c^{\prime})\tensor \mathrm{Hom}(d, d^{\prime})\in \mathrm{Ind}(\cA).$$
\eprop
\begin{proof}
We denote $\bC_{1}\coloneq \rind(\cC_1), \bC_{2}\coloneq \rind(\cC_{2})$, $\bA\coloneq\rind(\cA)$ and $\bD\coloneq \rind(\cD)$.  
Note that we have a functor as follow:
\begin{align}\label{long} 
\bC_{1}\tensor_{\bA} \bC_{2}\rightarrow \rfun_{\bA}^{L}(\bC_{1}^{\vee}\tensor \bC_{2}^{\vee}, \bA)\rightarrow \rfun_{\cA-linear}((\cC_{1})^{\op}\times (\cC_2)^{\mathrm{op}}, \bA) 
\end{align}
Here, the first arrow defined via counit map of $(\bC_{1}\tensor_{\bA} \bC_{2}) \tensor_{\mathbb{A}} (\bC_{1}\tensor_{\bA} \bC_{2})^{\vee}$ and the second defined by \Cref{universal_lurie}.
The corresponding functor
$$\bC_1\times_{\bA}\bC_2 \rightarrow \rfun_{\cA}((\cC_{1})^{\op}\times (\cC_2)^{\mathrm{op}}, \bA) $$
is by definition given by 
\begin{align}\label{short}
(c_1,c_2)\mapsto ((c_1^{\prime},c_2^{\prime})\mapsto \mathrm{Hom}_{\cC_1}(c_1^{\prime},c_1)\tensor \mathrm{Hom}_{\cC_2}(c_2^{\prime}, c_2)).
\end{align}

By construction, the essential image of  \Cref{long} is contained in the full subcategory 
$$\rfun_{\mathrm{bi}\text{-}\mathrm{ex}, \cA}(\cC_{1}^{\mathrm{op}}\times \cC_2^{\mathrm{op}},\bA)\subset \rfun_{\cA}(\cC_{1}^{\mathrm{op}}\times \cC_{2}^{\mathrm{op}}, \bA)$$
that consists of functors that are exact in each variable. We denote the resulting functor 
\begin{align}\label{equiv}
    h_{1,2}\colon\bC_1\tensor \bC_2\rightarrow \rfun^{L,L}(\cC_{1}^{\mathrm{op}}\times \cC_2^{\mathrm{op}},\bA).
\end{align}
We claim that this functor is an equivalence. Indeed, by \Cref{dual} one observe that 
$$\bC_1\tensor_{\bA}\bC_2\simeq \rfun_{\mathrm{ex}, \cA}(\cC_{1}^{\mathrm{op}},\bC_2)\simeq \rfun_{\mathrm{ex}, \cA-linear}(\cC_{1}^{\mathrm{op}}, \rfun_{\mathrm{ex}}(\cC_{2}^{\mathrm{op}}, \bA))\simeq \rfun_{\mathrm{bi}\text{-}\mathrm{ex},\cA-linear}(\cC_{1}^{\mathrm{op}}\times \cC_{2}^{\mathrm{op}}, \bA).$$
Now apply this equivalent functor $h_{1,2}$ to  (\ref{short}) finish the proof.
\end{proof}
\bprop[Criterion for fully faithful functor]\label{ffaith}
For  $ \cA\in \mathrm{CAlg}(\mathrm{Cat}_{\infty}^{\mathrm{st},\mathrm{perf}})$, with $\cA$-linear categories $\cC_1,\cC_2,\cD\in \mathrm{Cat}_{\infty}^{\mathrm{st},\mathrm{perf}}$, a symmetric monoidal functor $h\colon\ \cC_1\tensor_{\cA}\cC_2\rightarrow \cD$ is fully faithful if and only if for any $(c_1,c_2), (c_{1}^{\prime}, c_{2}^{\prime})\in \cC_1\times \cC_2$, $$\mathrm{Hom}_{\cD}(h(c_1\boxtimes c_2), h(c_{1}^{\prime}\boxtimes c_{2}^{\prime}))\simeq \mathrm{Hom}_{\cC_1}(c_1, c_{1}^{\prime})\tensor \mathrm{Hom}_{\cC_2}(c_2, c_{2}^{\prime}).$$
\eprop
\begin{proof}
    By  \Cref{hom}, the ``only if'' part is obvious. We prove the ``if'' direction. 
    
    We denote ${\bC}_{1}\coloneq \rind(\cC_1), \bC_{2}\coloneq \rind(\cC_{2})$ and $\bD\coloneq \rind(\cD)$. First by assumption, we know that this functor is fully faithful after restricts to the full sub-category $\cC_{fin}$ generated by exterior tensor products under finite colimit. Then we notice that by \Cref{obj}, any object $c\in\bC_1\tensor \bC_2$ can be written as a colimit of the form $\mathop{\mathrm{colim}}\limits_{I} (c_{1,i}\boxtimes c_{2,i})$. If $c\in \cC_1\tensor\cC_2$, then by compactness of $c$, one can assume that $c$ is a retract $c\stackrel{i}\hookrightarrow\mathop{\mathrm{colim}}\limits_{J} (c_{1,j}\boxtimes c_{2,j})\stackrel{f}\to c$ of such colimit,  where the index $J\subset I$ is finite. Then we can further write it as the sequential colimit $\mathop{\mathrm{colim}}\limits_{i\circ f} \mathbf{c}_{fin}$ where $\mathbf{c}_{fin} \in \cC_{fin}$, and also as a sequential limit $\mathop{\mathrm{lim}}\limits_{\mathrm{Id}-f\circ i}\mathbf{c}_{fin}$. Since retract are carried to retract under $h$, one also have $h(c)=\mathop{\mathrm{colim}}\limits_{\mathrm{Id}-h(i\circ f)} h(\mathbf{c}_{fin})=\mathop{\lim}\limits_{\mathrm{Id}-h(f\circ i)}h(\mathbf{c}_{fin})$.
    
    Now the full faithfulness follows from for any $c,c'\in \cC$,
$$\begin{aligned}
\mathrm{Hom}(c, c')&=\mathrm{Hom}(\mathop{\mathrm{colim}}\limits_{\mathrm{Id}-f\circ i} \mathbf{c}_{fin}, \mathop{\mathrm{lim}}\limits_{\mathrm{Id}-f'\circ i'}\mathbf{c'}_{fin})\\
    &=\mathop{\mathrm{lim}}\limits_{\mathrm{Id}-f\circ i}\mathop{\mathrm{lim}}\limits_{\mathrm{Id}-f'\circ i'}\mathrm{Hom}(\mathbf{c}_{fin}, \mathbf{c'}_{fin})\\
    &=\mathop{\mathrm{lim}}\limits_{\mathrm{Id}-h(f\circ i)}\mathop{\mathrm{lim}}\limits_{\mathrm{Id}-h(f'\circ i')}\mathrm{Hom}(h(\mathbf{c}_{fin}), h(\mathbf{c'}_{fin}))\\
    &=\mathrm{Hom}(\mathop{\mathrm{colim}}\limits_{\mathrm{Id}-h(f\circ i)} h(\mathbf{c}_{fin}), \mathop{\mathrm{lim}}\limits_{\mathrm{Id}-h(f'\circ i')}h(\mathbf{c'}_{fin}))\\
    &=\mathrm{Hom}(h(c), h(c')).
\end{aligned}$$
   
\end{proof}
\bprop[Criterion for equivalence]\label{esurj}
In the setting above, if $h$ is fully faithful, then $h$ is an equivalence if and only if the functor
$$g\colon \cD\rightarrow \mathrm{Fun}_{\mathrm{bi}\text{-}\mathrm{ex}}(\cC_{1}^{\mathrm{op}}\times \cC_{2}^{\mathrm{op}}, \rind(\cA))$$
Given by $d\mapsto ((c_1,c_2)\mapsto \mathrm{Hom}_{\cD}(h(c_1\tensor c_2), d)) $ is conservative.
\eprop
\begin{proof}
    Passing to Ind category, by \Cref{obj} the ``only if'' direction is straitforward. We now deal with ``if'' part.
    
  We notice that under the equivalence $h_{1,2}$ from \Cref{equiv}, g is the right adjoint of h. Indeed, by Yoneda $\rfun_{\mathrm{bi}\text{-}\mathrm{ex}}(\cC_{1}^{\mathrm{op}}\times \cC_{2}^{\mathrm{op}}, \rind(\cA))$ is compactly generated by Yoneda sheaf $\eta((c_1\times c_2))$ which corresponce to $c_1\tensor c_2$ under \Cref{equiv}. Thus one only needs to check the adjoint pair on $(c_1,c_2)$. That is:
  $$\mathrm{Hom}_{\rfun_{\mathrm{bi}\text{-}\mathrm{ex}}(\cC_{1}^{\mathrm{op}}\times \cC_{2}^{\mathrm{op}}, Sp)}(\eta(c_1,c_2),g(d))=\mathrm{Hom}_{\cD}(h(\eta(c_1\times c_2)), d).$$
  But again by Yoneda, $\mathrm{Hom}_{\rfun_{\mathrm{bi}\text{-}\mathrm{ex}}(\cC_{1}^{\mathrm{op}}\times \cC_{2}^{\mathrm{op}}, Sp)}(\eta(c_1,c_2),F)=F(c_1\times c_2)$, in particular, apply to $F=g(d), d\in \cD$, this gives $\mathrm{Hom}_{\cD}(h(c_1\tensor c_2), d)$, which is exactly $\mathrm{Hom}(h(\eta(c_1\times c_2)), d)$.

  Now, the statement follows by the following formal argument:

  \blem
   Assume $\cE,\cF \in \mathrm{Cat}_{\infty}^{\mathrm{st},\mathrm{perf}}$, have exact functor $L\colon \cE\rightarrow \cF$ which is a fully faithful embbeding and admits right ajoint $R\colon \cF\rightarrow \cE$. Then if $R$ being conservative we have $L$ is essential surjective.
  \elem
  \begin{proof}[Proof of the lemma]
      since $L$ fullyfaithful, $R\circ L\simeq \mathrm{Id} $, now we claim for any $f\in F$, $LR(f)\simeq f$. 
      Indeed, consider $R(\mathrm{Fib}(LR(f)\rightarrow f))\simeq \mathrm{Fib}(RLR(f), Rf)\simeq \mathrm{Fib}(Rf, Rf)\simeq 0$. 
      But by conservativity of $R$ this means $\mathrm{Fib}(LR(f)\rightarrow f)\simeq 0$, thus $LR(f)\simeq f$. This proves essential surjectivity.
  \end{proof}
\end{proof}

\section{Recollection on analytic geometry}

It has long been a challenge for rigid analytic geometry to incorporate a reasonable notion of stacks or derived enhancements together with their category of quasi-coherent sheaves, mainly due to the topological nature of basic affine objects in the rigid setting descent theory is harder to establish. Especially, a construction of Gabber \cite{gabbercount} prevent the naive notion of quasi-coherent sheaves in rigid analytic geometry behave nicely. In recent years, however, the work of Clausen and Scholze on condensed mathematics has provided a new perspective on this issue. In particular, the foundational work of \cite{camargo2024analyticrhamstackrigid} and \cite{anschütz2025analyticrhamstacksfarguesfontaine} has established a profound framework for the theory of analytic stacks over non-archimedean base.
As mentioned in the Introduction, the very formulation of our main result (\Cref{intro main}) is based on this notion. Therefore, we devote this chapter to recalling the relevant aspects of the theory of analytic stacks.

Most of the material presented here is well known to experts, but for the sake of completeness, we include detailed proofs.

\csub{Non-archimedean analytic stacks}\label{non-arch ana stack}

Let $\mathrm{Prof}^{\mathrm{light}}$ be the category of the metrizable profnite sets (i.e., can be presented as countable limit of finite sets), and we endow it with the Grothendieck topology whose covers are finitely many jointly surjective maps between the profinite sets.  Typical examples of metrizable profinite sets includes\upshape:
\benuma
\item The one point compactification $\mathbb{N}\cup\{\infty\}$ of $\mathbb{N}$ (endow $\mathbb{N}$ with discrete topology) is a profinite set. Indeed, we can write it as $\mathbb{N}\cup\{\infty\}=\lim_n \{1,2,...n,\infty\}.$
\item The Cantor set $\{0,1\}^{\mathbb{N}}.$
\eenum 
\hiddensubsubsection{Basic condensed mathematics}
We recall from \cite{mann2022padic6functorformalismrigidanalytic}*{Definition 2.1.1} that, for an $\infty$-category $\cC$, a \textit{light condensed object in $\cC$} is a sheaf $$T: \mathrm{Prof}^{\mathrm{light}, \mathrm{op}}\rightarrow \cC$$ 
The $\infty$-category of light condensed objects in $\cC$ are denoted by $\mathrm{Cond(\cC))}$. Frequently used condensed objects includes the following
\benuma
\item $\cC=\mathrm{Set}$, the category of sets.  \textit{condensed set}, we denote by $\mathrm{CondSet}.$
\item$\cC=\Ani$, the $\infty$-category of anima. The corresponding category of condensed objects is called \textit{condensed anima}, we denote by $\mathrm{CondAni}.$
\item $\cC=\mathrm{Sp}$, the $\infty$-category of spectra. The corresponding category of condensed objects is called \textit{condensed spectra}, we denote by $\mathrm{CondSp}.$
\item $\cC=D(\mathbb{Z})$, the derived ($\infty$-)category of abelian groups. The corresponding category of condensed objects is called \textit{\up{derived} condensed abelian group}, we denote by $D(\mathrm{CondAb}).$
\item $\cC=\mathrm{CAlg}(D(\mathbb{Z}))$, the $\infty$-category of commutative algebra objects in $D(\mathbb{Z})$. The corresponding category of condensed objects is called \textit{condensed ring}, we denote by $\mathrm{CondAlg}.$
\eenum
The infinite loop space functor $\Omega^{\infty}\colon D(\mathbb{Z})\rightarrow \Ani$ induces functor between condensed objects
\begin{equation}\label{adjunction between ani and ab}
   \begin{tikzcd}
D(\mathrm{CondAb}) \arrow[r, "\mathrm{Cond}(\Omega^{\infty})", shift left=2] & \mathrm{CondAni.} \arrow[l, "{\mathbb{Z}[-]}", shift left=2]
\end{tikzcd} 
\end{equation}

It admits a left adjoint $\mathbb{Z[-]}$, which is the sheafification of the presheaf sending a profinite set $T$ to the condensed abelian group $\mathbb{Z}[T]: S\mapsto \mathbb{Z}[\mathrm{Cont}(S,T)].$

The category of condensed sets includes a large class topological space. More precisely, for any locally compact Hausdorff space $T\in \mathrm{LCHaus},$ we can associate a condensed set $\underline{T}\colon S\mapsto \mathrm{Cont}(S,T).$  In particular, by choosing $\mathbb{Z}$ with discrete topology, this gives a condensed ring $\underline{\mathbb{Z}}.$
\hiddensubsubsection{Solid modules}

There is a class of (derived) condensed abelian group that are of ``non-archimedean'' nature called \textit{solid abelian groups}. It is a very important feature in non-archimedean analysis that null sequences are ``integrable'', indeed, for a $p$-complete $\mathbb{Z}_p$ module $M$, if there is a sequence $\{m_i\}_{i\in\mathbb{N}}$ such that the $m_i$ converge in the $p$-adic topology (i.e., null sequence), then $\Sigma_{i\in\mathbb{N}} m_i$ exist in $M.$  We denote the ``space'' of null sequence as above by $\mathrm{Null}(M)$, then the feature of integrability of null-sequence can be interpret as the map
$$\mathrm{Null}(M)\underset{\{m_i\}_{i\in\mathbb{N}}\mapsto \{m_i-m_{i+1}\}_{i\in \mathbb{N}}}{\longrightarrow}\mathrm{Null}(M)$$
being invertible. 
As an abstract incarnation of this feature, denote by $\mathrm{Shift}$ the endomorphism of $\mathbb{N}\cup  \{\infty\}$ induced by mapping $n$ to $n+1$, we call a derived condensed abelian group $M$ is \textit{solid} if 
 $$\mathrm{Map}(\mathbb{Z[\mathbb{N}\cup\{\infty\}]}, M)\underset{1-\mathrm{shift}}{\longrightarrow} \mathrm{Map}(\mathbb{Z}[\mathbb{N}\cup\{\infty\}], M)$$
is an isomorphism. 

We denote by $D_{\sol}(\mathbb{Z})$ the $\infty$-category of solid abelian groups, which by \cite{condensed}*{Proposition 7.5} is a full subcategory of $D(\mathrm{CondAb})$ that is stable under all limits and colimits (in particular, the forgetful functor admits a left adjoint $(-)^{\sol}$). The $\infty$-category $D_{\sol}(\mathbb{Z})$ admits a natural symmetric monoidal structure $-\tensor_{\mathbb{Z}}^{\sol}-$ that is computed as $M\tensor_{\mathbb{Z}}^{\sol}N= (M\tensor_{\underline{\mathbb{Z}}} N)^{\sol}.$ The first non-trivial computation about solid abelian groups is the following 
$$\mathbb{Z}[\![T]\!]\tensor_{\mathbb{Z}}^{\sol}\mathbb{Z}[\![T]\!]\simeq \mathbb{Z}[\![T_1,T_2]\!]$$
where the formal power series rings endowed with product topology. 

This suggests that $D_{\sol}(\mathbb{Z})$ provides a natural environment for performing homological algebra in non-archimedean analysis. Moreover, one can attempt to do ``geometry'' over this category, which leads to the concept of an \textit{analytic ring}.

\hiddensubsubsection{Analytic rings}
We recall from \cite{mann2022padic6functorformalismrigidanalytic}*{Definition 2.3.10} that an \textit{analytic ring} $\cA$ is a pair $(\underline{\cA}, D_{\sol}(\cA))$ consisting of a condensed ring $\underline{\cA}$ and the category of $\cA$-\textit{complete} modules $D_{\sol}(\cA)$, which is a full subcategory of $\mathrm{Mod}_{\underline{\cA}}(D_{\mathrm{Cond}}(\mathbb{Z}))$ satisfying  the following conditions
    \benuma
    \item It is a full subcategory of $D(\underline{\cA})$ that is stable under all limit and colimits.
    \item $D_{\sol}(\cA)$ admits a unique symmetric monoidal structure such that the left adjoint (which exist by (1)) $(-)\tensor_{\underline{\cA}}\cA: D(\underline{\cA})\rightarrow D_{\sol}(\cA)$ is symmetric monoidal and preserves connective objects.
    \item  $\underline{\cA}\in D_{\sol}(\cA)$
    \eenum
    We denote by $\AnRing$ the $\infty$-category of analytic rings, and by $\AnRing_{\cA}$ the slice category of $\AnRing$ over a analytic ring $\cA$.
\begin{eg}\label{example of analytic ring} We list some of the typical examples of analytic rings.
\benuma
\item   The pair $\mathbb{Z}_{\sol}\coloneq(\underline{\mathbb{Z}}, D_{\sol}(\mathbb{Z}))$ gives an analytic ring.
\item   The Laurent series $\mathbb{Z}(\!(T^{-1})\!)$ is an idempotent algebra in $\mathrm{Mod}_{\mathbb{Z}[T]}(D_{\sol}(\mathbb{Z}))$, and the pair $\mathbb{Z}[T]_{\sol}\coloneq(\underline{\mathbb{Z}[T]}, \mathrm{Mod}_{\mathbb{Z}[T]}(D_{\sol}(\mathbb{Z}))/\mathrm{Mod}_{\mathbb{Z}(\!(T^{-1})\!)}(D_{\sol}(\mathbb{Z})))$ defines an analytic ring.
\item   For any set $S$, $(\mathbb{Z}[\mathbb{N}[S]],\mathbb{Z})_{\sol}\coloneq(\underline{\mathbb{Z}}[\mathbb{N}[S]], \mathrm{Mod}_{\underline{\mathbb{Z}}[\mathbb{N}[S]]}D_{\sol}(\mathbb{Z}))$ defines an analytic ring. In fact, we have a much more general construction\upshape. Let $\cA$ be a analytic ring, similar to (\ref{adjunction between ani and ab}), we have the following adjunction 
$$\begin{tikzcd}
\AnRing_{\cA} \arrow[r, "\textit{forget}", shift left=2] & \mathrm{CondAni} \arrow[l, "{\underline{\cA}[\mathbb{N}[-]]}", shift left=2]
\end{tikzcd}.$$
In particular, for any condensed anima $M$, $(\underline{\cA}[\mathbb{N}[S]], \mathrm{Mod}_{\underline{\cA}[\mathbb{N}[S]]}D_{\sol}(\cA))$ defines an analytic ring.
\item   For any finite set $S$, $(\mathbb{Z}[\mathbb{N}[S]])_{\sol}\coloneq\underset{s\in S}{\tensor}\mathbb{Z}[T_s]_{\sol}$ defines an analytic ring.
\item  For any set $S$, $\mathbb{Z}[\mathbb{N}[S]]_{\sol}\coloneq \underset{S'\subset S, |S'|<\infty}{\colim} \mathbb{Z}[\mathbb{N}[S']]_{\sol}$ defines an analytic ring.
\item  Let $A\in \mathrm{CAlg}(D_{\sol}^{\leq 0}(\mathbb{Z}))$ and $S$ a subset of $\pi_{0}A(*)$, $(A,S)_{\sol}\coloneq (A,\mathbb{Z})_{\sol}\tensor_{(\mathbb{Z}[\mathbb{N}[S]],\mathbb{Z})_{\sol}}\mathbb{Z}[\mathbb{N}[S]]_{\sol}$ defines an analytic ring.
\item\label{tate base}   Given a complete Huber ring $(A,A^+)$, one can associate an analytic ring $(A,A^+)_{\sol}$ by taking the solid algebra $\underline{A}$ and the set $A^+\subset \underline{A}(*)$ as above (see also \cite{andreychev2021pseudocoherentperfectcomplexesvector}*{Proposition 3.23}). In particular, associated to the complete Huber ring $(\mathbb{Z}(\!(\pi)\!), \mathbb{Z}[\![\pi]\!])$ we have the analytic ring $R_{\sol}=(R,R^{+})_{\sol}\coloneq(\mathbb{Z}(\!(\pi)\!), \mathbb{Z}[\![\pi]\!])_{\sol}$ which is equivalent to the pair $(\underline{\mathbb{Z}(\!(\pi)\!)}, \mathrm{Mod}_{\mathbb{Z}(\!(\pi)\!)}D_{\sol}(\mathbb{Z})).$
\item\label{tate generalized algebra} for any profinite set $S$, let $R^{+}\langle\mathbb{N[S]}\rangle\coloneq  \lim_{n} (R^{+}/\pi^{n})_{\sol}[\bN[{S}]]$ and $R\langle\mathbb{N}[S]\rangle\coloneq R^{+}\langle\mathbb{N[S]}\rangle[1/\pi]$. Then $(R\langle\mathbb{N}[S]\rangle), \pi_0R^{+}\langle\mathbb{N[S]}\rangle(*))_{\sol}$ defines an analytic rings. This is a generalized version of Tate algebras.
    \eenum
\end{eg}

There is a natural Grothendieck topology on the $\infty$-category of analytic rings. Recall from \cite{camargo2024analyticrhamstackrigid}*{Definition 3.2.2} that, a morphism $f:\cA\rightarrow \cB$ of analytic rings is called\upshape:
\benuma
        \item \textit{Open} immersion if $f^*\colon D_{\sol}(\cA)\rightarrow D_{\sol}(\cB)$ is an \textit{localization} \up{that is, admits a fully faithful right adjoint}.
        \item \textit{Proper} if the analytic structure on $\cB$ is induced from $\cA$.
        \item \textit{$!$-able} if $f=j\circ p$ can be write as a composition of \textit{proper} map $p$ with \textit{open} immersion $j$.
        \eenum
This allow us to define the following $!$-\textit{topology}.

    \begin{definition}[\cite{camargo2024analyticrhamstackrigid}*{Definition 3.2.7}]
    Let $\{f_i\colon \cA\rightarrow 
        \cB_i\}$ be a family of maps in $\Aff_{\bZ_{\sol}}$.
        \benuma
      \item The maps $f_i$ statisfied \textit{universal *-descent} if for all base changes  $\{f_{i}^{'}\colon\cA'\rightarrow\cB_{i}'\}$ along any morphism of analytic rings $\cA\rightarrow\cA'$, the functor $D_{\sol}(\_)$ satisfies descent along $f_{i}'^*$, that is, let $\cB'=\Pi_{i} \cB_{i}'$ then we have equivalence of categories:  $$D_{\sol}(\cA')\simeq  \underset{[n]\in \Delta^{op}}{\lim} D_{\sol}(\cB'^{\tensor n/\cA'}).$$
      Where the limit is taking along pullback functors.
      \item Assume the maps $f_i$ are \textit{$!$-able}, then we call the maps $f_i$ satisfy \textit{universal $!$-descent} if for all base changes $\{f_{i}^{'}\colon\cA'\rightarrow\cB_{i}'\}$ along any morphism of analytic rings $\cA\rightarrow\cA'$, the functor $D_{\sol}(\_)$ satisfies descent along $f_{i}'^{!}$ \up{that is, all transition maps are $f_{i}^{\prime\ !}$}. That is, let $\cB'=\Pi_{i} \cB_{i}'$ then we have equivalence of categories:
      $$D_{\sol}(\cA')\simeq \underset{[n]\in \Delta^{op}}{\colim} D_{\sol}(\cB'^{\tensor n/\cA'}) $$
      Where the colimit is taking along $!$-functors.
      \item A $!$-cover is a family $\{f_i\colon \cA\rightarrow 
        \cB_i \}$ of \textit{$!$-able} maps such that they satify \textit{universal *-descent} and \textit{universal $!$-descent}.
      \eenum
    \end{definition}
\hiddensubsubsection{Solid stacks}

Following \cite{camargo2024analyticrhamstackrigid}, in order to define a reasonable class of analytic stacks capture the ``solid'' (resp., rigid analytic) counter part of analytic theory, we restrict our test category of affine objects to the \textit{solid affinoid algebras} (resp., \textit{bounded affinoid algebras}).

Let $\cA\in \AnRing_{\mathbb{Z}_{\sol}}$ be an analytic ring over $\mathbb{Z}_{\sol}$. Recall from \cite{camargo2024analyticrhamstackrigid}*{Definition 2.6.1}, that the subring of $solid$ element is the discrete animated ring given by the mapping space $$\cA^{+}= \rmap_{\AnRing_{\mathbb{Z}_{\sol}}}(\mathbb{Z}[T]_{\sol}, \cA).$$
We say  that $\cA$ is \textit{solid affinoid ring} if the natural map $(\underline{\cA}, \pi_0(\cA^{+}))\rightarrow \cA$ is an equivalence of analytic rings. Given $\cA$ a solid affinoid ring we let $\AffRing_{\cA}$ denote the slice category of solid affinoid $\cA$-algebras and $\Aff_{\cA}\coloneq \AffRing_{\cA}^{\mathrm{op}}$ the \textit{solid affinoid spaces} over $\cA$.

  We define a $\kappa\text{-}solid\ !\text{-} stack$ to be an object of the $\infty$-category $\rsh_{!}(\Aff_{\mathbb{Z}_{\sol},\kappa})$ the sheaf category on anima over $\kappa$-small solid affinoid spaces. We will usually omit $\kappa$ in the notation and write $\rsh_{!}(\Aff_{\mathbb{Z}_{\sol}})$ instead, and call it objects \textit{solid stacks}.

\hiddensubsubsection{Tate stacks}
In rigid analytic geometry, basic affine objects are complete analytic Huber pairs. Typical features of such Huber pairs $(A,A^+)$ are
\benuma
\item  There exist a \textit{pseudo-uniformizer} $\pi\in A^{*}$, such that $\pi$ is topological nilpotent.
\item   Any $f\in A$, there exist $N\gg0$ such that $\{(\pi^Nf)^n\}_{n\in \mathbb{N}}$ is summable (\textit{boundedness}).
\eenum

In the language of analytic rings, an analytic ring $\cA$ satisfy the condition (1) means that $\cA\in \AnRing_{R_{\sol}}$ (\Cref{example of analytic ring} \ref{tate base}). We now give the abstract incarnation of the condition (2).
Let $A\in \AniAlg_{\mathbb{Z}_{\sol}}$ be an animated algebra over $R_{\sol}$. Recall from \cite{camargo2024analyticrhamstackrigid}*{Definition 2.4.5} the the condensed subring of \textit{pwoer bounded elements} is the condensed animated ring with values at $S\in \extdis$ given by 
    \[A^{0}(S)= \rmap_{\AniAlg_{R_{\sol}}}(R_{\sol}\langle\mathbb{N}[S]\rangle, A).\]
    \item 
  The condensed $R$-subring of $bounded\ elements$ is defined as $A^{b}=A^{0}[\frac{1}{\pi}].$ An animated $R_{\sol}$-algebra $A$ is called $bounded$ if the map $A^{b}\rightarrow A$ is an equivalence, unwinding definition, this exactly means condition (2). We let $\AniAlg^{b}_{R_{\sol}}$ be the full subcategory of $\AniAlg$ consisting of bounded animated $R_{\sol}$-algebras.

Typical examples of bounded animated $R_{\sol}$-algebras includes $(A,A^+)_{\sol}$ for any complete analytic Huber pairs. It also contains another important class of condensed $R$-algebras, for example, the \textit{overconvergent function} $R\{T\}^{\dagger}\coloneq \colim_{n\in \mathbb{N}}R\langle \frac{T}{\pi^n}\rangle.$ In fact, we have much more general construction of such condensed rings: Let $S$ be profinite set, then the \textit{free overconvergent function} over $S$ is defined to be $R_{\sol}\{\bN[S]\}^{\dagger}= \underset{n\rightarrow +\infty}{\mathrm{colim}}R_{\sol}\langle \bN[\frac{S}{\pi^n}]\rangle$. For an analytic ring $\cA$ over $R_{\sol}$, the \textit{$\dagger$-nil-radical ideal} is the $\underline{A}^b$-ideal whose values on $T\in \mathrm{Prof}^{\mathrm{light}}$ are 
$$\mathrm{Nil}^{\dagger}(A)(T)=\rmap_{\AniAlg_{R_{\sol}}}(R_{\sol}\{\bN[T]\}^{\dagger}, A).$$

\brem
In the above definition, we implicitly used the fact of $R_{\sol}\langle S\rangle= R_{\sol}^{+}\langle S\rangle [1/\pi] $ and $R_{\sol}\{\bN[S]\}^{\dagger}= \underset{n\rightarrow +\infty}{\mathrm{colim}}R_{\sol}\langle \bN[\frac{S}{\pi^n}]\rangle$ are idempotent algebras over $R_{\sol}[\bN[S]]$ (see \cite{camargo2024analyticrhamstackrigid}*{Lemma 2.4.7}).
\erem

Finally, we arrive at the definition of \textit{Tate stack}, which serve as an analytic stack incarnation of rigid analytic geometry.

  \begin{definition}[\cite{camargo2024analyticrhamstackrigid}*{Definition 2.6.10}]\label{tate stack}\hfill
\benuma
    
    \item A solid affinoid $R_{\sol}$ algebra is $bounded$ if its underlying condensed ring is bounded. We let $\AffRing^b_{R_{\sol}}\subset \AffRing_{R_{\sol}}$ be the full subcategory of bounded affinoid $R_{\sol}$-algebras. For $\cA\in \AffRing^{b}_{R_{\sol}}$ we let $\AffRing^b_{\cA}$ be the slice category of bounded affinoid $\cA$-algebras.
    We let $\Aff_{\mathbb{Z}_{\sol}}\coloneq\AffRing^{\mathrm{op}}_{\mathbb{Z}_{\sol}}$ be the $\infty$-category of \textit{solid affinoid spaces}. For a ring $\cA\in \AffRing_{\mathbb{Z}_{\sol}}$, we let $\Aff_{\cA}$ be the slice category of solid affinoid spaces over $\cA$. We denote the \textit{analytic spectrum} $\Anspec(\cA)$ the presheaf on $\Aff_{\mathbb{Z}_{\sol}}$ valued in anima by $\cA$.
    We let $\Aff^{b}_{R_{\sol}}\coloneq\AffRing^{b,op}_{R_{\sol}}$ be the category of \textit{bounded affinoid spaces}. 
     \item We define $\kappa$-small \textit{Tate stack} over $R_{\sol}$ by the category $\mathrm{Sh}_{!}(\Aff^{b}_{R_{\sol},\kappa})$ of sheaves on anima over $\Aff^{b}_{R_{\sol,\kappa}}$ under $!$-topology. Usually we will omit $\kappa$ and call it \textit{Tate stack}.
    \item For a \textit{Tate stack} $X\in \mathrm{Sh}_{!}(\Aff^{b}_{R_{\sol},\kappa})$, the functor $D_{\sol}(\_)$ naturally extend as $D_{\sol}(X)=\underset{\Anspec \cA\rightarrow X}{\lim}D_{\sol}(\cA)$. We denote by $\bperf(X)$ the full subcategory of dualizable objects in $D_{\sol}(X)$.
    \eenum
    \end{definition}

For later application we also recall here another class of analytic stack. 
\hiddensubsubsection{Gelfand stack}
Analogous to Banach algebra, which admit a (semi)norm towards $\mathbb{R}_{\geq0}$, one attempts to define a norm on the abstract bounded algebras. This can be done partially in the sense that one can define what is the subring (resp., submodule) of norm $r\leq 1$ (resp., $r\leq r'$, $r'\in \mathbb{R}_{\geq 0}$). 

Recall from \cite{anschütz2025analyticrhamstacksfarguesfontaine}*{Definition 2.2.3} that, for any $r\in \mathbb{R}_{\geq 0}$ and any light profinite set $S$, one can define ${\mathbb{Q}_p}_{\sol}\{\mathbb{N}[S]\}_{\leq r}^{\dagger}$ the \textit{free overconvergent function of norm $r$ over $S$}. This is a solid subring of ${\mathbb{Q}_{p}}_{\sol}[\![\mathbb{N}[S]]\!]$ whose $S'$ points are series $\underset{n}{\sum} a_n$, with $a_n\in {\mathbb{Q}_p}_{\sol}[S^n/\Sigma_n](S')$ such that there is $r'>r$ with $|a_n|r'^n\rightarrow 0$, where the norm $|\cdot|$ is taken by the $p$-adic topology on ${\mathbb{Q}_p}_{\sol}[\mathbb{N}[S]](S')$.

For a solid ring $A$ over $\mathbb{Q}_p$, we define $A^{\leq 1}$ the \textit{subring of norm 1} by subsheaf of $A$ whose $S$ points are given by 
$$A^{\leq 1}(S)=\mathrm{Map}_{\mathrm{CAlg(D_{\sol}(\mathbb{Q}_p))}}({\mathbb{Q}_p}_{\sol}\{\mathbb{N}[S]\}_{\leq 1}^{\dagger},A).$$
Similarly, we have the solid $A^{\leq 1}$ submodule $A^{\leq r}$ define as a subsheaf of $A$ whose $S$ points are given by 
$$A^{\leq r}(S)=\mathrm{Map}_{\mathrm{CAlg(D_{\sol}(\mathbb{Q}_p))}}({\mathbb{Q}_p}_{\sol}\{\mathbb{N}[S]\}_{\leq r}^{\dagger},A).$$
For each $r\leq 1$, $A^{\leq r}$ is an ideal of $A^{\leq 1}$, and we denote by $A^{u}\coloneq \underset{r}({\lim} A^{\leq 1}/A^{\leq r})[1/p]$ the \textit{uniform completion} of $A$.
\begin{definition}\label{Gelfand stack}
     \benuma
     \item  A bounded $\mathbb{Q}_p$ algebra $\cR$ is called a \textit{Gelfand ring} if $\cR^{u}$ is a separated $\mathbb{Q}_p$ Banach algebra. We denote the site of \textit{Gelfand rings} equipped with topology generated by $!$-covers by $\mathrm{GelfRing}$, and the slice category over a Gelfand ring $\cR$ by $\mathrm{GelfRing}_{\cR}$. We denote the category of sheaves on $\mathrm{GelfRing}_{\cR}$ under $!$-topology by $\mathrm{GelfStk}_{\cR}$. We call a Gelfand stack an \textit{affinoid space} if it lies inside the essential image of Yoneda functor $\mathrm{GSpec}\colon \mathrm{GelRing}^{\mathrm{op}}\rightarrow \mathrm{GelStk}.$
    \item For a Gelfand stack $X\in \mathrm{GelfStk}$, the functor $D_{\sol}(-)$ on $\mathrm{GelfRing}^{\mathrm{op}}$ naturally extend as $D_{\sol}(X)\coloneq \underset{\mathrm{GSpec A}\rightarrow X}{\lim}D_{\sol}(A.)$. We denote by $\bperf(X)$ the full subcategory of dualizable objects in $D_{\sol}(X)$.
   \item  Denote functor $j_*\colon\mathrm{GelfStk}_{\cR}\rightarrow \rsh_{!}(\Aff^{b}_{\cR}) $ by the left Kan extension induced from forgetful functor $j\colon \mathrm{GelfRing}_{\cR}\rightarrow \AffRing_{\cR}^{b}$. 
    \eenum
\end{definition}
As a non-trivial example of Gelfand stack, associate to any condensed set $S$, one can construct the \textit{Betti stack} $S^{\mathrm{Betti}}$ of $S$ via left Kan extending of the association $S'\mapsto C(S,\mathbb{Z})\tensor \mathbb{Q}_p$ for any profinite set $S'$. 

To any Gelfand ring $A$, since $A^{u}$ is a Banach algebra, it admits a\textit{Berkovich spectrum} $\cM(A^{\mathrm{u}})$ the space of semi-norms on it. Recall from \cite{anschütz2025analyticrhamstacksfarguesfontaine}*{Section 4.3} that, to any Gelfand stack $Y$, one can associate funtorially its underlying \textit{Berkovich spectrum} $|Y|$ which is a topological space obtained via left Kan extending of the association $\mathrm{GSpec}A \mapsto \cM(A^{\mathrm{u}})$. Moreover, there is a map of Gelfand stacks $Y\rightarrow |Y|^{\mathrm{Betti}}.$ We call a morphism of Gelfand stacks $j\colon U\rightarrow Y$ is an \textit{open immersion} if it give rise to a Cartesian diagram
$$
\begin{tikzcd}
U \arrow[d] \arrow[r, "j"]            & Y \arrow[d]          \\
{|U|}^{\mathrm{Betti}} \arrow[r, "|j|"] & {|Y|}^{\mathrm{Betti}} 
\end{tikzcd}
$$
and $|j|$ is an open embedding of topological spaces.
\begin{definition}
 Let $\cR$ be a Gelfand ring. A \textit{derived Berkovich space} over $\mathrm{GSpec}\cR$ is a Gelfand stack $X$ such that there is a cover $\{U_i\rightarrow X\}_{i\in I}$ consist of open immersions, such that each $U_i$ admits an open immersion into an affinoid space over $\mathrm{GSpec}\cR$.
\end{definition}
Later on, we will see that sometimes it is more convenient to work with $\mathrm{GelfStk}_{\cK}$ rather than Tate stack. For our application, one of the reasons is that nil-perfectoid algebras are \textit{Fredholm}.

\begin{definition}
    An analytic ring $\cA$ is called \textit{Fredholm} if all the dualizable objects in $D_\sol({\cA})$ are discrete. Or equivalently: $\bperf(\cA)\simeq \mathrm{Perf}(\underline{\cA}(*))$. Here, we denote by $\bperf(\cA)$ the $\infty$-sub category consisting of dualizable objects in $D_{\sol}(\cA)$. 
\end{definition}
\begin{eg}
   Gelfand rings are Fredholm by \cite{anschütz2025analyticrhamstacksfarguesfontaine}*{Proposition 3.3.5}.
\end{eg}

\csub{Realization functors}
Let $\cK/R_{\sol}$ be a $!$-able \textit{bounded affinoid algebra} over $R_{\sol}$, with underlying discrete ring $K\coloneq\underline{\cK}(*)$ and $K^+\coloneq\underline{\cK^+}(*)$. We denote by $K^{\mathrm{disc}}$ the analytic ring associated to the discrete Huber pair $(K,\mathbb{Z})_{\sol}$.

Start from a classical algebraic geometric object, there are many ways to realize it as an object in the \textit{Tate stack} For our purpose, in this paper we only consider two of the realization functors, namely the \textit{algebraic realization} and \textit{analytification} functor. 

\begin{definition}\label{analytification def}\hfill
\benuma
    \item For any $A\in \AniAlg_K$, we can define the solid stack $\Spec{A}^{\mathrm{alg}}=\Anspec(A, \mathbb{Z})_{\sol}\times_{\Anspec K} \Spa \cK$. This defines a functor $\mathrm{Aff}_{K}^{\cK\text{-}\mathrm{alg}}\colon \mathrm{Affine}_K\rightarrow \Aff_{\cK}$ We define the \textit{algebraic realization} functor for algebraic stacks $(\_)_{\bullet}^{\mathrm{\cK\text{-}alg}}\colon \rsh_{\bullet}(\mathrm{Affine}_{K}^{\mathrm{op}})\rightarrow \mathrm{Sh}_{!}(\Aff_{\cK}^{\mathrm{op}})$  by left Kan extension of $\mathrm{Aff}_{K}^{\mathrm{alg}}$, where $\bullet\in \{\et, fppf\}$. We will simply denote by $(-)^{\mathrm{alg}}$ if the topology we are using and the choice of $\cK$ are clear from the context.
    \item We define the functor $\mathrm{Affine}_{K}^{\mathrm{an},\cK}\colon \mathrm{Affine}_K\rightarrow \rsh_{!}(\Aff_{\cK}^{b})$ to be the composite 
    $\\\mathrm{Affine}_K \stackrel{(\_)^{\mathrm{alg}}}\rightarrow \mathrm{PSh}(\Aff_{K^{\mathrm{disc}}})\stackrel{\times_{K^{\mathrm{disc}}}\cK}\rightarrow{\mathrm{PSh}}(\Aff_{\cK})\stackrel{k^{*}}\rightarrow \mathrm{PSh}(\Aff^{b}_{\cK})\rightarrow \mathrm{Sh}_{!}(\Aff^{b}_{\cK})$\\ where $k^*$ is the pullback functor along $k\colon \Aff^b_{\cK}\rightarrow \Aff_{\cK}$, the last is sheafification.
    And we define the \textit{analytification} functor $(-)_{\bullet}^{\mathrm{an},\cK}\colon \rsh_{\bullet}(\mathrm{Affine}_{K}^{\mathrm{op}})\rightarrow \rsh_{!}(\Aff_{\cK}^{b})$ to be the left Kan extension of $\mathrm{Affine}_K^{\mathrm{an}}$. We will simply denote by $(-)^{\mathrm{an}}$ if the topology we are using and the choice of $\cK$ are clear from the context.
    \item If $\cK$ is a Gelfand ring, then for $\Spec A \in \Affine_{K}$, we define $(\Spec A)^{\mathrm{Gel}\text{-}\mathrm{an}}$ as the the restriction of $(\Spec A)^{\mathrm{an}}\in \mathrm{Sh}_{!}(\Aff^{b}_{\cK})$. 
    We define the analytification functor $()^{\mathrm{Gel}\text{-}\mathrm{an}}\colon\rsh_{\bullet}(\Affine_{K}^{\mathrm{op}})\rightarrow \rsh_{!}(\mathrm{GelfRing}_{\cK}) $  as the left Kan extension of  $(\_)^{\mathrm{Gel}\text{-}\mathrm{an}}:\Affine_K^{\mathrm{op}}\rightarrow \mathrm{GelfStk}_{\cK}$.
 \eenum
  By construction, for each $A\in \AniAlg_K$, there is a natural sequence of functors 
    $$\mathrm{Qcoh}(A)\rightarrow \mathrm{Qcoh}(A)\tensor_{\mathrm{Qcoh(K)}}D_{\sol}(\cK)\simeq D_{\sol}((\Spec A)^{\mathrm{alg}})\rightarrow D_{\sol}((\Spec A)^{\mathrm{an}}).$$
    We denote by $j_{\Spec A}$ the composite functor. By descent, this associate to a functor $$j_Y\colon \mathrm{Qcoh}(Y)\rightarrow D_{\sol}(Y^{\mathrm{an}})$$ for any stack $Y$.
\end{definition}

Given a stack $X$ and an affine chart $\pi: U\surjects X$, one can recover $X$ via the geometric realization of $\check{C}ech$-nerve of $\pi$. It is a priori not clear that in general, for these realization functors $(-)^?$, $\check{C}ech$-nerve of $\pi^?$ can recover $X^{?}$ via geometric realization. In fact, there should be counter examples. We will show in the next two sections that, under certain topologies, these realization functors preserve colimits. In particular, this will imply preservation of geometric realization.

\subsubsection{Algebraic realization}

We will show the following.

\begin{thm}\label{alg preserve colimit}
    The \textit{analytification} functor $(-)^{\mathrm{alg}}\colon \rsh_{\et}(\Affine_K^{\mathrm{op}})\rightarrow \rsh_{!}(\Aff_{\cK})$ preserves colimit.
\end{thm}

Before we proceed to the proof, we first need to show the following necessary condition of \Cref{alg preserve colimit}.

\begin{lemma}\label{fppf to tate}
    The functor $(\_)^{\mathrm{alg}}\colon \Affine_{K}^{\mathrm{op}}\rightarrow \rsh_{!}(\Aff_{\cK})$  preserves fppf covers \up{i.e., if $R\rightarrow R'$ is a fppf map then $(\Spec R')^{\mathrm{alg}}\rightarrow (\Spec R)^{\mathrm{alg}}$ is $!$-surjection}.
\end{lemma}
\begin{proof}
This is straightforward since any fppf cover is descendable maps $R\rightarrow R'$ of algebras (\cite{MATHEW2016403}*{Corollary 3.33}), thus $\Anspec(R')\rightarrow \Anspec(R)$ is a $!$-surjection. Since $!$-surjection is stable under base change, this shows $$(\Spec R')^{\mathrm{alg}}=\Anspec (R')\times_{\Anspec K}\Spa \cK\rightarrow \Anspec (R)\times_{\Anspec K}\Spa \cK=(\Spec R)^{\mathrm{alg}}$$ is a $!$-surjection.

\end{proof}
\begin{proof}[Proof of \Cref{alg preserve colimit}]
    This directly follows from \Cref{continuous} and \Cref{fppf to tate}.
\end{proof}

\begin{lemma}\label{continuous}
    Given two sites $\cC$, $\cD$, and a functor $F:\cC \rightarrow \mathrm{Sh}(\cD)$. Suppose $F$ commutes with fiber product and if for any cover $\{C_j\rightarrow C\}_{i\in I}\in  \mathrm{Cov}(\cC)$, we have $\{F(C_j)\rightarrow F(C)\}_{i\in I}$ is an $epimorhpism$, then we have:\\
(1)   The functor $F^{p}\colon \mathrm{PSh}(\cD)\rightarrow \mathrm{PSh}(\cC)$ via $\cG\mapsto \cG\circ F$ preserves sheaf. We denote $F^{s}$ by its restriction to $\mathrm{Sh}(\cD)$\\
(2)    the left Kan extension functor $(F_{p})^{sh}\colon \mathrm{PSh(\cC)\rightarrow \mathrm{Sh}(\cD)}$ factors through $\mathrm{Sh}(\cC)$ via sheafification functor $()^{sh}\colon \mathrm{PSh}(\cC)\rightarrow \mathrm{Sh}(\cC)$. And the resulting functor $F_s\colon \mathrm{Sh}(\cC)\rightarrow \mathrm{Sh}(\cD)$ is left adjoint to $F^s$. In particular it commutes with colimit.

\end{lemma}
 Note that here if the target of $F$ is just $\cD$, then this is the  classical result of $continuous$ functor between sites. It is very possible there are many much better references recorded this lemma, but since the author did not managed to find it and the proof is almost identical, we decide to reproduce it here. 

\begin{proof}
For (1) it sufficient to show that for $\cG \in \mathrm{Sh(\cD)}$, every $\{f\colon C^{'}\rightarrow C\}\in  \mathrm{Cov}(\cC)$ satisfies the following:
$$\underset{[n]\in\Delta}{\lim} F^{p}(\cG((C')^{n/C})\simeq F^{p}(\cG)(C).$$
But by definition, $\underset{[n]\in\Delta}{\lim} F^{p}(\cG)((C')^{n/C})\simeq \underset{[n]\in\Delta}{\lim} \cG (F(C')^{n/F(C)})$. By our assumption, $F(f)\colon F(C')\rightarrow F(C)$ is an epimorphism (thus effective epimorphism), so we have that $F(C)\simeq \underset{[n]\in\Delta}{\colim}(F(C')^{n/F(C)})$, which implies the equivalence.

On the presheaf level $F_p$ is left ajoint to $F^p$ by definition, indeed since every presheaf is the colimit of $Yoneda$ sheaf $h_c$ for some $c\in \cC$, it sufficient to check the following\upshape: For any $c\in \cC$, $\cH \in \rpsh(\cD)$
$$\rmap(F_{p}(c), \cH)\simeq \rmap(c, F^p(\cH)).$$
But this is exactly the defination of presheaf $F^p(\cH)$.
Also recall that sheafification functor $()^{sh}$ is left adjoint to forgetful functor $\eta$ from sheaf category to presheaf category. By \Cref{continuous}(1), restricting to $\mathrm{Sh}(\cC)$ and $\mathrm{Sh}(\cD)$, we see that $(F_p\circ \eta)^{sh}$ is left adjoint to $F^s$. \\
To prove $(F^p)^{sh}$ factors through $()^{sh}$, sufficient to check that every $\cG\in \mathrm{PSh}(\cC)$, $(F_p(\cG))^{sh}\simeq (F_{p}\circ\eta\circ(\cG)^{sh})^{sh}$. This can be checked directly as follow\upshape: for any $\cF\in \mathrm{Sh}(\cD)$ we have
\begin{align*}
    \mathrm{Map}_{\mathrm{Sh}(\cD)}((F_{p}(\eta\circ(\cG)^{sh}))^{sh}, \cF)&= \mathrm{Map}_{\rsh(\cD)}((\cG)^{sh}, F^{s}(\cF))\\
                                                                           &= \rmap_{\rsh(\cD)}(\cG, F^p(\eta\circ \cF))\\
                                                                           &= \rmap_{\rpsh(\cD)}(F_{p}(\cG), \eta\circ \cF)\\
                                                                           &= \rmap_{\rsh(\cD)}((F_{p}(\cG))^{sh}, \cF).
\end{align*}

Which by Yoneda equivalence this is exactly the desired equivalence we want.
\end{proof}

\subsubsection{Analytification}
We will show the following.
\begin{thm}\label{main of analytification section}
    The \textit{analytification} functor $(-)^{\mathrm{an}}\colon \rsh_{\et}(\Affine_K^{\mathrm{op}})\rightarrow \rsh_{!}(\Aff_{\cK}^{b})$ preserves colimit.
\end{thm}

Before we proceed to the proof, we will need collect more basic properties of analytification functor.

\begin{lemma}\cite{camargo2024analyticrhamstackrigid}*{lemma 2.7.25}
    Let $\mathbb{D}_{R}\coloneq\Anspec R\langle T\rangle_{\sol}$ be the unit affinoid disc. For $S$ a profinite set let us write $\mathbb{A}_{R,S}^{\mathrm{alg}}\coloneq\Anspec(R_{\sol}[\mathbb{N}[S]])$ and $\mathbb{A}_{R,S}^{\mathrm{an}}\coloneq\underset{n\in \mathbb{N}}\bigcup \Anspec(R_{\sol}\langle\mathbb{N}[\pi^n S]\rangle)$. Then there is an equivalence 
    $$(\mathbb{D}_{R}^n\times \mathbb{A}_{R,S}^{\mathrm{alg}})^{\mathrm{an}}\simeq \mathbb{D}_{R}^n\times \mathbb{A}_{R,S}^{\mathrm{an}} $$
    It follows that the statement is also true when replace $R$ by any further base change to $\Spa \cK$.
\end{lemma}

We now describe analytification functor for general finite type $K$-algebras.
\begin{lemma}\label{an for affine}
    Let $A$ be an finite type $K$-animated ring \up{i.e., $\pi_0$ being finite type}. Choosing a map of animated ring $f\colon K[x_1,x_2,...x_n]\surjects A$ such that induces surjection on $\pi_0$, we have that $k_{*}(\Spec A)^{\mathrm{an}}$ fit into the following Cartesian diagram.
    $$\begin{tikzcd}
k_{*}\Spec(A)^{\mathrm{an}} \arrow[d] \arrow[r] & \mathbb{A}^{\mathrm{an},n} \arrow[d] \\
\Anspec (A) \arrow[r]           & \mathbb{A}^n          
\end{tikzcd}$$

\end{lemma}
\begin{proof}
    First we claim that $k^*(\Anspec(A)\times_{\mathbb{A}^n}\mathbb{A}^{n,\mathrm{an}})\simeq (\Anspec A)^{\mathrm{an}}$. Indeed, this amount to saying that for any bounded ring $\cB$ and a map $g\colon\Anspec(\cB)\rightarrow \Anspec A$, the composition $g\circ f\colon \Anspec(\cB)\rightarrow \mathbb{A}^n$ factor through $\mathbb{A}^{n,\mathrm{an}}$. But this is true by definition since any element in $\underline{\cB}(*)$ is \textit{bounded}, and the map $g\colon\Anspec(\cB)\rightarrow \Anspec A$ is equivalent to a map $g^{*}\colon A\rightarrow \underline{\cB}(*)$.

    Now we need to show that $k_{*}k^{*}(\Anspec(A)\times_{\mathbb{A}^n}\mathbb{A}^{n,\mathrm{an}})\simeq \Anspec(A)\times_{\mathbb{A}^n}\mathbb{A}^{n,\mathrm{an}}$. Since $k_{*}$ is defined via left kan extension, by \Cref{left kan being id} this amount to prove that for any \textit{solid affinoid space} $\cS$ and a map $h\colon \cS\rightarrow \Anspec(A)\times_{\mathbb{A}^n}\mathbb{A}^{n,\mathrm{an}}$, $h$ factors through some \textit{bounded affinoid space} $\Anspec(\cS^{b})$.
    $$\begin{tikzcd}
\Anspec(\cS) \arrow[rd, "h"] \arrow[r, "\exists", dotted] & {\Anspec(\cS^{b})} \arrow[d, "h'", dashed] \arrow[rd, "p"] \arrow[dd, "h^b" description, bend right, shift right=7] &                                    \\
                                                          & \Spec(A)^{\mathrm{an}} \arrow[d, "j"] \arrow[r, "f^{\mathrm{an}}"]                                                              & {\mathbb{A}^{n,\mathrm{an}}} \arrow[d, "j"] \\
                                                          & \Anspec (A) \arrow[r, "f"]                                                                                           & \mathbb{A}^n                      
\end{tikzcd}$$
We now construct $\cS^b$. By functor of point the map $h$ correspond to a map $(j\circ h)^*\colon A\rightarrow \underline{\cS}(*)$ where the image of $f(x_i)$ lies inside $(\underline{\cS})^b$, thus since bounded elements forms a sub-condensed ring, $(j\circ h)^*$ factors though $\underline{\cS}^b$. Now we denote $\cS$ by the analytic ring whose underlying condensed ring is $\underline{\cS}^b$ with induced analytic structure from $\Anspec A$. Then we know by construction, there exist a map $h^b\colon\Anspec(\cS^b)\rightarrow \Anspec(A)$ factorize $j\circ h$. Now we claim $h^b$ also factors through some $h'$ such that $j\circ h'\simeq h^b$, for this we only need to show that $f\circ h^b $ factors through $p\colon\Anspec(\cS^b)\rightarrow \mathbb{A}^{n,\mathrm{an}}$, but again this follows from $f(x_i)$ lands inside $(\underline{\cS^b})^b=\underline{\cS^b}$.
\end{proof}
\begin{lemma}\label{left kan being id}
    Let $k\colon \cC\rightarrow \cD$ be a functor of sites. Let $X\in \mathrm{Sh}(\cD)$ and assume there exist a idempotent functor $F\colon \cD\rightarrow \cC$ and a natural transformation $\eta\colon id\rightarrow F$, such that for any $d\in \cD$, $X(d)\underset{\eta}{\simeq} X(F(d))$. Then we have that $k_{*}k^{*}X\simeq X$. Here $k_{*}$ denote the left Kan extension functor $\mathrm{Sh}(\cC)\rightarrow \mathrm{Sh}(\cD)$ induced by $k$, and $k^{*}$ by the restrication functor on sheaves.
\end{lemma}
\begin{proof}
    By definition, $k_*k^{*}X=\underset{c\rightarrow k^{*}X,c\in \cC}{\colim} k_{*}c\simeq \underset{k_* c\rightarrow X,c\in \cC}{\colim}k_*{c}\simeq \underset{k_*F(k_{*}c)\rightarrow X,c\in \cC}{\colim}k_*F(k_*c)$. \\
   On the other hand $X\simeq \underset{d\rightarrow X,d\in \cD}{\colim}d\simeq \underset{k_*F(d)\rightarrow X,d\in \cD}{\colim}k_*F(d)$. \\
   But these two diagrams $\{k_*F(k_{*}c)\rightarrow X,c\in \cC\}$ and $\{k_*F(d)\rightarrow X,d\in \cD\}$ factor through each other by idempotentness of $F$.
    
\end{proof}

To prove \Cref{main of analytification section}, we will need the following immediate consequences of it.

\begin{prop}\label{descenting solid sheaf of an}
  For affine $K$-scheme $X$, any Zariski cover $X=\bigcup_{i\in I} \Spec(A_i)$. The corresponding \textit{analytification} $\bigcup_{i\in I}\Spec(A_i)^{\mathrm{an}}\rightarrow X^{\mathrm{an}}$ forms an universal \up{open} $!$-cover. In particular, $D_{\sol}(X^{\mathrm{an}})\simeq \lim_{i\in I}D_{\sol}(\Spec(A_i)^{\mathrm{an}}) $.
\end{prop}
\begin{proof}
   Let $X=\Spec(B)$ and we may assume $A_{i}=B_{f_i}$ for some $f_i\in B$, with condistion that $f_1+f_2+...f_n=1$. Furthermore, one can assume $n=2$. So sufficient to show $\Spec(B_f)^{\mathrm{an}}\bigcup\Spec(B_{1-f})^{\mathrm{an}} $ forms a universal $!$-cover of $(\Spec B)^{\mathrm{an}}$, but by \Cref{an for affine}, we have the following pull-back diagram:
    $$\begin{tikzcd}
\Spec(B_f)^{\mathrm{an}} \arrow[rdd] \arrow[rr]     &                                                    & \mathbb{G}_m^{\mathrm{an}} \arrow[rdd] \arrow[rd] &                               \\
\Spec(B_{1-f})^{\mathrm{an}} \arrow[rd] \arrow[rru] &                                                    &                                          & {\mathbb{A}^{1,\mathrm{an}}} \arrow[d] \\
                                             & \Spec(B)^{\mathrm{an}} \arrow[rr, "f"] \arrow[rru, "1-f"] & {} \arrow[r]                             & {\mathbb{A}^{1,\mathrm{an}}}          
\end{tikzcd}$$
Which, up to a change of coordinate, can be again expressed as the following \up{two} pullback diagrams:
$$\begin{tikzcd}
\Spec(B_{1-f})^{\mathrm{an}} \arrow[d] \arrow[rr]&  & \mathbb{G}_{m}^{\mathrm{an}} \arrow[d, "x \mapsto 1+x"] \\
\Spec(B)^{\mathrm{an}} \arrow[rr, "f"]            &  & {\mathbb{A}^{\mathrm{an},1}}                            \\
\Spec(B_f)^{\mathrm{an}} \arrow[u] \arrow[rr]     &  & \mathbb{G}_{m}^{\mathrm{an}} \arrow[u, "x\mapsto x"']  
\end{tikzcd}$$
Thus to prove the claim we only need to deal with the universal case where $B=K[X]$ and $f=X$.  By \Cref{list of disk} (3) $\Spec (B_{f})^{\mathrm{an}}$(resp. $\Spec (B_{1-f})^{\mathrm{an}}$) is complement of \textit{overconvergent disc} arround $0$ (resp. $1$), and by \Cref{list of disk} (4.5.2), the corresponding idempotent algebra doesn't intersect. Thus by \cite{camargo2024analyticrhamstackrigid}*{prop 3.1.15} this forms a (open) universal $!$-cover.
\end{proof}

\begin{lemma}\label{et to nilperf}
    The functor $(\_)^{\mathrm{an}}\colon \Affine_{K}^{\mathrm{op}}\rightarrow\mathrm{GelfStk}_{\cK}$\up{resp., $ \rsh_{!}(\Aff_{\cK}^{b})$}  preserves covers \up{i.e., if $R\rightarrow R'$ is an étale cover then $(\Spec R')^{\mathrm{an}}\rightarrow (\Spec R)^{\mathrm{an}}$ is $!$-surjection}.
\end{lemma}
\begin{proof}
    It is enough to show that every étale cover $R \rightarrow R'$ of $R\in \Affine_K$, $(\Spec R')^{\mathrm{an}}\rightarrow (\Spec{R})^{\mathrm{an}}$ is a $!$-surjection.
     By \Cref{descenting solid sheaf of an} we may pass to Zariski covers thus reduce to standard étale covers. By the proof of \cite{bhatt2011annihilatingcohomologygroupschemes}*{Lemma 2.2}, This cover is always refined by a finite flat cover $R\rightarrow R_1=R[x]/f$ (where $f$ is a monic polynomial and $R_1$ is finite free over $R$) then a Zariski open cover of $\Spec(R_1)$. By \cite{complex}*{Lemma 6.11}, we have $$(\Spec R_1)^{\mathrm{an}}=(\Spec R)^{\mathrm{an}}\times_{(\Spec R)^{\mathrm{alg}}}\Spec(R_1)^{\mathrm{alg}}.$$ 
      This is a descentable cover since $R\rightarrow R_1$ is a split map.
      Thus the lemma follows.

\end{proof}

As a immediate consequence, by reasons below, sometime we don't need to distinguish $(-)^{\mathrm{an}}$ and $(-)^{\mathrm{Ge
l}\text{-}an}$.
\begin{prop}
   If $\cK$ is the analytic ring assocaited to a separable Banach algebra, then for any $K$-scheme $X$ of finite type, $j_{*}(X)^{\mathrm{Gel}\text{-}\mathrm{an}}\simeq X^{\mathrm{an}}$.
\end{prop}
From now on we will abuse the notation $(\_)^{\mathrm{Gel}\text{-}\mathrm{an}}$ by $(\_)^{\mathrm{an}}$ if the target category is clear from context.

Finally, we arrives at our main result of this section.

\begin{thm}\label{an preserve colimit}
    The \textit{analytification} functor $(-)^{\mathrm{an}}\colon \rsh_{\et}(\Affine_K^{\mathrm{op}})\rightarrow \mathrm{GelfStk}_{\cK}$ \up{resp.,  $\rsh_{!}(\Aff_{\cK}^{b})$} preserves colimit.
\end{thm}
\begin{proof}
    This directly follows from \Cref{continuous} and \Cref{et to nilperf}.
\end{proof}

\section{GAGA theorem revisited}

In this chapter we will prove an ``enhancement'' of classical GAGA theorem. Namely, for any non-archimedean field $K$ (in fact much more general) and a $K$-scheme $X$, we will show that there is a notion of \textit{solid quasi-coherent sheaves} $D_{\sol}(X^{\mathrm{an}})$ of the analytificaion $X^{an}$ and certain enlargement $D_{\sol}(X^{\mathrm{alg}})$ of $\mathrm{Qcoh}(X)$ so that a natural pullback functor induces equivalences of categories $D_{\sol}(X^{\mathrm{an}})\simeq D_{\sol}(X^{\mathrm{alg}})$. To do this we will need to first recall the theory of analytic geometry developed in \cite{condensed}, \cite{analytic}, \cite{camargo2024analyticrhamstackrigid}. The strategy presented in this chapter is essentially due to \cite{complex}.

\csub{Enhanced GAGA theorem}
Let $\cK$ be a bounded analytic ring and we denote by $K\coloneq \underline{\cK}(*)$, we will abbreviate the analytification functor $(-)^{\cK\text{-}an}$ by $(-)^{\mathrm{an}}$.

In this section we will prove the following enhanced GAGA theorem:
\bthm\label{gagaga}
If $X$ is proper scheme over $\Spec{K}$, then $$D_{\sol}(X^{\cK\text{-}\mathrm{an}})\simeq D_{\sol}(X^{\mathrm{\cK\text{-}alg}}).$$
\ethm
We will carry out the proof by first recalling the axiomatic GAGA theorem from \cite{complex}. This consist of first recalling certain basic properties about non-archimedean disks, then using these basic properties to construct certain $categorical\ locale$ for $X^{\mathrm{an}}$ and $X^{\mathrm{alg}}$. The key observation is that in general the categorical locale for $X^{\mathrm{an}}$ is open in the one for $X^{\mathrm{alg}}$, but in the case of proper a scheme the corresponding open immersion becomes equivalence by valuation criterion. The proof essentially repeats the argument from \cite{complex}, still, we decide to present the detailed version of their proof.

\hiddensubsubsection{Overconvergent closed disc}\label{intro of overconve}
The key observation towards \Cref{gagaga} is the following: for simplicity let us assume $X$ is just projective line $\mathbb{P}_{K}^1$, we observe that $\mathbb{P}_K^{1}$ not only can be covered by two algebraic affine lines $\mathbb{A}_{K}^1$, it can also be covered by an algebraic affine line around $0$ and the ``overconvergent closed disc'' $\mathbb{D}_{\geq \infty}^{\dagger}$ \footnote{which is analogous to the \textit{analytic germ} in complex geometry, see \cref{overconvergent disc}} around $\infty$ (this is \cref{union of disc cover everything}). Unwinding definition, the ``complement'' of that overconvergent closed disc is the analytic affine line $\mathbb{A}_K^{1,\mathrm{an}}$, but this means that one can also cover $\mathbb{P}_K^1$ by two analytic affine lines (by enlarging the overconvergent disc, using \cref{disk equal intersect of disk}), which shows that in fact the algebraic projective line is the ``same'' as the analytic one since the latter is precisely glued from two analytic affine line in the same manner. To make this observation precise, we will start by recalling basic properties about an overconvergent closed disc.

\begin{definition}\label{overconvergent disc}Let $\pi\in K$ a pseudo-uniformizer corresponding to the map $R_{\sol}=(\mathbb{Z}(\!(\pi)\!), \mathbb{Z}[\![\pi]\!])_{\sol}\rightarrow \cK$.
\benuma
\item We define $\mathbb{D}_{R, \leq |\pi|^n}$ the \textit{open disk of radius $|\pi|^n$} over $\Spa R_{\sol}$  as $\mathbb{D}_{R, \leq |\pi|^n}\coloneq\mathbb{D}_{R }\times_{\mathbb{A}_{R}^1} \mathbb{A}_{R}^{1}$, where the map $\mathbb{A}_R^1\rightarrow \mathbb{A}_R^1$ is defined via $T\mapsto \pi^n T$. Let the \textit{open disk of radius $|\pi|^n$} over $\Spa \cK$ to be $\mathbb{D}_{\cK,\leq |\pi|^n}\coloneq\mathbb{D}_{R, \leq |\pi|^n}\times_{\Spa R} \Spa \cK$.  We define the \textit{open disk at infinity of radius $|\pi|^n$} to be $\mathbb{D}_{R, \geq |\pi|^{-n}}\coloneq \mathbb{G}_{m, R}\times _{\mathbb{A}_{R}^1} \mathbb{D}_{R, \leq |\pi|^n}$, where the map $\mathbb{G}_{m,R}\rightarrow \mathbb{A}_{R}^1$ is the composition of natural embedding and inverse map on $\mathbb{G}_{m,R}$. 
\item We define the \textit{overconvergent disk} over $\Spa R_{\sol}$ to be the limit (in the category of \textit{Tate stack}) $\mathbb{D}_{R}^{\dagger}\coloneq\lim_n \mathbb{D}_{R, \leq |\pi|^n}$. Notice that this correpond to, in the category of analytic rings, to the colimit of analytic rings $\colim_n (\mathbb{Z}(\!(\pi)\!)\langle \frac{T}{\pi^n}\rangle,\mathbb{Z}[\![\pi]\!]\langle \frac{T}{\pi^n}\rangle)_{\sol}.$ Similarly we define the \textit{overconvergent disk at infinity} to be the limit $\mathbb{D}_{R, \geq \infty}^{\dagger}\coloneq\underset{i\in \mathbb{N}}{\lim} \bar{\mathbb{D}}_{R, \geq i}.$
\item We define the \textit{overconvergent disk} (resp., \textit{overconvergent disk at infinity}) over $\cK$ to be $\mathbb{D}_{\cK}^{\dagger}\coloneq\mathbb{D}_{R}^{\dagger}\times_{\Spa R_{\sol}} \Spa\cK$ (resp., $\mathbb{D}_{\cK, \geq \infty}^{\dagger}\coloneq \mathbb{D}_{R, \geq \infty}^{\dagger}\times_{\Spa R_{\sol}}\Spa \cK$).
\item We define the \textit{closed disk of radius $|\pi|^n$} over $\Spa R_{\sol}$ to be $\Bar{\mathbb{D}}_{R, \leq |\pi|^n}\coloneq\Bar{\mathbb{D}}_{R }\times_{\mathbb{A}_{R}^1} \mathbb{A}_{R}^{1}$. And \textit{closed disk of radius $|\pi|^n$} over $\Spa \cK$ to be $\Bar{\mathbb{D}}_{\cK,\leq |\pi^{n}|}\coloneq \Bar{\mathbb{D}}_{R,\leq|\pi|^n}\times_{\Spa R_{\sol}} \Spa \cK$. We define the \textit{closed disk at infinity of radius $|\pi|^{-n}$ } to be $\Bar{\mathbb{D}}_{R, \geq |\pi|^{-n}}\coloneq \mathbb{G}_{m, R}\times _{\mathbb{A}_{R}^1} \Bar{\mathbb{D}}_{R, \leq |\pi|^n}$, where the map $\mathbb{G}_{m,R}\rightarrow \mathbb{A}_{R}^1$ is the composition of embedding and inverse map on $\mathbb{G}_{m,R}.$ 
\eenum
\end{definition} 

We defined overconvergent disk as the interesection of open disks, but in fact it is equivalent to the interesection of closed disks.

\begin{lemma}\label{disk equal intersect of disk}
   $ \mathbb{D}_{\cK}^{\dagger}\simeq \lim_{r}\Bar{\mathbb{D}}_{\cK, \leq |\pi|^r}.$
\end{lemma}
\begin{proof}
    It is sufficient to prove the statement for disks over $\Spa R_{\sol}$.\\
    One need to check that two limit diagrams from definition and the lemma are cofinal to each other. That is to say, if $r'<r$, then we have the factorization as follow\upshape: 
 $$   
\begin{tikzcd}
                                              & \closedisc_{\leq |\pi|^r} \arrow[rd, "(2)", hook] &   \\
\disc_{\leq |\pi|^r} \arrow[ru, "(1)", hook] \arrow[rr, "(3)", hook] &                         & \disc_{\leq |\pi|^{r'}}
\end{tikzcd}$$
The existence of $(1)$ is clear.\\
For $(2)$ is sufficient to prove that $\Anspec(R\langle\pi^rT\rangle, R^{+})\simeq\Anspec(R\langle\pi^rT\rangle, R^{+}\langle\pi^{r}T\rangle)$. Without loss of generality, we may assume that $r=0, r'=1$, and thus it is enough to see that $\pi T$ is \textit{solid} in $D_{\sol}(R\langle T\rangle,R^{+})$ which by definition equivalent to say that 
\begin{align}\label{idem}
    R\langle T\rangle\tensor_{\mathbb{Z}}^{\sol}\mathbb{Z}[\![q]\!][x]/(\pi Tx-1)\simeq 0 
\end{align}
but $R\langle T\rangle\tensor_{\mathbb{Z}}^{\sol}\mathbb{Z}[\![q]\!]\simeq (\lim_{n}(R^{+}/\pi^n)[\![q]\!][x])[1/\pi]$. Thus $(1-\pi Tx)^{-1}=1+(\pi Tx)+(\pi Tx)^2+ (\pi Tx)^3+....$ is convergent in this ring, implies $1-\pi Tx$ is invertible in  \Cref{idem} which finish the proof.
\end{proof}

In particular, this alternative description of overconvergent disk showed that it is a closed subspace of $\mathbb{A}_{\cK}^1$.

\begin{prop}[see also \cite{anschütz2025analyticrhamstacksfarguesfontaine}*{Lemma 2.2.11}]\label{list of disk}
    The \textit{closed disks} have the following properties\upshape:
\begin{enumerate}
  \item $\closedisc_{\cK, \leq |\pi|^r}\cap \closedisc_{\cK, \geq |\pi|^r{'}} = \emptyset \ \ \ if \ r<r' $\label{intersect of disc empty}
  \item  $\closedisc_{\cK, \leq |\pi|^r}\bigcup \closedisc_{\geq |\pi|^r{'}} = \mathbb{A}^1 \ forms\ a\ closed\ cover \textit{if r<r'}$\label{union of disc cover everything}
  \item $\closedisc_{\leq r}\times \closedisc_{\leq r'}\subseteq \closedisc_{rr'}$
  \item $\closedisc_{\leq r} + \closedisc_{\leq r'} \subseteq \closedisc_{\max\{r,r'\}}$\label{addition of disc}
  
\end{enumerate}
\end{prop}
 \begin{proof}
Again, it is sufficient to prove the corresponding statements for $R_\sol$.
      (2) follow from vanishing of tensor product of their corresponding idempotent algebra. More precisely, sufficient to prove that $O(\closedisc_{R, \leq  |\pi|^r})\tensor_{R[T]}^{\sol}O(\closedisc_{R,|\pi|^\geq r'})=0$, but by \Cref{overconvergent disc} we have that 
      \begin{align*}O(\closedisc_{\leq r})\tensor_{R_\sol}^{\sol}O(\closedisc_{\leq (r')^{-1}})/(1-T_1T_2)&\simeq O(\closedisc^2_{\leq r,\leq (r')^{-1}})/(1-T_1T_2)\\&= \bigg\{\Sigma_{m,n}a_{m,n}T_1^mT_2^n\ r^m/r'^n\bigg||a_{m,n}|\mapsto0\bigg\}. 
      \end{align*}
      
      Thus to show this algebra is zero, sufficient to show $1-T_1T_2$ is invertible in this algebra. Indeed, we claim that $(1-T_1T_2)^{-1}=1+T_1T_2+T_1^2T_2^2+....T_1^nT_2^n+....$ converge (in $O(\closedisc^2_{\leq r,\leq (r')^{-1}})$), since viewing it as a power series of the form $\Sigma_{m,n}a_{m,n}T_1^mT_2^n\ r^m/r'^n$,  we have that 
$$a_{m,n}=\begin{cases}
    (r'/r)^n &\text{m=n}\\
    0 & \text{m$\neq$n}
\end{cases}$$

      But $|r'/r|^n\mapsto 0$ as $n\mapsto \infty$ since $r> r'$. This shows $(1-T_1T_2)^{-1}\in O(\closedisc^2_{\leq r,\leq (r')^{-1}})$. 

      For (3) suffices to show that we have exact sequence as follow:
      $$0\rightarrow K[T]\rightarrow O(\closedisc_{\leq r})\oplus O(\closedisc_{\geq r'})\rightarrow O(\closedisc_{\leq r})\tensor_{K[T]}O(\closedisc_{\geq r'})\rightarrow 0 $$
      But this follows directly from the explicit description of function on discs\upshape: 
\begin{align*}
      &O(\closedisc_{\leq r})=\bigg\{\Sigma_{n \geq 0}a_nT^n\bigg | |a_n|r^n\mapsto0\bigg\}\\
      &O(\closedisc_{\geq r'})=\bigg\{\Sigma_{n\geq 0}b_nT^n\bigg | |b_n|r'^n\mapsto 0\bigg\}\\
      &O(\closedisc_{[r,r']})\coloneq O(\closedisc_{\leq r})\tensor_{K[T]}O(\closedisc_{\geq r'})=\bigg\{\Sigma_{n\in\mathbb{Z}}c_nT^n\bigg | |c_n|r^n\mapsto 0\ if\ n\mapsto +\infty, |c_n|r'^n\mapsto 0\ if\ n\mapsto -\infty \bigg\}
\end{align*}
For (4) this amount to show there is natural factorization as follows:
$$ \begin{tikzcd}
K[T] \arrow[r, "T\mapsto T_1T_2"] \arrow[d] & K[T_1, T_2] \arrow[d] \\
O(\closedisc_{\leq r+r'}) \arrow[r, dashed]        & O(\closedisc_{\leq r})\tensor O(\closedisc_{\leq r'})          
\end{tikzcd}$$
But the obvious map sending $\Sigma a_nT^n$ to $\Sigma a_n T_1^nT_2^n$ satisfy the convergence condition, thus defines the factorization.

 (5) is similar as (4), one needs to check that $\Sigma a_n(T_1+T_2)^n$ has the corresponding convergent condistion, but this equals to $\Sigma a_{m+n}\binom{m+n}{n}T_1^mT_2^n$, notice that $|a_{m+n}\binom{m+n}{n}|\leq |a_{m+n}|$ and $r^mr'^n\leq (r+r')^{m+n}$.
 \end{proof}
 
\begin{prop}\label{analytic affine line is open}
    The open immersion $\mathbb{A}^{1,\mathrm\mathrm{an}}\hookrightarrow \mathbb{A}^1$ is the complement of $\overdisc_{\geq \infty}\coloneq \underset{i\in \mathbb{N}}{\cap} \bar{\mathbb{D}}_{\geq i}$.
\end{prop}
\begin{proof}
    Denote by $\{A_i\}_{i\in \mathbb{N}}$ the corresponding idempotent algebras of (complement of) open immersions $\mathbb{D}_{\leq n}\hookrightarrow \mathbb{A}^{1}$. By \Cref{from locale to sheaf}, $\underset{i\in \mathbb{N}}{\colim} A_{i}=\underset{i\in \mathbb{N}}{\colim}O(\bar{\mathbb{D}}_{\geq i})$ is the corresponding idempotent algebra of $\mathbb{A}^{1,\mathrm{an}}$. Sufficient to prove that $\underset{i\in \mathbb{N}}{\colim} A_{i}\simeq \underset{i\in\mathbb{N}}{\colim}K\langle \pi^{-n}T^{-1}\rangle[T]$. By proof of \Cref{disk equal intersect of disk}, we know that the open immersion $\mathbb{D}_{\leq n}\hookrightarrow \mathbb{A}^1$ factors through $\bar{\mathbb{D}}_{\leq n+1}\rightarrow \mathbb{A}^1$, thus $\Anspec A_n$ and $\bar{\mathbb{D}}_{\leq n+1}$ form a closed cover of $\mathbb{A}^1$. By \Cref{list of disk} (1), we see that $\bar{\mathbb{D}}_{\leq n+1}\cap \bar{\mathbb{D}}_{\geq n+2}=\emptyset$. This implies that $\bar{\mathbb{D}}_{\leq n+1}\rightarrow \mathbb{A}^1$ factors through $\Anspec A_n
    \rightarrow \mathbb{A}^1$. Indeed, since $A_n$ and $O(\bar{\mathbb{D}}_{\leq n+1})$ forms a closed cover, we have fiber sequence \begin{equation}\label{lemma 3.3.6}
        1_{\mathbb{A}^1}\rightarrow O(\bar{\mathbb{D}}_{\leq n+1})\oplus A_n\rightarrow A_n\tensor O(\bar{\mathbb{D}}_{\leq n+1}).
    \end{equation}  But $\bar{\mathbb{D}}_{\leq n+1}\cap \bar{\mathbb{D}}_{\geq n+2}=\emptyset$, i.e., $O(\bar{\mathbb{D}}_{\leq n+1})\tensor O(\bar{\mathbb{D}}_{\geq n+2})=0$. Thus tensoring \Cref{lemma 3.3.6} with $O(\bar{\mathbb{D}}_{\geq n+2})$ we see that $A_n\tensor O(\bar{\mathbb{D}}_{\geq n+2})\simeq O(\bar{\mathbb{D}}_{\geq n+2})$. 
    
    Simlimarly, we deduce $\bar{\mathbb{D}}_{\geq n}\hookrightarrow \mathbb{A}^1$ factors through $\Anspec A_{n+2}\hookrightarrow \mathbb{A}^1$ via the facts of $\bar{\mathbb{D}}_{\leq n}\hookrightarrow \mathbb{A}^1$ factors through $\mathbb{D}_{\leq n+1}\hookrightarrow \mathbb{A}^1$, and the fact that (by \Cref{list of disk} (3)), $\bar{\mathbb{D}}_{\leq n}$ and $\bar{\mathbb{D}}_{\geq n}$ forms a closed cover of $\mathbb{A}^1$.
    
    This shows that $\underset{i\in \mathbb{N}}{\colim} A_{i}=\underset{i\in \mathbb{N}}{\colim}O(\bar{\mathbb{D}}_{\geq i})$ since two filtered diagram factor through each other.
\end{proof} 
 \hiddensubsubsection{Recollection of axiomatic GAGA}
As explained in the begining of \Cref{intro of overconve}, the main ingredient in \Cref{gagaga} is the usage of \textit{overconvergent disc} $\mathbb{D}_{\geq \infty}^{\dagger}$ which on the level of analytic ring correspond to an idempotent algebra $\cO(\mathbb{D}_{\geq \infty}^{\dagger})$ in $\mathrm{Mod}_{K[X]}(D_{\sol}(\cK))$. Indeed, one can in fact record the ``topology'' of a category (symmetric monoidal) by the idempotent algebras in that category, which leads to the notion of \textit{locale} that can trace back to the work of \cite{https://doi.org/10.1112/plms/pdq050}, see also \cite{complex}.
\begin{definition}
    A poset $\mathcal{G}$ is called a \textit{locale} if:
\benuma
\item 
Any subset $S\subset \cG$ has a supremum $\lor S$. \up{This also implies any subset $S\subset \cG$ has a infimum $\land S$ given by the supremum of the set $\{g\in\cG:(\forall V\in S) U\leq V \}$. In particular every pair of elements $U,V\in\cG$ have a meet $U\land V$}
\item
$ X\land(\lor_{I} Y_i)=\land_{I}(X\lor Y_i)$ For any \up{set} index $I$ and subset $X, Y_i$.

\eenum
Morphisms of locales are those order preserving map that preserve all \textit{meets} \up{i.e., $\land$} and \textit{joins} \up{i.e., $\lor$}.
\end{definition}

\begin{definition}
    A \textit{categorical locale} is the following paris $(X, C,f)$ where $X$ is a \textit{locale}, $C$ is a symmetric monoidal category and $f\colon \cS(\cC)\rightarrow X$ a morhpism of $locale$. Here $\cS(C)$ is the $locale$ consist of idempotent algebras in $C$.
\end{definition}
Now we recall some basic properties about \textit{categorical locale}.\\
Given symmetric monoidal $\infty$-category $C$, every closed subset $Z\in \cS(C)$ correspond to some idempotent algebra $A_Z\in C$, we denote $C(Z)\coloneq\rmod_{A_Z}(C)$. Formally we can denote $U$ the ``complement'' of $Z$ in $\cS(C)$ while we can define the Verdier quotient:
$$C(U)=C/C(Z)$$

\begin{prop}\cite{complex}*{Proposition 5.5}\label{from locale to sheaf}
    The functor $U\mapsto \cC(U)$ defines a sheaf of $\infty$-categories on the $locale$ $\cS(C)$.  In particular, for any $\textit{categorical locale}$ $(X, C, f\colon\cS(C)\rightarrow X)$, the assignment $U\underset{open}\subset X \mapsto C(f^{-1}(U))$ defines a sheaf of $\infty$-categories on $X$.
\end{prop}
\begin{proof}
    See \cite{complex}*{Proposition 5.3} and \cite{complex}*{Theorem 6.7}.
\end{proof}

We will proceed the proof of  enhanced GAGA theorem via identifying \textit{categorical locale} associated to scheme and its analytification.

We start to define a categorical locale out of a scheme of finite type (over $K$).
\begin{propcons}\label{construction}
For any ring $A$ and $f\in A$, we define $\Spec(A)^{\mathrm{f-an}},\Spec(A)^{f,\dagger} \in \rsh_{!}(\Aff_{K}^b)$ by the following Cartesian diagrams
$$
\begin{tikzcd}
\Spec(A)^{\mathrm{f-an}} \arrow[r] \arrow[d] & \mathbb{A}^{1,\mathrm{an}}\arrow[d] & \Anspec(A)^{f,\dagger}\arrow[r] \arrow[d] & \overdisc_{\leq 0} \arrow[d] \\
\Anspec(A) \arrow[r, "f"]   & \mathbb{A}^1    & \Anspec(A) \arrow[r]           & \mathbb{A}^1  
\end{tikzcd}
$$

Denote by $\mathcal{S}(A)$ the $locale$ of idempotent algebras in $D(A,K^+)$, then we have that 
\benuma
\item  $\mathcal{S}(A,f)\coloneq\cS(D((\Spec A)^{\mathrm{f-an}}))\subset \cS(A)$ is a open subset.
\item $\mathcal{S}(A,f)^{\dagger}\coloneq \cS(\Anspec(A)^{f,\dagger})\subset \cS(A)$ is a closed subset.
\eenum

\end{propcons}

\begin{proof}
    This follows from \Cref{analytic affine line is open}. Indeed, for (2) this corresponds to the fact that $\overdisc_{\leq 0}$ is closed in $\mathbb{A}^1$ and closed is a property that stable under base change. For (1) notice that $\Spec(A)^{\mathrm{f-an}}\simeq\underset{n}{\colim}(\Anspec A\times_{\mathbb{A}^1} \mathbb{D}_{\leq n})$ (filtered colimit commutes with finite limit), this implies that 
    \begin{align}
    D_{\sol}(\Spec(A)^{\mathrm{f-an}})&\simeq \underset{n}{\lim}D_{\sol}(\Anspec A\times_{\mathbb{A}^1} \mathbb{D}_{\leq n})\\
   &\simeq \underset{n}{\lim}D_{\sol}(\Anspec A)\tensor_{D_{\sol}(\mathbb{A}^1)}D_{\sol}(\mathbb{D}_{\leq n})\\
   &\simeq D_{\sol}(\Anspec A)\tensor_{D_{\sol}(\mathbb{A}^1)}\underset{n}{\lim}D_{\sol}(\mathbb{D}_{\leq n})\\
   &\simeq D_{\sol}(\Anspec A)\tensor_{D_{\sol}(\mathbb{A}^1)}D_{\sol}(\mathbb{A}^{1,\mathrm{an}}).
    \end{align}
  Here, in the second steo we used \Cref{affine kunneth} and in the third step we used that $D_{\sol}(\Anspec A)\simeq \mathrm{Mod}_A(D_{\sol}(\mathbb{A}^1))$ which is a dualizable category over $D_{\sol}(\mathbb{A}^1)$ thus tensor product commutes with limit. Therefore, (1) follows again from the dualizability of $D_{\sol}(\Anspec A)$ since the tensor product will then preserve localization \cite{efimov2025ktheorylocalizinginvariantslarge}*{Theorem 2.2}.
\end{proof}

The above construction allows us to perform the following.
\begin{prop}
    For any ring $A$, we have a \textit{categorical locale} $$(\Spa(A,K^+)', D_{\sol}(A,K^+), f_A\colon \cS(D_{\sol}(A,K^+))\rightarrow \Spa(A,K^+)'^{\mathrm{op}}).$$ The preimage of $\{|g|\neq 0\}$ is $\cS(A[1/g])$, and the pre-image of $\{|f|\ll|g|\}$ is $\cS(A,f/g)^{\dagger}$.
    
    Here $\Spa(A, K^+)'$ is the set of valuations $|.|\colon A\rightarrow \Gamma\bigcup \{0\}$ with values is some totally ordered abelian group $\Gamma$ (satisfying $|fg|=|f||g|$, $|f+g|\leq \max\{|f|, |g|\}$, $|0|=0$, $|K^*|=1$). We denote the inequalities in $\Gamma$ by the symbol $\ll$, and $\Spa(A,K^+)'$ is given by the topology for which the subset $\{|f|\ll|g|\}$, $|g|\neq 0$ forms a generating family of closed subsets.
\end{prop}
\brem
Note that here the underlying set of $\Spa(A,K^+)'$ is same as the usual adic space $\Spa(A, K^+)$ but the topology is different.
\erem
\brem\label{adic spcae and valuation rings}
 Recall that the underlying set of $\mathrm{Spa}(A,K^+)'$ is bijective to the set $$\{\text{valuation ring V over K}, f\colon A\rightarrow \mathrm{Frac}(V) \}/\sim$$ where the equivalent relation given by faithful flat cover of valuation rings, via the following correspondence 
 \begin{align}
v\colon A\rightarrow \Gamma\cup 0 &\mapsto f_v\colon A\rightarrow \mathrm{Frac}(A/\mathrm{supp}(v))     \\
|\cdot|_V\circ f\colon A\rightarrow \Gamma_V\cup 0 &\mapsfrom f\colon A\rightarrow \mathrm{Frac(V)}, |\cdot|_V\colon \mathrm{Frac}(V)\rightarrow \Gamma_V\cup 0. 
\end{align}
 Where $\Gamma_V$ is a valuation group of $V$ and $|\cdot|_V$ is the valuation. We will later provide a proof  of this correspondence in \Cref{appendic A} since we didn't find a clean reference.
\erem
\begin{proof}
    We only need to check that this assignment satisfies correct intersection relations, that is the axiom of valuations:
    \begin{align*}
        & (1) (\textit{pre-image of})\{|0|\ll|1|\}\  \textit{is everything}\\
        & (2) \{|f\ll|g||\}\cap\{|g|\ll|f|\}\ \textit{is empty}\\
        & (3) \{|f|\ll|h|\}\subset \{|f|\ll|g|\}\bigcup \{|g|\ll|h|\}\\
        & (4) \{|f|\ll|g+h|\}\subset \{|f|\ll|g|\}\bigcup \{|f|\ll|h|\}\\
        & (5) \{|f|\ll|g|\}\cap\{0\ll|h|\}= \{|fh|\ll|gh|\}\\
    \end{align*}
Property (1) follows from definition, since $i:0\hookrightarrow \mathbb{A}^1$ factors through overconvergent disc of radius 0.
(2)(5) follows formally from \Cref{intersect of disc empty}.

For (3)\upshape: One may assume $h$ is invertible. Moreover, after replace $f, g$ by $f/h, g/h$ one may assume that $|h|=1$. 
Now the question reduce to the case of $A=K[X,Y]$ via $\Spec A\underset{(f,g)}{\rightarrow}\mathbb{A}^{2} $.\\
Unwinding definition, we need to consider the following diagram where all the squares are Cartesian:
$$\begin{tikzcd}
D \arrow[r] \arrow[dd, bend right]                                     & \mathbb{D}_{\leq 0}^{\dagger}\times \mathbb{A}^{1} \arrow[r] \arrow[dd, bend right]      & \mathbb{D}_{\leq 0}^{\dagger} \arrow[dd] \\
\mathbb{D}_{\leq 0}^{\dagger} \arrow[r]                                & \mathbb{A}^{1}                                                                           &                                          \\
\mathbb{A}^{1}\times \mathbb{D}_{\leq 0}^{\dagger} \arrow[u] \arrow[r] & \mathbb{A}^{1}\times \mathbb{A}^1 \arrow[r, "\mathrm{Pr}_1"] \arrow[u, "\mathrm{Pr}_2"'] & \mathbb{A}^{1}                           \\
                                                                       & \mathbb{G}_{m}\times \mathbb{A}^{1} \arrow[u] \arrow[r, "{\pi\colon {(x,y)\mapsto y/x}}"]     & \mathbb{A}^1                             \\
C \arrow[uu] \arrow[r]                                                 & \pi^{-1}(\mathbb{D}_{\leq 0}^{\dagger}) \arrow[r] \arrow[u]                               & \mathbb{D}_{\leq 0}^{\dagger} \arrow[u]  
\end{tikzcd}
$$
Then (3) is equivalent to say that $C$ and $D$ forms a closed cover of $\mathbb{A}^{1}\times \mathbb{D}_{\leq 0}^{\dagger}$. By \Cref{closed cover check open closedly}, we can check this respectively on $\mathbb{D}_{\leq 0}^{\dagger}\times \mathbb{D}_{\leq 0}^{\dagger}$ and its complement. 

Notice that $D\simeq \mathbb{D}_{\leq 0}^{\dagger}\times \mathbb{D}_{\leq 0}^{\dagger}$. So over $\mathbb{D}_{\leq 0}^{\dagger}\times \mathbb{D}_{\leq 0}^{\dagger}$, $D$ and $C$ forms a closed cover. Over the complement, $C$ becomes the intersection of $\overdisc_{\geq \infty}$ and $\pi^{-1}(\mathbb{D}_{\leq 0})^{\dagger}$. Notice we have the following Cartesian diagram:
$$\begin{tikzcd}
 (\mathbb{G}_{m}\underset{(T\mapsto T^{-1})}{\times}_{ \mathbb{A}^{1}}\mathbb{A}^{1,\mathrm{an}})\times \mathbb{A}^{1}\arrow[r]  & \mathbb{A}^{1}  \\
  \overdisc_{\geq \infty}\times \mathbb{A}^{1} \arrow[r] \arrow[u]          & \overdisc_{\leq 0} \arrow[u].        
\end{tikzcd}$$
Indeed, this follows from the fact that $\mathrm{Nil}^{\dagger}$ forms an ideal, thus the morhpism $\mathbb{A}^{1,\mathrm{an}}\times \overdisc_{\leq 0}\underset{(x,y)\mapsto xy}{\rightarrow} \mathbb{A}^{1}$ factors through $\overdisc_{\leq 0}$. This shows that $C$ is precisely the complement of $\overdisc_{\leq 0}\times \overdisc_{\leq 0}$ in $\mathbb{A}^{1}\times \overdisc_{\leq 0}$.
\\

For (4)\upshape: One may assume that $g+h$ is invertible, moreover one may assume $g+h=1$. Then by \Cref{addition of disc}, $\{1\ll|g|\}$ and $\{1\ll|h|\}$ cover the whole space. Now notice that $\{|f|\ll1\}\cap \{g\gg1\}\subset \{|f|\ll|g|\}$. Indeed, by the same reason as in (3), one may assume $A=K[X,X^{-1},Y]$ and $g=X, f=Y$. Then this follows from the following diagram chase:
$$\begin{tikzcd}
\overdisc_{\leq 0} \arrow[r]       & \mathbb{A}^1                                                                                           &                              \\
                                   & \mathbb{G}_m \arrow[u, "(T\mapsto T^{-1})"]                                                            & \mathbb{A}^1                 \\
                                   & \mathbb{G}_m\times\mathbb{A}^1 \arrow[u, "\mathrm{Pr}_1"] \arrow[r, "\mathrm{Pr}_2"] \arrow[ru, "\pi"] & \mathbb{A}^1                 \\
C\times \overdisc_{\leq 0} \arrow[r] \arrow[ru,"j"] \arrow[uuu] & \mathbb{G}_m\times\overdisc_{\leq 0} \arrow[r] \arrow[u]                                               & \overdisc_{\leq 0} \arrow[u]
\end{tikzcd}$$
Here all the squares are Cartesian. Unwinding definition, (4) is amount to say that $\pi\circ j$ factors through $\overdisc_{\leq 0}$. But as the same reason in (3), $\mathbb{A}^{1,\mathrm{an}}\times \overdisc_{\leq 0}\underset{(x,y)\mapsto xy}{\rightarrow} \mathbb{A}^{1}$ already factors through $\overdisc_{\leq 0}$, and $\pi\circ j$ is a restrication of that to $C\times \overdisc_{\leq 0}$. This finishes the proof.

\end{proof}
\begin{lemma}\label{closed cover check open closedly}
    Let $\cC\in \mathrm{Pr}^{L,st}$, $C\in \cC$ be an idempotent algebra and $\mathrm{Mod}_{C}\cC\rightarrow \cC\underset{f}{\rightarrow}\cD$ be a localizing sequence. Let $A,B\in \cC$ are idempotent algebras. Then $A, B$ forms a \textit{closed cover} of $\cC$ if and only if $A\tensor C, B\tensor C$ forms a closed cover of $\mathrm{Mod}_{C}\cC$ and $f(A),f(B)$ forms a closed cover of $\cD$.
\end{lemma}
\begin{proof}
    Denote $j$ be the left adjoint of $f$, then for any $c\in \cC$ we have fiber sequence 
    \begin{equation}\label{excision}
    j\circ f(c)\rightarrow c\rightarrow c\tensor C.
    \end{equation}
    Since $A,B$ forms a closed cover of $\cC$ is equivalent to $1_{\cC}\rightarrow A\oplus B\rightarrow A\tensor B$ being a fiber sequence. By fiber sequence  \Cref{excision}, one can check this in $\mathrm{Mod}_{C}(\cC)$ and $\cD$.
\end{proof}
Our next goal is to, for any scheme construct a \textit{categorical locale}\upshape: $(X^{ad'/K^+}, D_{\sol}(X, K^+), f_X)$ via Zariski descent.

Given any finite type $K$-algebra $A$,  any Zariski open $U\subset \Spec{A}$, we can write $U=\bigcup_{i\in I}\Spec(A_{f_i})$ for some finite set $I$ of elements $f_i\in A$. Also, by construction, $\Spa(A_{f_i},K^{+})'\hookrightarrow \Spa(A, K^+)' $ is a closed inclusion \up{identified with a closed subset $|g|\neq 0$}. Thus, taking union this defines a closed subset $U^{ad'/K^+}\hookrightarrow \Spa(A, K^+)'$, which is independent of choice of covering. Indeed, one can always refine different covers by an additional affine cover, and the assertion follows from the following simple lemma.
\begin{lemma}
    For any finite type $K$-algebra A, If $f_1+ f_2...+f_n=1$ where $f_i\in A$, then $\bigcup \Spa(A_{f_i},K^+)'=\Spa(A, K^+)'$.
\end{lemma}
\begin{proof}
    Notice that $\Spa(A_{f_i},K^+)'$ is precisely $\{|f_i|\neq 0\}$. So, the lemma is equivalent to saying that for every valuation $|.|\colon A\rightarrow \Gamma\bigcup \{0\}$, there exist $i$ such that $|f_i|\neq 0$. Indeed, if contrast, by triangle inequality this says that $1=|\Sigma_i  f_i|\leq \max\{|f_i|\}=0$, leads to contradiction.
\end{proof}

\begin{defcons}
    For any finite type separated $K$-scheme $X$, there exists a topological space $X^{ad'/K^+}$\up{where when $X=\Spec A$, $X^{ad'/K^+}=\Spa(A,K^+)'$}. Defined as follow
    
    $$X^{ad'/K^+}=\mathop{\colim}_{\Spec A\underset{\mathrm{open}}{\subset}X}\Spa(A,K^+)'.$$ Here, colimit is taking in the category of topological space.

    By the above discussions, this is the gluing that binds local $\Spa(A_i,K^+)'$'s along closed subsets $(\Spec A_i\times_{X} \Spec A_j)^{ad'/K^+}$.
\end{defcons}
\brem
By \Cref{adic spcae and valuation rings}, the set of points on $X^{ad'/K^+}$ are bijective to
$$\{\text{valuation ring V}, \begin{tikzcd}
\mathrm{Spec}(\mathrm{Frac}(V)) \arrow[d] \arrow[r] & X \arrow[d]    \\
\mathrm{Spec}V \arrow[r]                          & \mathrm{Spec}K
\end{tikzcd}\}/\sim$$
Where the equivalence condition is given by faithful flat morphism $f\colon V\rightarrow W$ of valuation rings such that the following diagram commutes
$$\begin{tikzcd}
\mathrm{Spec}(\mathrm{Frac}(W)) \arrow[d] \arrow[r] & \mathrm{Spec}\mathrm{Frac}(V) \arrow[d] \arrow[r] & X \arrow[d]    \\
\mathrm{Spec}W \arrow[r, "f"]                       & \mathrm{Spec}V \arrow[r]                          & \mathrm{Spec}K.
\end{tikzcd}$$
\erem
\begin{propcons}
    For any finite type separated $K$-scheme $X$, there exist a \textit{categorical locale} $(X^{ad'/K^+}, D_{\sol}(X, K^+), f_X\colon \cS(D_{\sol}(X,K^+))\rightarrow X^{ad'/K^+})$, where $f$ defined as follow

    by Zariski descent of quasi-coherent sheaf, we have that $D_{\sol}(X, K^+)=\mathop{\mathrm{lim}}\limits_{\Spec A\subset X} D_{\sol}(A, K+)$. Now we define $f_X$ as the composition
    $$\cS(D_{\sol}(X, K^+))\rightarrow \mathop{\mathrm{lim}}\limits_{\Spec A\subset X}\cS(D_{\sol}(A, K^+))\xrightarrow{\underset{\Spec A\underset{\mathrm{open}}{\subset}X}\colim f_{A}} \mathop{\mathrm{lim}}\limits_{\Spec A\subset X} \Spa(A, K^+)'^{\mathrm{op}}= X^{ad'/K^+,op} $$

    This is a \textit{Categorical locale} by naturality of $\cS(-)$ functor and morphism of $locale $ is stable under colimit.
\end{propcons}

Now notice that there is a natural open subset $(X, X)^{ad'/K^+}$ of $X^{ad'/K^+}$. Indeed, over every affine open, there is $\Spa(A,A)'$ correspond to the subset of $\Spa(A, K^+)'$ where valuations $|.|\colon A\rightarrow \Gamma\bigcup \{0\}$ has property $|A|\leq 1$ and topology being the induced one. This is indeed an open subset by the following lemma:
\begin{lemma}\label{gluing open locale}
    Assume that the finite type $K$-algebra $A$ receives a surjection $\pi\colon K[x_1,x_2...x_n]\surjects A$, then $\Spa(A, A)'=\underset{f\in A}\bigcap \{|f| \gg 1\}^c=\underset{i=1...n}\bigcap {\{|\pi(x_i)|\gg 1\}}^c$. where $(-)^c$ denotes the complement. In particular, this shows $\Spa(A,A)'\hookrightarrow \Spa(A,K^+)'$ is an open inclusion. 
\end{lemma}
\begin{proof}
    The first equality is by definition. To prove the second equality, sufficient to prove the statement for its complement, that is, if there exist $f\in A$ such that in some valuation $|.|:A\rightarrow \Gamma\bigcup 0 $, $|f|>1$, then there exist $i$ such that $|x_i|>1$. Indeed, if $|x_i|\leq1$ for all $i$, since $\pi$ is surjective, one can always write $f=\Sigma a_{m_1m_2...m_n}x_{1}^{m_1}x_{2}^{m_2}...x_{n}^{m_n}$. Thus by triangular inequality $|f|\leq \max\{|a_{m_1m_2...m_n}x_{1}^{m_1}x_{2}^{m_2}...x_{n}^{m_n}|\}\leq \max{|a_{m_1m_2...m_n}|x_{1}^{m_1}||x_{2}^{m_2}|...|x_{n}^{m_n}|}\leq 1$
\end{proof}
Thus we can globalize this construction as follow:
\begin{propcons}\label{discrete adic space}
    For any finite type separated $K$-scheme $X$. there exist a topological space $(X,X)^{ad'/K^+}$ \up{where when $X=\Spec A$, $(X,X)^{ad'/K^+}=\Spa(A,A)'$} defined as follow\upshape:
    $$(X,X)^{ad'/K^+}=\mathop{\colim}_{\Spec A\underset{\mathrm{open}}{\subset}X}\Spa(A,A)'$$ Here the colimit is taken in the category of topological space. (By \Cref{gluing open locale} above, this is gluing local $\Spa(A_i,A_i)'$'s along the open subsets $(\Spec A_i\times_{X} \Spec A_j,\Spec A_i\times_{X} \Spec A_j)^{ad'/K^+}$.)

    Moreover the natural map $(X,X)^{ad'/K^+}\rightarrow X^{ad'/K^+}$ is an open immersion.
\end{propcons}
\begin{rem}\label{full valuation}
   Underthe correspondence in \Cref{adic spcae and valuation rings}, the set of points of $\mathrm{Spa}(A,A)'$ are bijective to 
   $$\{\text{Valuation ring V over $K$}, f\colon \mathrm{Spec V\rightarrow \mathrm{Spec}A}\}/(\sim\colon \mathrm{\pi\colon SpecW\rightarrow \mathrm{Spec}V} \text{ where $\pi$ is faithful flat } ).$$
   Hence, since morphism from specrtrum of valuation ring to a scheme always factors through an open affine subscheme, for finte type separated $K$-scheme $X$ the set of point of $(X,X)^{ad'/K^+}$ is bijective to 
   $$\{\text{Valuation ring V over $K$}, g\colon \mathrm{Spec} V\rightarrow X\}/(\sim\colon \mathrm{\pi\colon SpecW\rightarrow \mathrm{Spec}V} \text{ where $\pi$ is faithful flat } ).$$
   In particular, when $X$ is proper $K$-scheme, by valuative criterion of proper morphism, we have the bijection $|(X,X)^{ad'/K^+}|=|X^{ad'/K^+}|$. Indeed, since every map $$\begin{tikzcd}
\mathrm{Spec}(\mathrm{Frac}(V)) \arrow[d] \arrow[r] & X \arrow[d]    \\
\mathrm{Spec}V \arrow[r]                          & \mathrm{Spec}K
\end{tikzcd}$$
will then uniquely lift to a map $\mathrm{Spec} V\rightarrow X$.
\end{rem}
\begin{rem}\label{preimage of affine}
Since by construction, the pre-image of ${\{|\pi(f)|\gg1\}}^c$ under $f_A$ is precisely $\cS(A,f)$. We notice that the pre-image of $\Spa(A,A)'$ under $f_A$ is precisely $\cS(A,A)\coloneq \cS(D_{\sol}(\Spec(A)^{\mathrm{an}}))$. Thus by \Cref{descenting solid sheaf of an}, the pre-image of $(X,X)^{ad'/K^+}$ is $\cS(D_{\sol}(X^{\mathrm{an}}))$ by construction.

\end{rem}
\begin{proof}[Proof of \Cref{discrete adic space}]
    Only need to prove the statement about open immersion. By taking a finite open cover of affine schemes of $X$, one reduced to the following situation\textup{:} $X\simeq U\cup V$ where $U\simeq \mathrm{Spec}A, V\simeq \mathrm{Spec}B$ and $U\times_X V\simeq \mathrm{Spec} C$ for $K$-finite type algebras $A,B,C$. It suffices to show that $\mathrm{Spa}(A,A)^{'}\times_{X^{\mathrm{ad}/K^+}} \mathrm{Spa}(B,K^+)'\simeq \mathrm{Spa}(C,A)'\rightarrow \mathrm{Spa}(B,K^{+})'$ is open immersion.
    Furthermore, by cover both $\mathrm{Spec} C$ and $\mathrm{Spec}B$ by (finitely many) distinguished opens, one may assume that $C\simeq B[1/f]$ for some $f\in B$. 
    
    By seperatedness of $X$, we see that the natural map $A\tensor_{K} B\surjects C$ surjectively. In particular, we have that the sub-algebra $\bar{A\cup B}$ in $C$ generated by $A$ and $B$ is isomorphic to $C$. Therefore we reduce to show that the map $\mathrm{Spa}(\bar{A\cup B},A)'\rightarrow \mathrm{Spa}(B,K^+)'$ is an open immersion. The rest follows from the following lemma.
    \begin{lemma}
        Let $x\in B[1/f]$, then $\mathrm{Spa}(B[x],K[x])'\rightarrow \mathrm{Spa}(B,K^+)'$ is an open immersion.
    \end{lemma}
\begin{proof}
    Indeed, write $x=\frac{b}{g^n}$ for some $b\in B$, then $\mathrm{Spa}(B[x],K[x])'\simeq \{|f|\leq |g|^n\}\subset \mathrm{Spa}(B,K^+)'$ is an open subset.
\end{proof}
    
\end{proof}
Now we finish the proof of enhanced GAGA theorem.

\begin{thm}\label{categorical GAGA}
    If $X$ is a proper $K$-scheme, then $D_{\sol}(X^{\mathrm{an}})\simeq D_{\sol}(X^{\mathrm{alg}})$.
\end{thm}
\begin{proof}
For general finite type $K$-scheme $X$ without properness condition, we claim that the \textit{categorical locale} defines a sheaf $\cD$ of $\infty$-category on $X^{ad/'K^+}$, where $\cD((X,X)^{ad'/K^+})\simeq D_{\sol}(X^{\mathrm{an}})$.  

Indeed, the first part follows from \Cref{from locale to sheaf}. While the second is true because 1) locally it is true since $\cD(\Spa(A,A)')\simeq D_{\sol}((\Anspec A)^{\mathrm{an}})$ by \Cref{preimage of affine} and 2) notice by sheaf property of $\cD$, $\cD((X,X)^{ad'/K^+})\simeq \cD(\underset{\Spec{A}\underset{\mathrm{open}}{\subset} X}{\mathrm{colim}}(\Spa(A,A)'))\simeq \underset{\Spec{A}\underset{\mathrm{open}}{\subset}X}{\mathrm{lim}}\cD((\Spec A)^{\mathrm{an}})$, combine with descent $D_{\sol}(X^{\mathrm{an}})\simeq \underset{\Spec{A}\underset{\mathrm{open}}{\subset} X}{\mathrm{lim}}D_{\sol}((\Spec A)^{\mathrm{an}})$ thus the claim follows.\\ Now when $X$ is proper, by valuative criterion \Cref{full valuation}, the open inclusion $(X,X)^{ad'/K^+}\hookrightarrow X^{ad'/K^+}$ is actually a homeomorphism. Thus $D_{\sol}(X^{\mathrm{an}})\simeq \cD((X,X)^{ad'/K^+})\simeq \cD(X^{ad'/K+})\simeq D_{\sol}(X^{\mathrm{alg}}) $.
\end{proof}
As a immediate consequence, we show that analytification functor applying to proper schemes preserves certain ``cohomological properness'' and ``cohomological smoothness''.

\begin{prop}\label{proper is prim}
 If $X/K$ is proper scheme then $X^{\mathrm{an}}/\Spa(K, K^+)$ has the property that $1_{X^{\mathrm{an}}}$ is $prim$, If $X/K$ is of finite Tor-dimension then $1_{X^{\mathrm{an}}}$ is $suave$ object.
\end{prop}

For the definition of $suave$ and $prim$ object we refer to \cite{heyer20246functorformalismssmoothrepresentations}*{Definition 4.4.1}.
\begin{proof}
For $primness$, notice $X^{\mathrm{alg}}$ is weakly proper, since on every affine chart it is (the analytic structure is induced), and $X^{\alg}$ is glued via closed immersions. On the other hand we already have that the pullback functor via $\eta\colon X^{\mathrm{an}}\rightarrow X^{\mathrm{alg}}$ gives an equivalence of derived categories. This implies $X^{\mathrm{an}}$ is also weakly cohomological proper.\\
For $suaveness$, since it can be checked locally under open covers, one may assume $X/K$ is a finite type (over $K$) affine scheme $\Spec A$ with finite Tor-dimension. In this case, taking a surjection $K[x_1,x_2...x_n]\surjects A $ we may reduce to check individually $\pi\colon \mathbb{A}^{n,\mathrm{an}}_{\cK}\rightarrow \Spa \cK$ satisfies $1_{\mathbb{A}_{\cK}^{n,\mathrm{an}}}$ is $\pi$-\textit{suave} and $i\colon (\Spec A)^{\mathrm{an}}\rightarrow \mathbb{A}_{\cK}^{n,an}$ has the suaveness property. The first one is straightforward since $\mathbb{A}^{n,an}$ is cohomologically smooth. For the second one, by \cref{an for affine}, we have the following Cartesian diagram:
$$\begin{tikzcd}
\Spec(A)^{\mathrm{an}} \arrow[d] \arrow[r,"i"] & \mathbb{A_{\cK}}^{\mathrm{an},n} \arrow[d] \\
\Anspec (A) \arrow[r,"i^{\mathrm{alg}}"]           & \mathbb{A}_{K}^n.        
\end{tikzcd}$$
Thus sufficient to prove \textit{suaveness} for $i^{\mathrm{alg}}$. But since $i^{\mathrm{alg}}$ is proper, $i^{\alg,!}(M)\simeq \underline{\mathrm{Hom}}_{K[x_1,x_2...x_n]}(A, M).$ for any $M\in D_{\sol}(\mathbb{A}_K^n)$ By finite Tor-dimension property, one can find a finite resolution of $A$ by finite projective $K[x_1,x_2...x_n]$-modules. This shows that $i^{alg,!}(1_{\mathbb{A}_{K}^{n}})$ is dualizable object and $i^{\alg,!}(M)\simeq i*{M}\tensor i^{alg, !}(1_{\mathbb{A}^{n,an}})$, thus \textit{suaveness}.                                                                                                                                                         
\end{proof}

\csub{Categorical Künneth formula}

In algebraic geometry, given a qcqs scheme $Z$ and $X, Y$ are qcqs schemes over $Z$, then we have the categorical Künneth formula of $\infty$-categories of quasi-coherent sheaves on $X$ and $Y$, that is:
$$\mathrm{Qcoh}(X\times_{Z}Y)\simeq \mathrm{Qcoh}(X)\tensor_{\mathrm{Qcoh}(Z)}\mathrm{Qcoh}(Y)$$
Whose proof essentially reduce to \Cref{mod}. Under certain condition, the similar categorical Künneth formular also holds in analytic geometry, which for our application, will provide a more explicit categorical description of solid quasi-coherent sheaves on analytification of proper schemes (\Cref{catkunne}).

\begin{lemma}[\cite{kesting2025categoricalkunnethformulasanalytic}]\label{affine kunneth}
    For any $!$-able morphism of analytic rings $\cA\rightarrow \cB$, $\cA\rightarrow \cC$, we have that $D_{\sol}(\cB)\tensor_{D_{\sol}(\cA)}D_{\sol}(\cC)\simeq D_{\sol}(\cB\tensor_{\cA}\cC)$ \up{i.e., Künneth formula}. In particular this shows that $D_{\sol}(-)$ is a symmetric monoidal functor from $\mathrm{Corr}(\Affine_{K}^{b}, E)$ to $\mathrm{Pr}_{\mathrm{st}}^{L}$. Moreover, $D_{\sol}(\cB)$ \up{resp. $D_{\sol}(\cC)$} is \textit{dualizable} over $D_{\sol}(\cA).$
\end{lemma}
\begin{proof}
    In fact, via six functor formalism, it is sufficient to prove that $f\colon \cA\rightarrow \cB$ satisfies Künneth formula. By assumption, we may assume that $f$ is either\upshape:\\
    (1) proper, i.e., $D_{\sol}(\cB)\simeq \mathrm{Mod}_{\underline{\cB}}(D_{\sol}(\cA))$;\\
    (2) open, or equivalently, there exist an idempotent algbebra $C\in D_{\sol}(\cA)$ ($C=\mathrm{cofib}(f_{!}\underline{\cB}\rightarrow \underline{\cA})$) such that we have the following localizing sequence:
    $$\mathrm{Mod}_{C}(D_{\sol}(\cA))\rightarrow D_{\sol}(\cA)\rightarrow D_{\sol}(\cB).$$\\
    In the case of $(1)$, this is simply \Cref{mod}. In the case of $(2)$, since $D_{\sol}(\cB)$ is a retract of $D_{\sol}(\cA)$, $D_{\sol}(\cB)$ is dualizable over $D_{\sol}(\cA)$. Thus we have that
    $$\mathrm{Mod}_{C}(D_{\sol}(\cB))=D_{\sol}(\cB)\tensor_{D_{\sol}(\cA)} \mathrm{Mod}_{C}(D_{\sol}(\cA))\rightarrow D_{\sol}(\cB)\rightarrow D_{\sol}(\cB)\tensor_{D_{\sol}(\cA)}D_{\sol}(\cB)$$
    is a localizing sequence. But $C\simeq 0\in D_{\sol}(\cB)$ by construction. Thus $D_{\sol}(\cB)\simeq D_{\sol}(\cB)\tensor_{D_{\sol}(\cA)}D_{\sol}(\cB)$.\\
    But $D_{\sol}(\cB\tensor_{\cA}\cB)\simeq D_{\sol}(\cB)$ since $\cB\tensor_{\cA}\cB\simeq \cB$ by idempoteness.\\
    For dualizability, this is \cite{kesting2025categoricalkunnethformulasanalytic}*{Corollary 3.15.1}.
    
\end{proof}
\bthm[Categorical Künneth for solid module]\label{catkunne}

For any proper scheme $X$ over $\Spec{K}$ and $\cA$ a bounded affinoid analytic ring over $\cK$\upshape:
$$ D_{\sol}(X^{\mathrm{an}}\times_{\Anspec K} \Spa{\cA})\simeq \rqcoh(X)\tensor_{\rqcoh(K)}D_{\sol}(\cA)$$
\ethm
\begin{proof}
  We can cover $X^{\mathrm{an}}$ via open affinoid charts:
  
  $X^{\mathrm{an}}\times_{\Anspec{(\cK)}} \Anspec\cA=\bigcup_{i\in I} \Anspec (B_i, B_{i}^+)\times_{\Anspec\cK}\Anspec \cA.$ This forms a open $!$-cover thus we have that $ D_{\sol}(X^{\mathrm{an}}\times_{\Anspec K^{\mathrm{disc}}} \Spa{\cA})\simeq \lim_{i\in I} D_{\sol}(\Anspec ((B_{i}, B_{i}^{+})\tensor_{\cK} \cA)$. By \Cref{affine kunneth}, we have that\upshape: $D_{\sol}(\Anspec ((B_{i}, B_{i}^{+})\tensor_{\cK} \cA)\simeq D_{\sol}(B_{i},B_{i}^{+})\tensor_{D_{\sol}(\cK)}D_{\sol}(\cA)$. Now $D_{\sol}(\cA)$ is dualizable over $D_{\sol}(\cK)$ (follows from \cite{kesting2025categoricalkunnethformulasanalytic}*{Corollary 3.15.1}) and we have that
 \begin{align*} 
 D_{\sol}(X^{\mathrm{an}}\times_{\Anspec(\cK)}\Anspec(\cA))&\simeq\lim_{i\in I}(D_{\sol}(B_{i},B_{i}^{+})\tensor_{D_{\sol}(\cK)}D_{\sol}(\cA))\\&\simeq (\lim_{i\in I}D_{\sol}(B_{i},B_{i}^{+}))\tensor_{D_{\sol}(\cK)}D_{\sol}(\cA)\\&\simeq D_{\sol}(X^{\mathrm{an}})\tensor_{D_{\sol}\cK)}D_{\sol}(\cA).
\end{align*}
    On the other hand we know by \Cref{categorical GAGA} $D_{\sol}(X^{\mathrm{an}})\simeq D_{\sol}(X^{\mathrm{alg}})$. By construction of \textit{algebraic realization} we have $D_{\sol}(X^{\alg})\simeq \lim_{j\in J}D_{\sol}(A_j,K^+)$ where $X= \bigcup_{j\in J}\Spec A_j$. Since the analytic structure of $(A_j, K^+)$ is induced from $\cK$, we know that $D_{\sol}(A_j, K^+)\simeq \mathrm{Mod}_{A_j}(D_{\sol}(\cK))\simeq \rqcoh(A_j)\tensor_{\rqcoh(K)}D_{\sol}(\cK)$ (by \Cref{mod}). So we get that 
 \begin{align}\label{dualizabel and tensor commute}   D_{\sol}(X^{\mathrm{alg}})\simeq \lim_{j\in J}D_{\sol}(A_j,K^+)&\simeq \lim_{j\in J}  \rqcoh(A_j)\tensor_{\rqcoh(K)}D_{\sol}(\cK)\nonumber\\ &\simeq (\lim_{j\in J}\rqcoh(A_j))\tensor_{\rqcoh(K)}D_{\sol}(\cK)\\ &\simeq \rqcoh(X)\tensor_{\rqcoh(K)} D_{\sol}(\cK)\nonumber
 \end{align}
    Here in step \Cref{dualizabel and tensor commute} we are using that $D_{\sol}(\cK)$ is dualizable. Combining both computations together, we see that 
    \begin{align*}
        D_{\sol}(X^{\mathrm{an}}\times_{\Anspec(\cK)}\Anspec(\cA))&\simeq \rqcoh(X)\tensor_{\rqcoh(K)} D_{\sol}(\cK)\tensor_{D_{\sol}(\cK)}D_{\sol}(\cA)\\&\simeq  \rqcoh(X)\tensor_{\rqcoh(K)}D_{\sol}(\cA).
    \end{align*}
\end{proof}


\csub{Relative GAGA theorem for perfect complexes}
In this section we finish the proof of \Cref{thm1}, that is
\bthm\label{main}
Let $\cK$ be a bounded Fredholm analytic ring, $K=\underline{\cK}(*)$. $X$ is a proper scheme over $K$ with finite Tor-dimension and $\mathcal{A}$ is a Fredholm bounded affinoid algebra over $\cK$. Then we have the following equivalence of categories induced by the box tensor product\upshape:
$$\bperf(X^{\mathrm{an}})\tensor_{\bperf(K)}\bperf(\cA)\isomto \bperf(X^{\mathrm{an}}\times_{\Spa(\cK)}\Spa(\cA))$$

\ethm
\begin{proof}

    We first prove fully faithfulness\upshape: 
Consider the following diagram 
$$\begin{tikzcd}
                                   & {X_{\cA}^{\mathrm{an}}\coloneq X^{\mathrm{an}}\times_{\Spa(\cK)}\Spa(\cA)} \arrow[ld, "\pi_1"'] \arrow[rd, "\pi_2"] &                             \\
X^{\mathrm{an}} \arrow[rd, "f_2"'] &                                                                                        & \Spa(\cA) \arrow[ld, "f_1"] \\
                                   & \Spa(\cK)                                                                              &                            
\end{tikzcd}$$

    By \Cref{ffaith}, one only need to check the following\upshape: for any $M,M'\in \bperf(X^{\mathrm{an}}), N,N'\in \bperf(\cA)$, we have that
    $$\mathrm{Hom}_{D_{\sol}(X^{\mathrm{an}}\times_{\Spa \cK}\Spa \cA)}(M\boxtimes N, M'\boxtimes N')\isomto \mathrm{Hom}_{D_{\sol}(X^{\mathrm{an}})}(M,M')\tensor_{K} \mathrm{Hom}_{D_{\sol}(\cA)}(N,N')$$ in $\mathrm{Qcoh}(K).$

    But this follows straightforward from the following computation:
    
\begin{align}\label{computation kunneth}
 \operatorname{\mathrm{Hom}}_{D_{\sol}(X_{\cA}^{\mathrm{an}})}\left(\pi_1^* M \otimes \pi_2^* N, \pi_1^* M^{\prime} \otimes \pi_2^* N^{\prime}\right)
 & =\operatorname{Hom}_{D_{\sol}(X_{\cA}^{\mathrm{an}})}\left(1, \pi_1^* M^\vee \otimes \pi_2^* N^\vee \otimes \pi_1^* M^{\prime} \otimes \pi_2^* N^{\prime}\right) \nonumber \\
& =\operatorname{Hom}_{D_{\sol}(X_{\cA}^{\mathrm{an}})}\left(1, \left(\pi_1^* M^\vee \otimes \pi_1^* M^{\prime}\right) \otimes\left(\pi_2^* N^\vee \otimes \pi_2^* N^{\prime}\right)\right) \nonumber \\
& =\operatorname{Hom}_{D_{\sol}(X_{\cA}^{\mathrm{an}})}\left(1, \pi_1^*\left(M^\vee \otimes M^{\prime}\right) \otimes \pi_2^*\left(N^\vee \otimes N^{\prime}\right)\right) \nonumber \\
& =\operatorname{Hom}_{D_{\sol}(\cA)}\left(1,\pi_{2, *}  \pi_{1}^* \operatorname{\iHom}\left(M, M^{\prime}\right) \otimes \operatorname{\iHom}\left(N, N^{\prime}\right)\right) \\
& =\operatorname{Hom}_{D_{\sol}(\cA)}\left(1,f_{1}^*  f_{2, *} \operatorname{\iHom}\left(M, M^{\prime}\right) \otimes \operatorname{\iHom}\left(N, N^{\prime}\right)\right)\\ 
& = \operatorname{Hom}_{D_{\sol}(\cK)}\left(1, f_{2,*}\underline{\operatorname{Hom}}\left(M, M^{\prime}\right) \otimes f_{1,*}\underline{\operatorname{Hom}}\left(N, N^{\prime}\right)\right) \nonumber\\
& =\operatorname{Hom}\left(M, M^{\prime}\right) \tensor_{K} \operatorname{Hom}\left(N, N^{\prime}\right). \nonumber
\end{align}

Here, (\ref{computation kunneth}) used the projection formula.

Now we prove essential surjectivity:

By  \Cref{esurj} we only need to check the following:
\blem\label{conservative} If $E\in \bperf(X^{\mathrm{an}}\times_{\Spa(K, K^+)}\Spa\mathcal{A})$, such that for any $F\in \bperf(X^{\mathrm{an}}), G\in \bperf(\cA)$, we have $\mathrm{Hom}(F\boxtimes G, E)\simeq 0$, then $E\simeq 0$.
\elem
Equivalently, this amounts to show that the functor
$$\begin{aligned}
    \pmb{\mathrm{FM}}^{\mathrm{perf}}\colon \pmb{\mathrm{Perf}}(X^{\mathrm{an}}\times_{\Spa \cK} \Spa \cA)&\rightarrow \mathrm{Fun}(\pmb{\mathrm{Perf}}(X^{\mathrm{an}}), \pmb{\mathrm{Perf}}(\cA))\\
     E &\mapsto \mathrm{FM}_E\coloneq(M\mapsto \pi_{2 *}(\pi_{1}^{*}M\tensor E))
\end{aligned}$$
is conservative. Here, we used that by \Cref{sm+proper preserve dualizable}, $\mathrm{FM}_{E}$ defines a functor into $\pmb{\mathrm{Perf}}(\cA)$. Indeed, given the condition of \Cref{conservative}, we have that the Fourier transform functor induced by $E$ from $\bperf(X^{\mathrm{an}})$ to $\bperf(\cA)$ is equivalent to zero functor: 
For any $G\in \pmb{\mathrm{Perf}(\cA)}$, $\mathrm{Hom}_{\bperf(\cA)}(G, \pmb{\mathrm{FM}}_E(M))\simeq \mathrm{Hom}(\pi_{2}^*G, \pi_{1}^{*}M\tensor E)\simeq \mathrm{Hom}(M^{\vee}\boxtimes G, E)\simeq 0$ by our condition of the lemma upon taking $F=M^{\vee}$. This implies that $\pmb{\mathrm{FM}}_E(M)\simeq 0$ for any $M\in \bperf(X^{\mathrm{an}})$.

\begin{proof}[Proof of \Cref{conservative}]
 Consider the functor 
$$\begin{aligned}
    \pmb{\mathrm{FM}}\colon D_{\sol}(X^{\mathrm{an}}\times_{\Spa \cK} \Spa \cA)&\rightarrow \mathrm{Fun}_{D_{\sol}(\cK)}^{\mathrm{L}}(D_{\sol}(X^{\mathrm{an}}), D_{\sol}(\cA))\\
     F &\mapsto \mathrm{FM}_F\coloneq(M\mapsto \pi_{2 *}(\pi_{1}^{*}M\tensor E)).
\end{aligned}$$
Here, $\mathrm{FM}_F$ is indeed a colimit preserving functor by primness of $\pi_2$ (\cref{proper is prim}). By \Cref{catkunne}, we have that $D_{\sol}(X^{\mathrm{an}})\simeq \mathrm{Qcoh}(X)\tensor_{\mathrm{Qcoh(K)}}D_{\sol}(\cK).$ Thus we have that 
\begin{align*}
    \mathrm{Fun}_{D_{\sol}(\cK)}^{\mathrm{L}}(D_{\sol}(X^{\mathrm{an}}), D_{\sol}(\cA))&\simeq \mathrm{Fun}_{D_{\sol}(\cK)}^{\mathrm{L}}(\mathrm{Qcoh(X)}\tensor_{\mathrm{Qcoh}(K)}D_{\sol}(\cK),D_{\sol}(\cA)) \\ &\simeq\mathrm{Fun}_{\mathrm{Qcoh(K)}}^{\mathrm{L}}(\mathrm{Qcoh(X)},D_{\sol}(\cA))\\
    & \simeq \mathrm{Fun}_{\mathrm{Ind}(\mathrm{Perf}(K))}^{\mathrm{L}}(\mathrm{Ind}(\mathrm{Perf}(X)),D_{\sol}(\cA))\\
    &\simeq \mathrm{Fun}_{\mathrm{Perf}(K)}^{\mathrm{ex}}(\mathrm{Perf}(X),D_{\sol}(\cA)).
\end{align*}

In summary we have the the following commutative diagram: 
$$
    \begin{tikzcd}
\pmb{\mathrm{Perf}}(X^{\mathrm{an}}\times_{\Spa \cK}\Spa \cA) \arrow[d] \arrow[r,"\pmb{\mathrm{FM}}^{\mathrm{perf}}"]  & {\mathrm{Fun}_{\mathrm{Perf}(K)}^{\mathrm{ex}}(\mathrm{Perf}(X),\mathrm{Perf}(\underline{\cA})(*))} \arrow[d] \\
\mathrm{D}_{\sol}(X^{\mathrm{an}}\times_{\Spa \cK}\Spa \cA) \arrow[r,"\pmb{\mathrm{FM}}"]           & {\mathrm{Fun}_{\mathrm{Perf}(K)}^{\mathrm{ex}}(\mathrm{Perf}(X),\mathrm{D}_{\sol}(\cA))}                     
\end{tikzcd}
$$

By \Cref{fm is conservative}, $\pmb{\mathrm{FM}}$ is conservative, thus $\pmb{\mathrm{FM}}^{\mathrm{perf}}$ is also conservative, this finishes the proof.

\end{proof}

\end{proof}
\begin{lemma}\label{sm+proper preserve dualizable}
    Let $p\colon Y\rightarrow S$ be a morphism of \textit{Tate Stack} which $1_{Y}$ is $p$-$suave$ and $p$-$prim$. Then $p_*$ preserves dualizable object.
\end{lemma}
\begin{proof}
    Since $1_{Y}$ is $f$-suave with smooth dual $D_{f}(1_{Y})=f^{!} 1_{S}$. For any dualizable object $P\in D_{\sol}(Y)$, we have $P$ is also $f$-suave since pulling back preserve dualizable object (together with its dual) and sufficient to check the natural map $\pi_{1}^{*}D_{f}(P)\tensor \pi_{2}^{*}(P)\rightarrow \underline{\mathrm{Hom}}(\pi_{1}^{*}(P), \pi_{2}^{!}(P)) $ gives equivalence. But by dualizability the corresponding map is equivalent to the corresponding map for $1_Y$ then tensor with $\pi_{1}^{*}P^{\vee}\tensor\pi_{2}^* P$, which is an equivalence by definition of $1_{Y}$ being $f$-suave. Thus by \cite{heyer20246functorformalismssmoothrepresentations}*{4.4.9. Lemma (ii)} (apply to $g=p$ and $f=\mathrm{id}_S$), $p_*{P}\simeq p_{!} P$ is dualizable (equivalent to being $\mathrm{id}_S$-suave).
\end{proof}
As a final remark, Assuming $\underline{\cK}$ is a Gelfand ring, then \Cref{main} formally implies that the analogous statement for derived Berkovich spaces. 
\begin{cor}
 Let $X$ be a proper $K$-scheme, and $A$ a Gelfand ring over $\cK$, then we have that 
 $$\bperf(X^{\mathrm{Gel\text{-}an}}\times_{\mathrm{GSpec}\cK} \mathrm{GSpec} A)\simeq \mathrm{Perf}(X\times_{\Spec K} \Spec A(*)).$$
\end{cor}
\begin{proof}
    Note that \begin{align*}
        D_{\sol}((X^{\mathrm{Gel\text{-}an}}\times_{\mathrm{GSpec}\cK}\mathrm{GSpec} A)\simeq & D_{\sol}(X^{\mathrm{Gel\text{-}an}})\tensor_{D_{\sol}(\cK)}D_{\sol}(A)\\ \simeq & D_{\sol}(X^{an})\tensor_{D_{\sol}(\cK)}D_{\sol}(A)\\\simeq &D_{\sol}(X^{\mathrm{an}}\times_{\Spa \cK}\Spa A)
    \end{align*}
    Thus $\bperf(X^{\mathrm{Gel\text{-}an}}\times_{\mathrm{GSpec}}\cK \mathrm{GSpec} A)\simeq \bperf (X^{\mathrm{an}}\times_{\Spa \cK}\Spa A)$. Then it follows from \Cref{main}.
\end{proof}
\csub{Relative GAGA theorem for pseudo-coherent complexes}
In this section we deduce (a generalization of) the relative GAGA theorem for coherent sheaves in the spirit of \Cref{gaga of coh}, which is also a generalization of \Cref{main}, with the caveat of working within less general setting. Namely, let $\cK$ be a bounded affinoid algebra whose underlying condensed ring $\underline{\cK}$ is nuclear over $R_{\sol}$ and $\cK^{\dagger\text{-}\mathrm{red}}\simeq (A,A^+)_{\sol}$ for some complete analytic Huber pair $(A,A^+)$. For any $X/K$ proper scheme we will relate ``coherent sheaves'' on $X^{an,\cK}$ and coherent sheaves on $X$.

Before we proceed, let us explain the condition we put on $\cK$. The condition on $\underline{\cK}$ is to guarantee that, there is a well defined notion of \textit{nuclear module} over $\cK$ (see \Cref{cons of nuc sheaf}), the condition about $\dagger$-reduction of $\cK$ being complete analytic Huber pair is to guanrantee that $\cK$ together with bounded affinoid algebras shows up as (analytic open inside) analytification of finite type scheme over $K$ are \textit{Fredholm} (see \Cref{they are fredholm} below). The condition will be less artificial if one knows that the tensor product of Fredholm analytic rings are Fredholm. This section is again essential a copy of \cite{complex}*{Lecture XIII}, we claim no originality.

\begin{prop}\label{they are fredholm}
    let $\cK$ be a bounded affinoid algebra whose underlying condensed ring with $\cK^{\dagger\text{-}\mathrm{red}}\simeq (A,A^+)_{\sol}$ for some complete analytic Huber pair $(A,A^+)$. Then $\cK$ is Fredholm.
\end{prop}
\begin{proof}
    By \cite{andreychev2021pseudocoherentperfectcomplexesvector}*{Theorem 5.50} $\cK^{\dagger\text{-}\mathrm{red}}$ is Fredholm, thus by \cite{reallanglands}*{Proposition VI.1.7} it sufficient to show that for any $M\in D_{\sol}(\cK)$ that is static and finitely presented over $\pi_0(\cK)$, if $M\tensor_{\cK}\cK^{\mathrm{red}}\simeq 0$, then $M\simeq 0.$ This is done by \cite{anschütz2025analyticrhamstacksfarguesfontaine}*{Lemma 3.3.2} (notice that although in \cite{anschütz2025analyticrhamstacksfarguesfontaine}*{Lemma 3.3.2} they only dealt with ${\mathbb{Q}_p}_{\sol}$ bounded affinoid algebra, the exact same proof works if one replace $\mathbb{Q}_p$ by $\mathbb{Z}(\!(T)\!)$).
\end{proof}

\begin{definition}\label{pseudo-coherent}
    Let $\cC$ be a stable homotopy $\infty$-category with \textit{t-structure} $(\cC^{\leq0}, \cC^{\geq 0})$.
    We say an object $M\in \cC$ is \textit{pseudo-coherent} if 
    \benuma
    \item $M$ is bounded above
    \item For any $n\in \mathbb{Z}$ and any small diagram $p\colon I\underset{i\mapsto c_i}{\rightarrow} \cC^{\geq n}$, we have that $$\mathrm{Hom}(M, \underset{I}{\colim}c_i)\simeq \underset{I}{\colim}\mathrm{Hom}(M, c_i).$$
    \eenum
    We denote the subcategory of \textit{pseudo-coherent} objects as $\mathrm{Pcoh}(\cC)$.
\end{definition}

\begin{defprop}\label{def of pseudo}
    Let $\cX/R_{\sol}$ be a derived adic space such that for any affinoid subspace \up{under analytic topology} $U=\Spa \cO_U$, one have that $\underline{\cA}$ is nuclear over $R_{\sol}$ and $\cA$ is \textit{Fredholm}. Then let $U=\Spa (\cO_U)$ be arbitrary affinoid subspace of $
    \cX$, the association $\mathcal{P}coh_\cX:U\mapsto \mathrm{Pcoh}(\underline{\cO_U}(*))$ defines a sheaf of $\infty$-categories on $\cX$ under analytic topology. We will denote by $\mathrm{Pcoh}(\cX)$ the global section of $\mathcal{P}coh_{\cX}$ and call it the $\infty$-cateogry of \textit{pseudo-coherent complexes on $\cX$}.
\end{defprop}

The main goal of this session is to prove \Cref{def of pseudo} then deduce the following:
\begin{theorem}\label{GAGA of pseudo}
    If $X/K$ is proper scheme, then $\mathrm{Pcoh}(X^{an/\cK})\simeq \mathrm{Pcoh(X)}=\mathrm{Pcoh}(\mathrm{Qcoh}(X))$.\\

\end{theorem}
To do this, we will first give a alternative characterization of pseudo-coherent complexes of derived adic spaces in \Cref{def of pseudo}, which is more intrinsic to the analytic nature: pseudo-coherent complexes=\textit{nuclear modules} $\cap$ \textit{pseudo coherent objects in complete modules}. Intuitively one could think of nuclear module of a analytic ring $\cA$ is the ``discrete'' objects in $D_{\sol}(\cA)$, and pseudo coherent objects are roughly compact objects in $D_{\sol}(\cA)$; then this characterization is amounts to say that if a topological space is discrete and proper then it is a finite set. 

\begin{definition}
    For a \up{$D(\mathbb{Z})$-linear} compactly generated category $\cC$, an object $P\in \cC$ is called $nuclear$ if for any compact object $c\in \cC$, we have that the natrual map:
$$\mathrm{Hom}(\mathbf{1}, c^{\vee}\tensor P)\rightarrow \mathrm{Hom}(c, P)$$
is an equivalence in $D(\mathbb{Z})$.

\end{definition}
\begin{definition}
    For an analytic ring $\cA$, we call $M$ an $\cA$ nuclear module if $M\in D(\cA)$ is an nuclear object.
    Equivalently, for any totally disconnected sets $S$, we have that\upshape:
$$\cA_{\sol}[S]^{\vee}\tensor M(*)\rightarrow M(S)$$
is an equivalence. We denote category of $\cA$ nuclear module by $\mathrm{Nuc}(\cA)$. In fact, by \cite{andreychev2021pseudocoherentperfectcomplexesvector}*{Prop 5.35} there is an equivalent internal description of nuclear module\upshape:
For every profinite set $S$, one have\upshape:
$$\cA_{\sol}[S]^{\vee}\tensor M\simeq \underline{\mathrm{RHom}}(\cA_{\sol}[S], M) $$
\end{definition}
\begin{lemma}\label{nuc is mod cat}
    Let $\cA$ be an analytic rings such that for any profinite set $S$, we have that $\underline{C}(S,\underline{\cA})$ is a \textit{nuclear} $\cA$ module. Let $f:\cA\rightarrow \cB$ be a morphism of analytic rings such that $\underline{\cB}\in \mathrm{Nuc}(\cA)$, then $\mathrm{Nuc}(\cB)\simeq \mathrm{Mod}_{\underline{\cB}}(\mathrm{Nuc}(\cA)).$
\end{lemma}
\begin{proof}
    By Barr-Beck-Lurie theorem, one sufficient to prove that the forgetful functor $f_*\colon D_{\sol}(\cB)\rightarrow D_{\sol}(\cA)$ preserves nuclear objects. That is, if $M\in D_{\sol}(\cB)$ such that for any profinite set $S$, $\underline{C}(S,\underline{\cB})\tensor_\cB M\simeq \underline{C}(S, \underline{M})$, then we have that $\underline{C}(S,\underline{\cA})\tensor_\cA f_*M\simeq \underline{C}(S, \underline{M})$. Since $\underline{\cB}\in \mathrm{Nuc}(\cA)$, we see that $\underline{C}(S,\underline{\cB})\simeq \underline{C}(S,\underline{\cA})\tensor_{\cA}\underline{\cB}$ for any profinite set $S$. Thus by projection formula we have that
    $$\underline{C}(S,\underline{M})\simeq f_*(\underline{C}(S,\underline{\cB})\tensor_\cB M)\simeq f_*((\underline{C}(S,\underline{\cA})\tensor_\cA \cB)\tensor_{\cB}M)\simeq \underline{C}(S,\underline{\cA})\tensor_{\cA}M.$$
   Here we are using that by \cite{andreychev2023ktheorieadischerraume}*{Satz 3.16}, nuclear module can be realize as colimits of continuous function spaces $\underline{C}(T,A), T\in \mathrm{Profin}$. Which finishes the proof.
\end{proof}
\begin{prop}\label{discrete pseudo}
    Let $\cA$ be a \textit{Fredholm} analytic ring, then we have that $$\mathrm{Pcoh}(\underline{\cA}(*))\simeq \mathrm{Pcoh}(D_{\sol}(\cA))\cap \mathrm{Nuc}(\cA).$$
\end{prop}

\begin{proof}
    This is \cite{andreychev2021pseudocoherentperfectcomplexesvector}*{Theorem 5.50}.
\end{proof}

\begin{proof}[Proof of \Cref{def of pseudo}]
    We will proceed the proof by showing both the associations $U\mapsto \mathrm{Nuc}(\cO_U)$ and $U\mapsto \mathrm{Pcoh}(D_{\sol}(\cO_U))$ forms a sheaf of $\infty$-categories on $X$ under analytic topology.\\
    \begin{lemma}\label{cons of nuc sheaf}
       The association $\mathrm{\mathrm{Nuc_X}}
    \colon U\mapsto \mathrm{Nuc}(\cO_U)$ forms a sheaf of $\infty$-categories on $X$ for the analytic topology. 
    \end{lemma}
    \begin{proof}
        First we show that it forms a presheaf. For any morphism of analytic ring over $R_{\sol}$,  $f:\cA\rightarrow \cB$, if $\underline{\cB},\underline{\cA}\in \mathrm{Nuc}(R_{\sol})$, then we claim $(-)\tensor_\cA\cB$ preserve nuclear module. Indeed, any $M\in \mathrm{Nuc}(\cA)$ can be write as a filtered colimit of objects of form $C(S,\underline{\cA})$ for some profinite set  $S$(\cite{andreychev2023ktheorieadischerraume}*{Satz 3.16}). Thus sufficient to prove the claim in the case of $M=C(S,\underline{\cA})$. But $C(S,\underline{\cA})\tensor_{\cA} \underline{\cB}\simeq C(S,\underline{\cB})$ (\cite{andreychev2021pseudocoherentperfectcomplexesvector}*{Proposition 5.35}) is already complete with respect to $\cB$, and is a nuclear $\cB$-module, thus we proved the claim. Now presheaf property follows from the claim applying to $\cA=\cO_U$ and $\cB=\cO_V$ for some immersion of affinoid subspaces $V\subset U\subset X$.
        
        The fully faithfulness in the sheaf condition is automatic due to descent property of $D_{\sol}(\cO_U)$. We focus on showing the essential surjectiveness. That is to say, for any (finite) covering $\{U_i\}_{i\in I}$ of $U$ of affinoid subspaces and $M\in D_{\sol}(\cO_U)$, $M\in \mathrm{Nuc}(\cO_U)$ if $M_{U_i}\in\mathrm{Nuc}(\cO_{U_i})$ for every $i\in I$. Indeed, by induction one may assume $|I|=2$, thus $M\simeq \mathrm{cofib}(M_{U_1}\oplus M_{U_2}\rightarrow M_{U_1\cap U_2})[-1]$. Now that $M_{U_i}\in\mathrm{Nuc}(\cO_{U_i})$ and $\cO_{U_i}\in \mathrm{Nuc(\cO_U)}$, this implies that $M_{U_i}\in \mathrm{Nuc}(\cO_U)$ (by \Cref{nuc is mod cat}). Since nuclear modules are stable under colimit, we showed that $M_{U}\in \mathrm{Nuc}(\cO_U)$.
        
    \end{proof}
    \begin{definition}
        We define $\mathrm{Nuc}(\cX)$ the category of nuclear modules on $\cX$ to be the global section of the sheaf $\mathrm{Nuc}_X$ in \Cref{cons of nuc sheaf}.
    \end{definition}
    \begin{lemma}\label{pcoh is a sheaf}
        The association $U\mapsto \mathrm{Pcoh}(D_{\sol}(\cO_U))$ forms a sheaf of $\infty$-categories on $X$ under analytic topology.
    \end{lemma}
    \begin{proof}
        By \cite{camargo2024analyticrhamstackrigid}*{Proposition 2.7.8}, one may assume $U_i$'s are rational opens (more precisely,of the forms of pullback of rational localization of the adic spaces $\Spa(R\langle T_J\rangle,R^+)$ for some finite set $J\subset \underline{\cO_U}(*)$). We claim that the functor $(-)\tensor_{\cO_U}\cO_{U_i}$ have the property of there exist $n\in \mathbb{Z}$, such that $((-)\tensor_{\cO_U}\cO_{U_i})\colon D_{\sol}^{\geq \bullet}(U)\rightarrow D_{\sol}^{\geq \bullet+n}(U_i)$. Indeed, this boils down to the analogous statement for base change along $(\mathbb{Z}[T],Z)_{\sol}\rightarrow (\mathbb{Z}[T],\mathbb{Z}[T])_{\sol}$, which follows from \cite{andreychev2021pseudocoherentperfectcomplexesvector}*{Proposition 3.13}. Then the lemma follows from applying \Cref{descent pseudo} to $\cC=D_{\sol}(U)$ and $\cC_i=D_{\sol}(U_i)$.
    \end{proof}
  Now by \Cref{discrete pseudo}, we see that $U\mapsto \mathrm{Pcoh}(\underline{\cO_U}(*))=\mathrm{Nuc}(\cO_U)\cap \mathrm{Pcoh}(D_{\sol}(U))$ also forms a sheaf.  
\end{proof}

 \begin{lemma}\label{descent pseudo}
    Let $p\colon I\rightarrow \mathrm{Pr}^{L}_{\mathrm{st}}$ be a finite diagram of presentable $\infty$-categories $\{\cC_{i}\}_{i\in I}$ with \textit{t-structure} where the arrows between them preserve connective objects. Let $\cC=\lim_{i\in I}\cC_i$, assume $\cD\subset \cC$ is a \up{presentable stable} full subcategory with \textit{t-structure}, such that the functors $p_{i}\colon \cD\rightarrow \cC_i $ defined via projection is compatible with \textit{t-structure} and have the property that there exist $n\in \mathbb{Z}$ such that $p_i(\cD^{\geq \bullet})\subset \cC_{i}^{\geq n+\bullet}$ for every $i\in I$. Then an object $d$ in $\cD$ is pseudo-coherent if $p_{i}(d)$ is pseudo-coherent object in $\cC_i$ for all $i\in I$.
\end{lemma}

\begin{proof}
    We need to show that, for every $m\in \mathbb{Z}$ and every collection of objects $\{N_{j}\}_{j\in J}\in \cD^{\geq m}$, we have $$\mathrm{Hom}_{\cD}(d, \bigoplus_{j\in J} N_j)\simeq \bigoplus_{j\in J} \mathrm{Hom}_{\cD}(d, N_j)$$ if $p_i(d)$ is \textit{pseudo-coherent}. By our condition $p_i(\cD^{\geq \bullet})\subset \cC_{i}^{\geq n+\bullet}$ for every $i\in I$, pseudo-coherence of $p_i(d)$ implies that $\mathrm{Hom}_{\cC_i}(p_i(d), \bigoplus_{j\in J} p_i(N_j))\simeq \bigoplus_{j\in J}\mathrm{Hom}_{\cC_i}(p_i(d), p_i(N_j))$.
    By fully faithfulness of $\cD\subset \cC$, we have the following 
    \begin{align*}
        \mathrm{Hom}_{\cD}(d, \bigoplus_{j\in J} N_j)&\simeq \lim_{i\in I} \mathrm{Hom}_{\cC_i}(p_i(d), p_i(\bigoplus_{j\in J} N_j))\\
                                            &\simeq \lim_{i\in I} \mathrm{Hom}_{\cC_i}(p_i(d), \bigoplus_{j\in J} p_i(N_j))\\
                                            &\simeq \lim_{i\in I} \bigoplus_{j\in J}\mathrm{Hom}_{\cC_i}(p_i(d), p_i(N_j))\\
                                            &\simeq \bigoplus_{j\in J}\lim_{i\in I} \mathrm{Hom}_{\cC_i}(p_i(d), p_i(N_j))\\
                                            &\simeq \bigoplus_{j\in J}\mathrm{Hom}_{\cD}(d, N_j)
    \end{align*}
    The second to last identification we used that filter colimits commutes with finite limits.
\end{proof}
In order to apply \Cref{def of pseudo} to our setting, we need the following lemma.
\begin{lemma}\label{locally nuclear}
    Let $\cB$ be a analytic ring over $R_{\sol}$ such that $\underline{\cB}\in \mathrm{Nuc}(R_{\sol})$. Then any affinoid subspace $U\subset \Spa \cB$, we have that $\underline{\cO_U}\in \mathrm{Nuc}(R_{\sol})$ \up{In particular this implies that $\underline{\cO_U}\in \mathrm{Nuc}(\cB)$}. In particular, this shows for $\cK$ as in {\textup{\Cref{def of pseudo}}}, $X/K$ a finite type $K$-scheme, also affinoid subspace \up{under analytic topology} $U\in X^{\mathrm{an},\cK}$ has the property that $\underline{\cO_U}$ is \textit{nuclear} over $R_{\sol}.$
\end{lemma}
\begin{proof}
    By \cite{camargo2024analyticrhamstackrigid}*{Proposition 2.7.8}, one may assume we can cover $U$ by rational opens $\{U_i\}_{i\in I}$ of $\Spa \cB$, which by definition are of the form $\Spa \cB\tensor_{\Spa (R\langle\underline{T_{J}}\rangle,R^+)}V$ where $V=\{f_i\leq g\ | i=1,2...n\}\subset \Spa (R\langle\underline{T_{J}}\rangle,R^+)$. Rational localizations (of complete analytic Huber pairs) are nuclear (over $R_{\sol}$), hence steady \cite{analytic}*{Proposition 13.14}, which shows that $\underline{\Spa \cB\tensor_{\Spa (R\langle\underline{T_{J}}\rangle,R^+)}V}\simeq \underline{\cB}\tensor_{\Spa(D\langle \underline{T_J}\rangle,R^+)}\underline{V}$ is nuclear. Thus we show each $\underline{\cO_{U_i}}$ is nuclear, note that by descent $\underline{\cO_U}$ can be write as finite limit (also colimit up to shift) of such, which is still a nuclear module. This finishes the proof.
\end{proof}
\begin{prop}
    For any $X/K$ finite type $K$-scheme, $X^{\mathrm{an},\cK}$ is a \textit{derived Tate adic sapce} satisfying conditions in \upshape\Cref{def of pseudo}.
\end{prop}
\begin{proof}
This is a direct consequence of \Cref{locally nuclear} and \cite{reallanglands}*{Proposition VI.1.8}.   
\end{proof}
Finally we need to identify subcategory of \textit{pseudo-coherent} and \textit{nuclear} objects in $D_{\sol}(Y^{an/\cK})$ with $\mathrm{Pcoh}(Y)$. For this we need to show the following
\begin{prop}\label{coherent pushforward}
    Let $Y$ be a proper $K$-scheme, consider the map $f\colon Y^{\mathrm{an},\cK}\rightarrow \Spa \cK$. We have that $f_*$ send $\mathrm{Pcoh}(Y)$ to $\mathrm{Pcoh}(\cK)\simeq \mathrm{Pcoh}(K).$
\end{prop}
\begin{proof}
    We first show that $f_{*}$ preserve \textit{nuclear} objects. Indeed, for any $M\in \mathrm{Nuc}(Y^{\mathrm{an},\cK})$, choosing finite many open affinoid covers $\{U_i\}_{i\in I}$ of $Y^{\mathrm{an},\cK}$, one have the limit (also colimit since index is finite) diagram
    $$M\rightarrow \bigoplus_{i}j_{i,*}M_{U_i}\rightarrow\bigoplus_{ij}j_{ij,*}M_{U_ij}\rightarrow....\rightarrow j_{12...n,*}M_{U_{i_{12..n}}}.$$
    Here, each term $M_{U_{J}}$ for $J\subset I$ is a \textit{nuclear} $\cO_{U_J}$-module. Thus $f_* M$ is a colimit of \textit{nuclear} $\cK$-module (since $\underline{\cO_{U_J}}$ are \textit{nuclear} over $\cK$).  
    
    Next we show that $f_{*}$ preserve \textit{pseudo-coherent} objects. For this we need the following lemma
    \begin{lemma}\label{homological bound}
        Let $A/K$ be a finite type algebra, then for any affinoid \up{open} subspace $U\subset (\Spec A)^{\mathrm{an}}$ and the projection $p\colon U\rightarrow \Spa \cK$, we have that there exist $n\in \mathbb{Z}$ such that $p^{!}\colon D_{\sol}^{\leq \bullet}(\cK) \rightarrow D_{\sol}^{\leq \bullet+n}(U)$. Moreover, the natural map $p^!(\oplus_i M_i)\simeq \oplus_i(p^! M_i)$ is an isomorphism for $\{M_i\}_{i\in I}$ a collection of objects of $D_{\sol}(\cK)$ which is uniformly homologically bounded. 
    \end{lemma}
    \begin{proof}
       By  \Cref{an for affine}, we may assume $p$ can factorize as $j\colon U\hookrightarrow \mathbb{D}_{\cK}^{n}$ and the projection $\mathbb{D}_{\cK}^n\rightarrow \Spa \cK$, where $j$ is the pullback of a Zariski closed embedding $i:Z\rightarrow \mathbb{A}_
       {K}^{n}$:
      $$ \begin{tikzcd}
U \arrow[d, hook] \arrow[r]    & Z \arrow[d, hook]    \\
\mathbb{D}_{\cK}^{n} \arrow[r] & \mathbb{A}_{K}^{n}.
\end{tikzcd}$$
We claim that $j^!$ send $D_{\sol}^{\geq\bullet}(\mathbb{D}_{\cK}^n)$ to $D_{\sol}^{\geq \bullet}(U)$. Indeed, since $j$ is proper morphism, we have that $j_*j^!M\simeq j_*\underline{\mathrm{Hom}}(1_{U}, j^{!}M)\simeq \underline{\mathrm{Hom}}(j_*1_U, M)$ for any $M\in D_{\sol}(\mathbb{D}_{\cK}^n)$. But $j_*1$ is connective as its the base change of $i_*1_Z$ which is connective. Thus  $j^!M\in D_{\sol}^{\geq \bullet}(U)$ if $M\in D_{\sol}^{\geq \bullet}(\mathbb{D}_{\cK}^n).$  Moreover if $\{M_i\}_{i\in I}$ is a collection of objects in $D_{\sol}^{\geq \bullet}(\mathbb{D}_{\cK}^n)$, then then natural map $j^{!}(\oplus_i M_i)\rightarrow \oplus_i (p^!M_i)$ is isomorphic since left hand side can be compute as $\underline{\mathrm{Hom}}(j_*1_U, \oplus_i M_i)\simeq \oplus_i\underline{\mathrm{Hom}}(j_*1_U, M_i)\simeq \oplus_i j^!M_i$. Here, the first equality comes from $j_*1_U$ being pseudo-coherent since $i_*1_Z$ is.

The analogous property for projection map $\mathbb{D}_{\cK}^n\rightarrow \Spa \cK$ boils down to the projection map $\Spa (\mathbb{Z}[T],\mathbb{Z}[T])_{\sol}\rightarrow \Spa \mathbb{Z}_{\sol}$. Which follows from \cite{andreychev2021pseudocoherentperfectcomplexesvector}*{Proposition 3.13}.
 
       \end{proof}
    Now for any $\cF\in \mathrm{Pcoh(X^{\mathrm{an},\cK})}$, and any collection of uniformly bounded objects $\{M_i\}_{i\in I}\in D_{\sol}^{\geq \bullet}(\cK)$, We need to check that the natural map $\mathrm{Hom}(f_*\cF,\oplus_iM_i)\rightarrow \oplus_i \mathrm{Hom}(f_*\cF, M_i).$ is an isomorphism. Since $f$ is proper, $f_*\simeq f_!$, thus $\oplus_i\mathrm{Hom}(f_*\cF, M_i)\simeq \oplus_i\mathrm{Hom}(f_!\cF,M_i)\simeq \oplus_i\mathrm{Hom}(\cF, f^{!}M_i).$
    
    For suitable affinoid open (those of the form affinoid opens of analytification of affine charts on $X$) $j:\Spa \cB\subset X^{\mathrm{an},\cK}$, applying \Cref{homological bound} we have that $\oplus_i\mathrm{Hom}(j^*\cF, j^*p^!M_i)\simeq \mathrm{Hom}(\cF,\oplus_i j^*p^! M_i).$ Cover $X$ by finite many affinoid opens, by noticing finite limit commutes with arbitrary filtered colimits, we deduce that $\mathrm{Hom}(f_*\cF,\oplus_iM_i)\simeq \oplus_i \mathrm{Hom}(f_*\cF, M_i).$
\end{proof}
\begin{proof}[Proof of \Cref{GAGA of pseudo}]

In order to prove $\mathrm{Pcoh}(X^{\mathrm{an},\cK})\simeq \mathrm{Pcoh}(X)$, it is sufficient to show that any $M\in \mathrm{Pcoh}(D_{\sol} (X^{\mathrm{an},\cK}))$, and any affine open $j:\Spec A\hookrightarrow X$, the pullback $j^{*}M\in D_{\sol}(B\tensor_{K}\cK)$ is \textit{discrete}. Since $A$ is discrete, sufficient to see the cohomology (as $\cK$-module) is discrete. That is $R\Gamma(\Spa B\tensor_K \cK, j^*M)\simeq R\Gamma(X^{\mathrm{an}},j_*j^*M)$ is discrete. Up to taking blowup, one may assume $\Spec A$ is the complement of a Cartier divisor $D=X-\Spec B$. Thus we have $j_*j^*M\simeq \colim_n M\tensor \cO(-nD)$. Each term $M\tensor \cO(-nD)$ lies in $\mathrm{Pcoh}(X^{\mathrm{an},\cK})$ thus by \Cref{coherent pushforward} its cohomology lies in $\mathrm{Pcoh}(\cK)$ which is discrete. This finishes the proof.
\end{proof}

\section{Analytification of mapping stack}
Given two stacks $X$, $Y$ over $K$, one forms \textit{mapping stack} $\underline{\rmap}(X,Y)$, now when we try to consider the analytic counter part of this construction we faces two option\upshape: $\underline{\rmap}(X^{\mathrm{an}},Y^{\mathrm{an}})$, and $\underline{\rmap}(X,Y)^{\mathrm{an}}$. The first one certainly is more natural, since this really has a moduli interpretation; but in practice the later is, sometimes, easier to work with, since analytification can preserve many nice geometric property such as being ``Artin stack'' in a suitable sense. Although it's easy to see these two construction doesn't coincide in general. indeed, one can consider when $X=Y=\mathbb{A}^1$. In this chapter we will restrict to the case where $K$ being a field.

As an application of the \textit{relative GAGA} theorem, we deduce that under nice condition, \textit{analytification} of \textit{mapping stack} is isomorphic to the \textit{mapping stack} of their \textit{analytification}. More precisely, we prove the following:

\begin{thm}\label{mapp ana}
    If $X$ proper $K$-scheme, $Y$ a perfectly Tannakian \textit{Artin stack} \up{\upshape\Cref{perfectly tannakian}}, and assume that $\underline{\rmap}(X,Y)$ is an \textit{Artin stack}. Then
    $$\underline{\rmap}(X^{\mathrm{an}},Y^{\mathrm{an}})\simeq \underline{\rmap}(X,Y)^{\mathrm{an}}\in \mathrm{GelfStk}_{\cK}.$$
\end{thm}

The proof roughly, goes as follow\upshape: one would hope that the analytification functor has a rather explicit description, namely
\begin{align}\label{naive}
\underline{\rmap}(X,Y)^{\mathrm{an}}(A)\simeq \rmap(X\times \Spec{A(*)}, Y).
\end{align}

While in the mean time left hand side is just $\underline{\rmap}(X^{\mathrm{an}},Y^{\mathrm{an}})(A)\simeq \underline{\rmap}(X^{\mathrm{an}}\times \Anspec A, Y^{\mathrm{an}}).$
Then one could combine relative GAGA theorem (i.e., $\bperf(X^{\mathrm{an}}\times\Anspec A)\simeq \bperf (X\times \Spec(A(*)))$) with Lurie's Tannaka duality theorem to identify two mapping spaces with certain subcategory of 
$\rfun^{L}(\bperf(Y),\bperf(X\times \Spec(A(*))))$.

But in fact (\ref{naive}) is not necessarily true in general. Indeed. $\underline{\rmap}(X,Y)^{\mathrm{an}}$ is defined abstractly as a left Kan extension, there is a priori no control of functor of point. However, there is a variant of analytification functor, namely consider the presheaf:
$$\underline{\rmap}(X,Y)^{\mathrm{P\text{-}an}}\colon R\mapsto \rmap(X\times \Anspec (R(*)), Y)$$
Then take sheafication 
$$\underline{\rmap}(X,Y)^{\mathrm{S\text{-}an}}\coloneq(\underline{\rmap}(X,Y)^{\mathrm{P\text{-}an}})^{sh}.$$
Turns out in the case of $\underline{\rmap}(X,Y)$ is Artin stack, $\underline{\rmap}(X,Y)^{\mathrm{S\text{-}an}}\simeq \underline{\rmap}(X,Y)^{\mathrm{an}}$. Which explains the assumption in \Cref{mapp ana}.

Although this is still not enough for controling functor of point, since this still involves sheafication, in the case we are dealing with $\underline{\rmap}(X,Y)^{\mathrm{P\text{-}an}}$ is already a sheaf. Which again follows from combining relative GAGA theorem and Lurie's Tannaka duality.

\csub{Variant of analytification functor}

In general it is not clear how to explicitly describe the analytification functor in the sense that, given a stack $Y$ and analytic ring $\cR$, there is no explicit formula to compute $Y^{\mathrm{an}}(\cR)$. In this section we show that, when $Y$ is an Artin stack, then $Y^{\mathrm{an}}$, up to sheafication, admits a straightforward functor of point description.
\begin{definition}
    For any stack $Y\in \mathrm{Sh}_{\mathrm{fppf}}(\Affine_{K}^{\mathrm{op}})$, we define the \textit{simple analytification} $Y^{\mathrm{S\text{-}an}}\in \mathrm{GelfStk}_{\cK}$ as follow:
First we consider the presehaf:
$$Y^{\mathrm{P\text{-}an}}\colon R\mapsto Y(R(*)).$$
Then take sheafication:
 $$Y^{\mathrm{S\text{-}an}}\coloneq (Y^{\mathrm{P\text{-}an}})^{sh}.$$
One can do the same for stack $X\in \mathrm{Sh}_{\et}(\Affine_{K}^{\mathrm{op}})$, we denote by $X^{\mathrm{S\text{-}an}}_{\et}$. 
We will call $Y$ a \textit{$\cK$-Fredholm stack} if $Y^{P\text{-}an}$ is already a sheaf.
\end{definition}
\begin{remw}
    It is not clear how relevant of $\cK$ is in the definition of a stack being \textit{$\cK$-Fredholm stack}.
\end{remw}
\begin{prop}\label{stupid ana preserve colimit}
    The functor $(-)^{\mathrm{S\text{-}an}}_{\et}\colon \mathrm{Sh}_{\et}(\Affine_{K}^{\mathrm{op}})\rightarrow \mathrm{GelfStk}_{\cK}$ preserves colimits.
\end{prop}
This actually follows from more general facts about $\infty$-topos theory:
\begin{definition}[\cite{stacks-project}*{\href{https://stacks.math.columbia.edu/tag/00XI}{Tag 00XI}}]    
Let $\cC$ and $\cD$ be sites and $u:\cC\rightarrow \cD$ be a functor between them. Then $u$ is called \textit{co-continuous} if for every $U\in Ob(\cC)$ and every covering $\{V_{j}\rightarrow u(U)\}_{j\in J}$ of $\cD$ there exists a covering $\{U_{i}\rightarrow U\}_{i\in I}$ of $\cC$ such that the family of maps $\{u(U_i)\rightarrow u(U)\}_{i\in I}$ refines the covering $\{V_{j}\rightarrow u(U)\}_{j\in J}$.
\end{definition}
\begin{lemma}\label{cocnoti}
    Let $F\colon \cC\rightarrow \cD$ be \textit{cocontinous} functor between sites. Then the right adjoint $F_{*,p}=PSh(\cC)\rightarrow \mathrm{Psh}(\cD)$ of $F^{*,p}\colon \mathrm{Psh}(\cD)\rightarrow \mathrm{Psh}(\cC)$ \up{defined via $\cF\mapsto \cF\circ u$} preserve sheaves.
\end{lemma}
\begin{proof}
    This is essentially in [\href{https://stacks.math.columbia.edu/tag/00XI}{Tag 00XI}], although there it only deals with classical sheaf. For completeness we reproduce the proof here for the $\infty$-categorical context.
    
    Given $V'\surjects V$ a epimorphism in $\cD$, $\cG\in \mathrm{Sh}(\cC)$, we needs to show that $F_{*,p}(\cG)(V)\simeq \underset{[n]\in\Delta^{\mathrm{op}}}{\lim}F_{\*,p}(\cG)(V'^{\times n/V})$.\\
    Since $F_{*,p}$ is the right adjoint to $F^{*,p}$, we have that 
    $$F_{*,p}(\cG)(V)=\mathrm{Hom}_{\mathrm{Psh(\cD)}}(h_{V}, F_{*,p}(\cG))\simeq \mathrm{Hom}_{\mathrm{Psh}(\cC)}(F^{*,p}(h_V),\cG)\simeq \mathrm{Hom}_{\mathrm{Sh}(\cC)}(F^{*}h_V,\cG).$$
    Here, $F^*(-)$ denote the functor compose $F^{*,p}$ with sheafification. Hence, it suffices to show that $\{F^{*}(h_{V'^{\times n/V}})\}$ forms a hypercover of $F^{*}(h_V)$.
    
    We first show that the morphism $F^{*}(h_{V'})\rightarrow F^{*}(h_V)$ is an epimorphism. That is, for any $U\in \cC$, $s\in F^*{h_V}(U)$, there exist a epimorphism $U'\rightarrow U$ together with the following commutative diagram:
    $$\begin{tikzcd}
U' \arrow[r, "s'"] \arrow[d] & F^{*}(h_{V'}) \arrow[d] \\
U \arrow[r, "s"]             & F^{*}(h_V).             
\end{tikzcd}$$
Note that by definition, $s$ corresponds to a morphism $s\colon F(U)\rightarrow V$, thus by hypothesis, one can always find $U'\surjects U$ such that we have the refinement:
  $$\begin{tikzcd}
F(U') \arrow[r, "s'"] \arrow[d] & V' \arrow[d] \\
F(U) \arrow[r, "s"]             & V.             
\end{tikzcd}$$
This shows that $F^{*}(h_{V'})\rightarrow F^{*}(h_{V})$ is epimorphism. Now apply \cite{lurie2008highertopostheory}*{Lemma 6.2.3.19} (to $f_{*}=F_{*,p}$ and $f^*=F^{*,p}$) finishes the proof.
\end{proof}
\begin{lemma}
   For every bounded affinoid $R\in \AffRing^{b}_{K}$, and any $A\in \mathrm{CAlg}_{K}$, we have $$\mathrm{Map}(\Anspec R, (\Spec A)^{\mathrm{an}})=\mathrm{Map}(A, R(*)).$$
\end{lemma}
\begin{proof}
    This is essentially follows from definition. Indeed, by construction $$(\Spec A)^{\mathrm{an}}=((\Spec A)^{\mathrm{alg}})^{\mathrm{an}}=(\Anspec(A,K^+))^{\mathrm{an}}.$$
    So by definition of $()^{\mathrm{an}}$, $(\Anspec(A,K^+))^{\mathrm{an}}$ is the sheafication of the presheaf functor sending any bounded affinoid $R$ to $\rmap(\Anspec R, \Anspec(A,K^+))=\rmap(A, R(*))$, but this is already a sheaf since $\Anspec(A,K^+)$ is representable.
\end{proof}
\begin{lemma}\label{coconti nilperf}
    The following functor between site is cocontinous
  \begin{align}\label{algebrazation}
       ()_{alg}\colon & (\mathrm{GelfRing}_{\cK}, !)\rightarrow (\Affine_{/K}, \et)\\
                 & \ \ \ \ \ \ \ \ \ \ R\longmapsto R(*) 
  \end{align} 
    Here, $\mathrm{GelfRing}_{\cK}$ denote the site of \textit{Gelfand rings} over $\cK$ with  $!$-topology.
\end{lemma}
\begin{proof}
     It is enough to show that every étale cover $R(*)\rightarrow\bigcup_{i} R_i$ of $R(*)$, there exist a !-cover $R\rightarrow T$ of $R$.
     We may assume $T$ is deduced by base change from some étale cover of finite type algberas $R'\rightarrow T'$. Then this is precisely \Cref{continuous}.
\end{proof}

\begin{proof}[Proof of the \Cref{stupid ana preserve colimit}]
    Apply \Cref{cocnoti} and \Cref{coconti nilperf} to the functor between site as in (\ref{algebrazation}).

\end{proof}

Now we are ready to identify two definition of analytification functor under nice conditions.
\begin{thm}\label{S and an}
    Let $Y\in \mathrm{Sh}_{\mathrm{fppf}}(\Affine_{K}^{\mathrm{op}})$ be an Artin stack. Then we have an  isomorphism $Y^{\mathrm{S\text{-}an}}\simeq Y^{\mathrm{an}}$.
\end{thm}
\begin{proof}
    When $Y$ is affine we have that:
    $$Y^{\mathrm{S\text{-}an}}(R)=Y(R(*))=\rmap(\mathrm{GSpec} R, Y^{\mathrm{an}})=Y^{\mathrm{an}}(R).$$
    This implies $Y^{\mathrm{S\text{-}an}}\simeq Y^{\mathrm{an}}$.\\
    In general, since $Y$ is Artin stack, one can find smooth atlas $h\colon Y' \rightarrow Y$ where $Y_i$ are affine schemes. Then we have that $\underset{[n]\in\Delta}{\colim}Y^{'\times n/Y}\simeq Y \in \rsh_{\mathrm{fppf}}(\Affine_{K}^{\mathrm{op}})$. But notice that by \cite{toen2009flat}*{Théorème 0.1}, this is also true as étale stack, that is, $\underset{[n]\in\Delta}{\colim}Y^{'\times n/Y}\simeq Y \in \rsh_{\et}(\Affine_{K}^{\mathrm{op}})$. Now by \Cref{an preserve colimit} and \Cref{stupid ana preserve colimit}, both functors $()^{S-\mathrm{an}}$ and $()^{\mathrm{an}}$ from $\rsh_{\et}(\Affine_{K}^{\mathrm{op}})$ to $\mathrm{GelfStk}_{\cK}$ preserve colimit. In particular we have that 
    $$(Y)^{S-\mathrm{an}}\simeq (\underset{[n]\in\Delta}{\colim}Y^{'\times n/Y})^{S-\mathrm{an}}\simeq \underset{[n]\in\Delta}{\colim}(Y^{'\times n/Y})^{S-\mathrm{an}}\simeq \underset{[n]\in\Delta}{\colim}(Y^{'\times n/Y})^{\mathrm{an}}\simeq Y^{\mathrm{an}}.$$
\end{proof}

\csub{Tannaka duality and mapping stacks}

By the disscusion of last section, assuming $\underline{\mathrm{Map}}(X,Y)$ is an Artin stack, in order to compare $\underline{\mathrm{Map}}(X,Y)^{\mathrm{an}}$ and $\underline{\mathrm{Map}}(X^{\mathrm{an}},Y^{\mathrm{an}})$ is sufficient to compare $\underline{\mathrm{Map}}(X,Y)^{\mathrm{S\text{-}an}}$ with $\underline{\mathrm{Map}}(X^{\mathrm{an}},Y^{\mathrm{an}}).$ On the right hand side the analytic stack has rather straightforward moduli interpretation, while the left hand side involve taking sheafication, which makes the comparision difficult. In this section, using Tannaka duality, we show that under reasonable condition (\textit{perfectly Tannakian} and of \textit{geometric nature}) on $Y$, the sheafication procedure of simple analytification is redundant; furthermore, again using Tannaka duality, we deduce \Cref{mapp ana} via identifying there moduli interpretations.
\begin{definition}\label{perfectly tannakian}
    We say a stack $Y\in \rsh_{\mathrm{fppf}}(\Affine_{K}^{\mathrm{op}})$ is \textit{perfectly Tannakian} if it is a \textit{perfect stack}\footnote{Recall that a stack $Y$ is called a \textit{perfect stack} if $\rqcoh(Y)\simeq \rind(\bperf(Y))$.} and satisfies that For any derived stack $X\in \rsh_{\mathrm{fppf}}(\Affine_{K}^{\mathrm{op}})$, the functor:
    $$P\colon \rmap_{\rsh_{\mathrm{fppf}}(\AffRing_K)}(X,Y)\rightarrow \mathrm{Fun}^{\tensor}(\rqcoh(Y), \rqcoh(X))$$
    is fully faithful. And the essential image consist of exact symmetric monoidal functors $\rqcoh(Y)\rightarrow \rqcoh(X)$ that commutes with colimits and preserve connective \textit{pseudo-coherent} objects.\\

\end{definition}
 In fact the condition of \textit{perfectly Tannakian} is not too strong by the following example.
\begin{eg}\label{tannakian}
    If $Y/K$ is a Noetherian algebraic stack with quasi-affine diagonal such that $D(Y)$ is compactly generated. Then $Y$ is \textit{perfectly Tannakian} by \cite{bhatt2015tannakadualityrevisited}*{Theorem 1.3+Lemma 3.13}.
\end{eg}

\begin{definition}\label{geometric nature}
    We say a \textit{perfectly Tannakian} stack $Y$ is of \textit{geometric nature} if there exist a quasi-separated algebraic space $|Y|$ and morphism $p\colon Y\rightarrow |Y|$, such that étale locally on $|Y|$, $Y\simeq X/G$ for $X$ some qcqs scheme and $G$ a smooth affine group locally of finite type.
\end{definition}
The main goal of this section is to prove the following
\begin{thm}\label{analytification}
    If $X$ is a proper $K$-scheme, Y is an Artin \textit{perfectly Tannakian Stack} of \textit{geometric nature}. Then we have that 
    $$\underline{\rmap}(X,Y)^{S\text{-}an}\simeq \underline{\rmap}(X^{\mathrm{an}},Y^{\mathrm{an}}).$$
\end{thm}\label{good moduli}
\begin{eg}
  If $Y/K$ admits a good moduli space in the sense of \cite{Alper_2023}, then it is of geoemtric nature by \cite{alper2025etalelocalstructurealgebraic}*{Theorem 3.14}.   
\end{eg}
\brem\label{canonical map} Note that there is a  map $\delta\colon \underline{\rmap}(X,Y)^{S\text{-}an}\rightarrow\underline{\rmap}(X^{\mathrm{an}},Y^{\mathrm{an}})$. Indeed, Given any $R\in \AffRing^{b}_{K}$, we have the following sequence of maps

 $ \underline{\rmap}(X,Y)^{P\text{-}an}(R)=\rmap(X\times \Spec R(*),Y)\rightarrow \rmap(X^{\mathrm{an}}\times (\Anspec R(*))^{\mathrm{an}})\rightarrow \rmap(X^{\mathrm{an}}\times \Anspec R, Y^{\mathrm{an}})  =\underline{\rmap}(X^{\mathrm{an}},Y^{\mathrm{an}})(R).
$

Then since $\underline{\rmap}(X^{\mathrm{an}},Y^{\mathrm{an}})$ is naturally a sheaf, this map factors through $\underline{\rmap}(X,Y)^{S\text{-}an}$.
In fact in the proof we shall see that $\delta$ is already an isomorphism.
\erem

\brem At this point we can already prove some special case such that $X$ is proper locally complete intersection scheme and $Y=B\mathrm{GL}_n$. Indeed, $\underline{\rmap}(X,Y)^{P\text{-}an}(R)\simeq \rmap(X\times \Spec R(*),Y)\simeq\mathrm{Vect_n(X\times \Spec R(*))}\simeq \mathrm{Vect_n}(X^{\mathrm{an}} \times \mathrm{GSpec} R) $ by \Cref{main}(combine with \cite{reallanglands}*{Corollary VI.1.4}). But this also naturally identified with $\underline{\rmap}(X^{\mathrm{an}},Y^{\mathrm{an}})(R)$ by definition. Thus $\underline{\rmap}(X,Y)^{P\text{-}an}$ is already a sheaf and is isomorphic to $\underline{\rmap}(X^{\mathrm{an}},Y^{\mathrm{an}})$.
\erem

In particular we see by \Cref{tannakian} that Artin stacks with affine diagonal are Tannakian.
\begin{rem}
    In practice, the condition of \textit{geometric nature} will satisfy whenever the \textit{perfectly Tannakian} stack $Y$ admits a \textit{good moduli} in the sense of \cite{Alper_2023} by \cite{Alper_2023}*{proposition 4.2}. 
\end{rem}
We now apply the \textit{Tannakian} property to our question.
\begin{lemma}\label{descent of perf}
    For any $X\in \mathrm{GelfStk}_{\cK}$, the pullback functor $d\colon \bperf(X)\rightarrow \underset{\mathrm{GSpec} A\rightarrow X}{\bperf(A)}$ is fully faithful.
\end{lemma}
\begin{proof}
    Indeed, first we know that by definition, $D_{\sol}(X)\simeq \underset{\mathrm{GSpec} A\rightarrow X}{\lim}D_{\sol}(A)$. Recall the subcategory of \textit{dualizable object} of any symmetric monoidal stable $\infty$-category is always a fully faithful subcategory. Also the limit of fully faithful embedding between $\infty$-categories is still fully faithful since $\mathrm{Hom}$ in the limit of $\infty$-categories is computed as limit of $\mathrm{Hom}$ space.  Thus the pullback functor fit into the following diagram 
    $$\begin{tikzcd}
D_{\sol}(X) \arrow[r, "(1)", hook]         & \underset{\mathrm{GSpec}(A)\rightarrow X}{\lim}D_{\sol}(A)               \\
\bperf(X) \arrow[u, "(2)", hook] \arrow[r,"d"] & \underset{\mathrm{GSpec}(A)\rightarrow X}{\lim}\bperf(A). \arrow[u, "(3)", hook]
\end{tikzcd}$$
Both (1) (2) (3) are fully faithful functors. Thus $d$ is also fully faithful.
\end{proof}
\begin{lemma}\label{Tannkian for analytic}
    Assume Artin stack $Y/K$ is \textit{perfectly Tannakian}. Then for any $X\in \mathrm{GelfStk}_{\cK}$ We have\upshape:
    $$\rmap_{\cK}(X, Y^{\mathrm{an}})\hookrightarrow \mathrm{Fun}_{\mathrm{Perf}(K)}^{\tensor}(\rperf(Y), \bperf(X)).$$
    Is a fully faithful embedding of $\infty$-categories.
\end{lemma}
\begin{proof}
    By \Cref{descent of perf}, we may assume $X\simeq\mathrm{GSpec} A$ is a Gelfand affinoid space. Thus we have 
    \begin{align*}
        \rmap_{\mathrm{GelfStk}_{\cK}}(X, Y^{\mathrm{P\text{-}an}})\simeq \rmap_{\mathrm{Sh}_{\mathrm{fppf}}(\AffRing_K)}(\Spec A(*), Y)&\underset{(*)}{\hookrightarrow} \mathrm{Fun}_{\mathrm{Perf}(K)}^{\tensor}(\rperf(Y), \bperf A(*))\\ &\simeq \mathrm{Fun}_{\mathrm{Perf}(K)}^{\tensor}(\rperf(Y), \bperf(X))
    \end{align*}
  
    Here, the arrow $(*)$ is fully faithful. This shows that, the morphism of pre-sheaves:
    $$\mathrm{Map}_{\mathrm{GelfStk}}(-,Y^{\mathrm{P\text{-}an}})\rightarrow \mathrm{Fun}_{\mathrm{Perf}(K)}^{\tensor}(\mathrm{Perf(Y)},\bperf(-))$$
    is fully faithful. But by construction, $\mathrm{Map}_{\mathrm{Gelf}}(-,Y^{\mathrm{P\text{-}an}})^{\mathrm{sh}}\simeq \mathrm{Map}(-, Y^{\mathrm{S\text{-}an}}).$ Since sheafification functor is left exact, this shows that 
    $$\rmap(X, Y^{\mathrm{an}})\hookrightarrow \mathrm{Fun}_{\mathrm{Perf}(K)}^{\tensor}(\rperf(Y), \bperf(X))$$
   is fully faithful. 
    Which finishes the proof.
\end{proof}
Before we go into the general case of  \Cref{analytification}, let us deal with a basic special case of it where $Y=BG$ the classifying stack of smooth affine group scheme over $K$.
\begin{lemma}\label{quotient stack tannakian}
    Let $\cA$ be a Gelfand ring, and a map $f\colon \mathrm{GSpec} \cA\rightarrow BG^{\mathrm{an}}$. Then the induced functor 
    $$f^{alg,*}\colon \mathrm{Qcoh(BG)}\simeq\mathrm{Ind}(\rperf(BG))\rightarrow \mathrm{Ind}(\bperf(BG^{\mathrm{an}}))\underset{f^{*}}{\rightarrow} \mathrm{Ind}(\bperf(\cA))\simeq \mathrm{Qcoh}(\underline{\cA}(*)).$$
    preserve connective objects and flat objects. In particular, by \cite{lurie2005tannakadualitygeometricstacks}*{Theorem 5.11}, this shows $\mathrm{Map}(\mathrm{GSpec} \cA, BG^{\mathrm{an}})\simeq \mathrm{Map}(\Spec \underline{\cA}(*),BG).$
\end{lemma}
\begin{proof}
    Since $BG$ is smooth, the natural $t$-structure on $\mathrm{Qcoh}([*/G])$ carries to $\rperf(BG)$,every connective object of $\mathrm{Qcoh}(BG)$ can be written as a filtered colimit of perfect complex with Tor-amplitude $\leq 0$.
    
   \  We first show that $f^{\alg, *}$ preserves connective objects.
    By construction, $p:\mathrm{GSpec} \cK\rightarrow BG^{\mathrm{an}}$ is a $!$ surjection, thus one can find a $!$-cover $\pi\colon \mathrm{GSpec} \cB\rightarrow \mathrm{GSpec} \cA$ such that there is a lifting of $f\circ\pi$ to the map $q:\mathrm{GSpec} \cB\rightarrow \mathrm{GSpec} \cK$. Now given any connective object $E\in \mathrm{Qcoh}^{\leq 0}(BG)$, write it as a filtered colimit of $E_i\in \rperf^{[-i,0]}(BG)$, we see that $\pi^*\circ f^{alg,*}E_i\in \bperf^{[-i,0]}(\cB) $ since $p^*E_i$ is a finite complex of vector space concentrate in degree $[-i,0]$ and $\pi^*\circ f^{alg,*}E_i\simeq q^*(p^*E_i)$.  By \cite{reallanglands}*{Proposition VI.1.3}, Tor-amplitude of perfect complex can be checked $!$-locally, this shows $f^{alg,*}E_i\in \bperf^{[-i,0]}(\cA) $, and $f^{\alg, *}E$ is a filtered colimit of such objects thus connective, i.e., the functor $f^{alg,*}$ restrict to a functor $g\colon \mathrm{Qcoh}^{\leq 0}(BG)\rightarrow \mathrm{Qcoh}^{\leq 0}(\underline{\cA}(*))$.\\
    
    Now any flat object $M\in \mathrm{Qcoh}(BG)$ corresponds to a representation of $G$ in $K$-vector spaces, which can always be write as a filtered colimit of finite dimensional representations, thus $M\simeq \mathrm{colim}_i M_i$, where $M_i\in \mathrm{Vect}(BG)$ are vector bundles on $BG$. Similar to the above proof, we know that $f^{alg,*}$ carries vector bundles (i.e., objects in $\rperf^{[0,0]}(BG)$) to vector bundles. This shows $f^{alg,*}(M)$ is a filtered colimit of vector bundles thus flat.
\end{proof}
\begin{proof}[Proof of \Cref{analytification}]
Since $X$ is proper and $A$ is nil-perfectoid (in particular, is \textit{Fredholm}), by \Cref{main} we have that $\bperf(X^{\mathrm{an}}\times \mathrm{GSpec}(A))))\simeq \rperf(X\times \Spec(A(*)))$.
 So We have the following diagram.
    $$
    \begin{footnotesize}
\begin{tikzcd}
{\underline{\rmap}(X, Y)^{\mathrm{an}}(A)\simeq \rmap(X\times \Spec(A(*)),Y)} \arrow[rd, "\mathrm{Tann}^{\mathrm{alg}}", hook] \arrow[dd, "\delta"] &                                                                 \\
                                                                                                                                   & {\mathrm{Fun}^{\tensor}(\rperf(Y),\rperf(X\times \Spec(A(*))))}. \\
{\underline{\rmap}(X^{\mathrm{an}},Y^{\mathrm{an}})(A)\simeq \rmap(X^{\mathrm{an}}\times \mathrm{GSpec} A, Y^{\mathrm{an}})} \arrow[ru, "\mathrm{Tann}^{\mathrm{an}}",hook]                 &                                                                
\end{tikzcd}
\end{footnotesize}
$$
Here, $\delta$ is the functor defined in \Cref{canonical map}. $\mathrm{Tann}^{\mathrm{an}}$ and $\mathrm{Tann}^{\mathrm{alg}}$ are fully faithful by \Cref{Tannkian for analytic} and the assumption of $Y$ being \textit{Tannakian}.  Thus we see $\delta$ is also fully faithful. 
It remains to prove that $\delta$ is essential surjective. Which is equivalent to say, given a map $g\colon X^{\mathrm{an}}\times\mathrm{GSpec} A \rightarrow Y$, we need to check the pullback functor
$$g^{*}\colon \rqcoh(Y)\rightarrow \rind(\bperf)(Y^{\mathrm{an}})\rightarrow\rind(\bperf)(X^{\mathrm{an}}\times \mathrm{GSpec} A)(\simeq \rind(\rperf(X\times \Spec{A(*)})))$$  
preserves connective \textit{pseudo-coherent} objects. For this we need the following two lemmas.
\begin{lemma}\label{local pseudo}
    Given any map $f\colon \mathrm{GSpec} A \rightarrow Y^{\mathrm{an}}$ from some \textit{Fredholm bounded affinoid algebra } $A$. The induced composition of functors 
    $$f^{*}\colon \mathrm{\rqcoh}(Y)\rightarrow\mathrm{Ind}(\bperf) (Y^{\mathrm{an}})\rightarrow \mathrm{Ind}\bperf (A)$$
    preserve connective \textit{pseudo-coherent} objects. 
\end{lemma}
Before we enter the proof, let us introduce the following definition.
\begin{definition}\label{pcohdes}
    We call a $!$-cover of derived Berkovich space $f\colon X\rightarrow Y$ is \textit{Pcoh-detectable} if there exist a $!$-cover $f'\colon X'\rightarrow Y $ that factors through $f$, such that for any affinoid Gelfand stack $S$ and $g\colon S\rightarrow Y$, $\mathrm{Pcoh}(S)\simeq \underset{[n]\in \Delta^{op}}{\lim} \mathrm{Pcoh}({X'}_S^{\times n/S})$, where the limit taking along $f_S^*$ the base change of $f$ along $g$.
\end{definition}
\begin{lemma}\label{pcohdete}
\benuma
   \item\label{opendescoh} Rational open cover of derived Berkovich space is Pcoh-detectable. 
   \item\label{etaledescoh} Analytification of étale covers between $K$-affine schemes is Pcoh-detectable.
    \eenum
\end{lemma}
\begin{proof}
    \ref{opendescoh} is precisely \Cref{descent pseudo}. For \ref{etaledescoh}, by passing to Zariski open cover, we may assume that we are considering $f: \Spec B\rightarrow \Spec A$, where $B= A[x]/(f(x))[1/f'(x)]$ for $f(x)=x^n+a_{n-1}x^{n-1}+...a_0$ is the standard étale morphism. By \cite{bhatt2011annihilatingcohomologygroupschemes}*{Lemma 2.2}, one can see that, there is a Zariski open cover of the finite free morphism $g: \Spec C\rightarrow \Spec A$, where $C=A[x_1,x_2,..x_n]/(\sigma_i(x_1,x_2,...x_n)-(-1)^ia_i)$ (Here $\sigma_i$ is the $i$-th symmetric polynomial) that refines $f.$ Thus by \ref{opendescoh} we remains to show that $g^{\mathrm{an}}\colon (\Spec C)^{an}\rightarrow (\Spec A)^{\mathrm{an}}$ the analytification of $g$ is Pcoh-detectable. We first claim that $(\Spec C)^{\mathrm{an}}\simeq (\Spec A)^{\mathrm{an}}\times_{(\Spec A)^{\mathrm{alg}}}(\Spec C)^{\mathrm{alg}}.$ Then for any $g\colon S=\mathrm{GSpec}(D)\rightarrow (\Spec A)^{\mathrm{an}}$, the pullback of $(\Spec C)^{an}$ is $\mathrm{GSpec}(D')\coloneq\mathrm{GSpec}(D\tensor_{A} C)$, and for $f_{S}\colon \mathrm{GSpec}(D')\rightarrow \mathrm{GSpec}(D)$, we have that $f_S^{*}(-)\simeq (-)\tensor_A C$. Since $C$ is a finite free $A$ module, we see that
    $f_S^*$ preserve pseudo-coherent objects. Conversely, if $M\in D_{\sol}(D)$ and $f_S^*(M)\in \mathrm{Pcoh}(D')$, we need to show that $M\in \mathrm{Pcoh}(D)$. 

    Let $i\in\mathbb{ Z}$, and $J\rightarrow D_{\sol}^{\geq i}(D)\colon j\mapsto M_j$ be any small filtered diagram. We will show that the map 
\begin{equation}\label{pcoh}
\underset{J}{\colim} (\mathrm{Hom}_{D}(M, M_j))\rightarrow \mathrm{Hom}_{D}(M,\underset{J}{\colim}M_j)
\end{equation}
is an equivalence. Since $C$ is finte free over $A$, \Cref{pcoh} is a retract of 
$$\underset{J}{\colim}\mathrm{Hom}_{\cA}(M, M_j\tensor_{A} C)\rightarrow \mathrm{Hom}_{\cA}(M, \underset{J}{\colim}M_j\tensor_{A}C).$$
So it is enough to show the the above map is an equivalence. Which follows from the following computation:
\begin{align*}
    \underset{J}{\colim}(\mathrm{Hom}_{\cA}(M, M_j\tensor_A C)
    \simeq &\underset{J}{\colim}(\mathrm{Hom}_{D'}(M\tensor_{D}D', M_{j}\tensor_{D}D')\\
    \simeq &\mathrm{Hom}_{D'}(M\tensor_{D}D', \underset{J}{\colim}M_{j}\tensor_{D}D')\\
    \simeq & \mathrm{Hom}_{D}(M, \underset{J}{\colim}(M_j\tensor_{D}D'))\\
    \simeq & {\mathrm{Hom}_{D}(M, (\colim M_j)\tensor_{D}D')}\\
    \simeq & \mathrm{Hom}_{D}(M, (\underset{J}{\colim}M_j)\tensor_{A}C)
\end{align*}
\end{proof}
\begin{proof}[Proof of \Cref{local pseudo}]
   Notice by \cite{bhatt2014algebraizationtannakaduality}*{Theorem 1.5}, the composition $\pi\circ f$ factors through $f'\colon \Spec A(*)\rightarrow Y'$. Replace $Y'$ by étale covers we may assume $g'\colon U'\rightarrow Y'$ is an étale cover, such that $U=U'\times_{Y'}Y\simeq [X/G]$. The pullback of $g'$ gives an étale cover $g\colon \Spec A'\rightarrow \Spec A(*)$, one may assume that this étale cover is base changed from some étale cover between finite type algebras $R\rightarrow R'$. Upon analytification this provides a $!$-cover $a\colon T\rightarrow \mathrm{GSpec } A$ of $\mathrm{GSpec} A$ by \Cref{continuous}. By construction and \Cref{pcohdete} \ref{etaledescoh} $a$ is Pcoh-detectable. Now we have the following diagram:
$$\begin{tikzcd}
\mathrm{GSpec} \cT' \arrow[d, "b"] \arrow[rr, "\Tilde{f'}", dashed] &  & X \arrow[d,"q"] \arrow[r]                 & * \arrow[d] \\
\mathrm{GSpec} T \arrow[d, "a"] \arrow[rr, "\Tilde{f}"]           &  & {[X/G]} \arrow[d, "g"] \arrow[r, "p"] & {[*/G]}     \\
\mathrm{GSpec} A \arrow[d] \arrow[rr, "f"]                        &  & Y \arrow[d, "\pi"]                    &             \\
\mathrm{GSpec} A(*) \arrow[rr, "f'"]                                &  & Y'.                                    &            
\end{tikzcd}
$$
The composition $\Tilde{f}\circ p$ gives a $G$-torsor on $\mathrm{GSpec} T$, which by $Fredholm$ property, factors through $g:\Spec T(*)\rightarrow [*/G]$. Since $G$-torsor on schemes are étale locally trivial, one can find a étale cover $\pi:\Spec T'\rightarrow \Spec T(*)$ of $\Spec T(*)$ such that the $G$-torsor trivialize after pullback to $\Spec T'$, and further more we may assume $\pi$ is base changed from a etale cover of finite type algebras $\pi^{'}:\Spec W'\rightarrow \Spec W$, i.e., we have the following commutative diagram where the one on the right is a Cartesian diagram.
$$\begin{tikzcd}
\Spec T^{'} \arrow[r] \arrow[d, "\pi"] & * \arrow[d] & \Spec T^{'} \arrow[d] \arrow[r] & \Spec W' \arrow[d,"\pi^{''}"] \\
\Spec T(*) \arrow[r, "g"]               & {[*/G]}           & \Spec T(*) \arrow[r,"r"]             & \Spec W           
\end{tikzcd}$$
Refining $\pi'^{,an}$ give rise to a $!$-cover $b:\mathrm{GSpec}\cT'\rightarrow \mathrm{GSpec} T$ such that the $G$-torsor correspond to $\tilde{f}\circ p$ split after pullback to $\mathrm{GSpec} \cT'$.
Moreover, by \Cref{pcohdete} \ref{etaledescoh}, $b$ is Pcoh-detectable.

 The splitting of $G$-torsor on $\mathrm{GSpec} \cT'$ induces a $\Tilde{f'}:\mathrm{GSpec} T'\rightarrow X$ lifting $\Tilde{f}$. Since $X$ is a scheme, again by \cite{bhatt2014algebraizationtannakaduality}*{Theorem 1.5} $\Tilde{f'}$ factors through $\Spec T'\rightarrow X$, In particular $\Tilde{f'}^*$ preserves \textit{pseudo-coherent} objects. Notice that $(g\circ q)^*$ also preserves pseudo coherent complexes by Tannakian property (\Cref{tannakian}), both $a$ and $b$ are Pcoh-detectable, thus we see that $f^*$ preserves \textit{pseudo-coherent} objects. Then apply \cite{reallanglands}*{Proposition VI.1.2} we see that connectivity of \textit{pseudo-coherent} sheaf can be checked $!$-locally, thus $f^{*}$ preserve connective \textit{pseudo-coherent} objects. 

\end{proof}

Since $X$ is proper, $X^{\mathrm{an}}$ is represented by a quasi-compact rigid analytic space, so we can find a finite open affinoid cover $X^{\mathrm{an}}\simeq \bigcup_{i} \mathrm{GSpec} (U_{i})$ where $(U_i,U_{i}^{+})$ are complete analytic Huber pairs and we have $X^{\mathrm{an}}\times \mathrm{GSpec} A\simeq \bigcup_i \mathrm{GSpec}(A\tensor_{\cK}U_i)$. By \Cref{local pseudo}, we know that restrict   to each affine $\mathrm{GSpec}(A\tensor_{\cK}U_i))$, pullback functor preserve connective and pseudo-coherent object. By \cite{andreychev2021pseudocoherentperfectcomplexesvector}*{proof of Proposition 5.39} the cover $\bigcup \Anspec(A\tensor^{\sol}(U_{i},U_{i}^{+})_{\sol})\rightarrow j_*(X^{\mathrm{an}}\times \mathrm{GSpec} A)$ has finite cohomological dimension, thus by \Cref{descent pseudo}, $g^{*}$ preserves connective and pseudo-coherent objects.
\end{proof}

For practical applications, let us record a corollary of \Cref{analytification} where the conditions in the statement are easier to check.
\begin{cor}
    Let $Y/K$ be a noetherian algebraic stack with quasi-affine diagonal and reductive stablizers that admits an adequate moduli space \up{see definition in \cite{Alper_2023}}. Let $X$ be a proper scheme over $K$. Then we have that 
$$\underline{\mathrm{Map}}(X^{an},Y^{an})\simeq (\underline{\mathrm{Map}}(X,Y))^{\mathrm{an}}$$

\end{cor}
\begin{proof}
    By \cite{halpernleistner2019mappingstackscategoricalnotions}*{Theorem 5.1.1},  $\underline{\mathrm{Map}}(X,Y)$ is an algebraic stack. By \cite{Alper_2023}*{Proposition 4.2} $Y$ is of geometric nature. By \cite{bhatt2015tannakadualityrevisited}*{Theorem 1.3+Lemma 3.13} and \cite{Hall_2017}*{Theorem B} $Y$ is Perfectly Tannakian. Thus \Cref{analytification} applies.
\end{proof}

As a side product, we proved the following:

\begin{prop}\label{map from fredholm}
    Let $Y$ be a \textit{perfectly Tannakian} stack of \textit{geometric nature}. Then for any Fredholm analytic ring $\cA$, $Y^{\mathrm{an}}(\cA)\simeq Y(\underline{\cA}(*))$.
\end{prop}
\begin{proof}
    This follows from the proof of \Cref{analytification} where one take $X$ as $\Spec K$. 
\end{proof}
\begin{section}{Analytification of cotangent complex}
    As explained in the introduction, the presence of this paper is motivated by trying to control (or even to prove existence of) cotangent complex of certain nice behaved analytic stack (e.g., those analytic stack comming from analytification of Artin stacks). The main result of this section is the following:
    \begin{thm}\label{analytification of cotangent complex}
        Let $X$ be an Artin stack over $\Spec K$. Then if $X^{\mathrm{an}}$ admits a \up{relative} cotangent complex $\mathbb{L}_{X^{\mathrm{an}}/\Anspec K}$, it is the analytification of algebraic cotangent complex $\mathbb{L}_{X/K}$. More precisely, we have the following 
        $$\mathbb{L}_{X^{\mathrm{an}}/\Anspec K}\simeq j^{an, *}\mathbb{L}_{X/K}.$$
        Here, $j^{\mathrm{an}}$ is the natural map $j^{\mathrm{an}}\colon X^{\mathrm{an}}\rightarrow X^{\mathrm{alg}}$ via counit of analytification functor.
    \end{thm}
\end{section}
We will first recall the definition of cotangent complex in the context of analytic geometry.
\csub{Analytic cotangent complex}
We will essentially follow \cite{camargo2024analyticrhamstackrigid}*{Section 3.4}.
Let $\cA\in \mathrm{GelfRing}_{\cK}$ be a $\cK$-Gelfand ring. We let $D_{\sol}(\_):\mathrm{GelfStk}_{\cK}^{\mathrm{op}}\rightarrow \mathrm{CAlg}(\mathrm{Cat^{colim,ex}_{\infty}})$ denote the right Kan extension of the functor of complete modules of analytic rings.
\begin{definition}
    Let $M\in D_{\sol}(\cA) $, we say $M$ is \textit{almost connective} if $M[n]$ is connective for some $n\geq 0$. Let $\cF\colon \mathrm{GelfRing}_{\cK}^{\mathrm{op}}\rightarrow \Ani$ be a functor and $M\in D_{\sol}(\cF)$, we say that $M$ is \textit{locally almost connective} if for all analytic ring $\cA$ and all $x\in \cF(\cA)$, $x^{*}M\in D_{\sol}(A)$ is almost connective.
\end{definition}
\begin{Construction}[Square zero extension]
    Let $\cA$ be an analytic ring and $M$ a connective $\underline{\cA}$-module. We have the adjunction $\textit{Forget}\colon \AniAlg_{\underline{\cA}}\rightarrow \mathrm{Mod}_{\geq 0}(\underline{\cA})\colon \mathrm{Sym}_{\underline{\cA}}$ between animated algebras over $\underline{\cA}$ and connective $\underline{\cA}$-modules (in $D_{\cond}(\Sp)$). Given $M\in \mathrm{Mod}_{\geq 0}(\underline{\cA}) $ we can define the $\underline{\cA}$-algebra $\underline{\cA}\oplus M$ as a condensed animated ring, by forgetting degree $\geq 2$ part in $\mathrm{Sym}_{\underline{\cA}}^{\bullet}M$. This defines a analytic ring by completing the pre-analytic ring $(\cA\oplus M)^{pre}=(\underline{\cA}\oplus M, \mathrm{Mod}_{\underline{\cA}\oplus M}(D_{\sol}(\cA)))$.\\
    Given $M\in \mathrm{Mod}_{\geq}(\cA)$ we shall denote by $\cA\oplus M$ the \emph{trivial square-zero extension} of $\underline{\cA}$ by $M$ as the analytic ring defined above.\\
    \begin{definition}\label{def of cotangent}
             Let $\cF,\cG\colon \mathrm{GelfRing}_{\cK}^{\mathrm{op}}\rightarrow \Ani$ be a presheaevs, let $f\colon \cF\rightarrow \cG$ be a natural transformation between them. We say that $\cF/\cG$ admits an \emph{absolute cotangent complex} if there exists a locally almost connective complex $\cF$ such that the functor mapping a triple $(\cA, \eta, M)$ consisting of $\cA\in \AnRing$, $\eta\in \cF(\cA)$ and $M\in \mathrm{Mod_{\geq0}(\cA)}$ to the fiber product of 
    $$\begin{tikzcd}
            & \eta \arrow[d] \\
\cF(\cA\oplus M) \arrow[r] & \cF(\cA)\times_{\cG(\cA)}\cG(\cA\oplus M)          
\end{tikzcd}$$
    is correspresented by $\mathbb{L}_{\cF (\eta)}$
    \end{definition}

\end{Construction}
\begin{prop}[\cite{camargo2024analyticrhamstackrigid}*{Proposition 3.4.5}]
    Let $f:\cA\rightarrow \cB$ be a morhpism in $\AnRing$ and let $f^{'}\colon \Anspec \cB\rightarrow \Anspec \cA$ be the associated natural transformation at the level of presheaves. Then there exists a \textit{cotangent complex} $\mathbb{L_{\Anspec \cB/\Anspec \cA}}$ that is assocaited to the $\cB$-module $\mathbb{L}_{\cB/\cA}$. Further more, let $\mathbb{L}_{\underline{\cB}/\underline{\cA}}$ be (if exist) the \textit{cotangent complex} of the map of underlying condensed rings, then 
    $$\mathbb{L}_{\cB/\cA}\simeq \cB\tensor_{\underline{\cB}}\mathbb{L}_{\underline{\cB}/\underline{\cA}}.$$
    
    \end{prop}
\begin{prop}\label{fiber seq of cotan}
    Let $\cF\rightarrow\cF'\rightarrow\cF''$ be a sequence of natural transformation of functors from $\mathrm{GelfRing}_{\cK}^{\mathrm{op}}$ to $\Ani$. suppose $\mathbb{L}_{\cF'/\cF''}$ exits. Then there is an fiber sequence. Then there exist a fiber sequence 
    $$\mathbb{L}_{\cF'/\cF''}|_{\cF}\rightarrow \mathbb{L}_{\cF/\cF''}\rightarrow \mathbb{L}_{\cF/\cF'}$$
    in the sense that if either the second or the third exists, then so does the other on, and they forms the fiber sequence above.
\end{prop}
\begin{prop}\label{base change of cotan}
    Let $\cF\rightarrow \cG$ be a natural transformation of functor from $\mathrm{GelfRing}_{\cK}^{\mathrm{op}}$ to $\Ani$. Suppose that $\mathbb{L}_{\cF/\cG}$ exists, let $\cG'\rightarrow\cG$ be a natural transformation and $\cF'\simeq \cF\times_{\cG}\cG'$. Then $\mathbb{L}_{\cF/\cG}|_{\cF'}\simeq \mathbb{L}_{\cF'/\cG'}.$
\end{prop}

\begin{prop}\label{idempotent cotangent}
    Suppose $X,Y\in \mathrm{Psh}(\mathrm{GelfRing}_{\cK}^{\mathrm{op}})$ admits cotangent complexes. If $j\colon X\rightarrow Y$ is an idempotent morphism of pre-sheaves \up{i.e., $X\times_Y X\simeq X$}, then $\mathbb{L}_X\simeq j^{*}\mathbb{L}_{Y}$.
\end{prop}

\begin{proof}
    We have that $j^{*}\mathbb{L}_Y\rightarrow \mathbb{L}_X\rightarrow \mathbb{L}_{X/Y}$.
    We claim that $\mathbb{L}_{X/Y}\simeq 0$. Indeed, by idempoteness and base change property of relative cotangent complex \Cref{base change of cotan} we have that $0\simeq \mathbb{L}_{X/X}\simeq \mathbb{L}_{X/Y}$.
\end{proof}

Comparing to algebraic geometry, the criterion for Gelfand stack admiting cotangent complex is less clear. In the following we will collect some practically useful situations where the Gelfand stacks admits cotangent complexes.

\begin{prop}\label{berk has cot}
    If $X$ is a derived Berkovich space, then it admits cotangent complex.
\end{prop}
\begin{proof}
    Choose a affinoid cover $\{j_i\colon \mathrm{GSpec}(A_i)\rightarrow X\}_{i\in I}$ of $X$ ($\colim_{i\in I} \mathrm{GSpec}(A_i)\simeq X$). The collection $\{\mathbb{L}_{A_i}\}_{i\in I}$ defines a object $\mathbb{L}_X\in D_{\sol}(X)\simeq \lim_{i\in I}D_{\sol}(A_i)$ that statisfies $j_i^{*}\mathbb{L}_X\simeq \mathbb{L}_{A_i}.$ Sufficient to show that $\mathbb{L}_X$ is the cotangent complex of $X$, that is, for any Gelfand ring $A$, $\eta \colon \mathrm{GSpec}A\rightarrow X$ and $M\in D_{\sol}(A)$ we have that 
    \begin{equation}\label{berk has cot proof}
        X(A\oplus M)\times_{X(A)} \eta\simeq \mathrm{Map}(\eta^{*}\mathbb{L}_X, M).
    \end{equation}
     Since $j=\cup_{i\in I}j_i\colon \cup_{i\in I}\mathrm{GSpec}A_i\rightarrow X $ is a surjection, up to a $!$-cover $g\colon \mathrm{GSpec}B\rightarrow \mathrm{GSpec}A$, $\eta$ factor thought $\cup_{i\in I}\mathrm{GSpec}A_i$.
     $$\begin{tikzcd}
\mathrm{GSpec}B \arrow[d, "g"'] \arrow[r, "\eta'"] & \cup_{i\in I}\mathrm{GSpec}A_i \arrow[d, "j", two heads] \\
\mathrm{GSpec}A \arrow[r, "f"]                  & X                                                       
\end{tikzcd}$$
Then notice that $X(A\oplus M)\simeq \lim_{[n]\in \Delta}X(B^{\tensor n/A}\oplus g_n^*M)$ where $g_n$ is the projection $g_n\colon \mathrm{GSpec}B^{\times n/ \mathrm{GSpec A}}\rightarrow \mathrm{GSpec}A$ and $X(A)\simeq \lim_{[n]\in \Delta}X(B^{\tensor n/A})$, we see that the left hand side of (\ref{berk has cot proof}) is equivalent to $\lim_{[n]\in \Delta}X(B^{\tensor n/A}\oplus g_n^*M)\times_{X(B^{\tensor n/A})}\eta\circ g_n$. Clearly $\cup_{i\in I} \mathrm{GSpec} A_i$ admits a cotangent complex and $j$ is formally étale. This shows that for each $n$,
$$(\cup_{i\in I}\mathrm{GSpec}A_i) (B^{\tensor n/A}\oplus g_n^{*}M)\simeq X(B^{\tensor n/A}\oplus g_n^*M)\times_{X(B^{\tensor n/A})}(\cup_{i\in I}\mathrm{GSpec}A_i) (B^{\tensor n/A}).$$
Thus we have that \begin{align*}
    \lim_{[n]\in \Delta}X(B^{\tensor n/A}\oplus g_n^*M)\times_{X(B^{\tensor n/A})}f\circ g_n &\simeq  \lim_{[n]\in \Delta}X(B^{\tensor n/A}\oplus g_n^*M)\times_{X(B^{\tensor n/A})}j\circ \eta_n'\\
    &\simeq \lim_{[n]\in \Delta}(\cup_{i\in I}\mathrm{GSpec}A_i) (B^{\tensor n/A}\oplus g_n^{*}M)\times_{(\cup_{i\in I}\mathrm{GSpec}A_i) (B^{\tensor n/A})} \eta_n'\\
    &\simeq \lim_{[n]\in \Delta}\mathrm{Map}((j\circ\eta_n')^{*}\mathbb{L}_X, g_n^*M)\\
    &\simeq \lim_{[n]\in \Delta}\mathrm{Map}((g_n\circ\eta)'^{*}\mathbb{L}_X, g_n^*M) \\
    &\simeq \mathrm{Map}(\eta^* \mathbb{L}_X,M).
\end{align*}
\end{proof}
\begin{theorem}[Cotangent complex of mapping stack]\label{cot of mapp}
    Let $\cS$ be a \textit{derived Tate space}, and $f:\cX\rightarrow \cS$ be a \textit{Tate stack} such that $1_{\cX}$ is $f$-\textit{prim} and $f$-\textit{suave}. Let $\cY/\cS$ be a \textit{Tate stack} admiting \up{relative} cotangent complex $\mathbb{L}_{\cY/\cS}$. Then $\underline{\mathrm{Map}}_{\cS}(\cX,\cY)/\cS$ admit a \up{relative} cotangent complex. More precisely, consider the following diagram
$$    \begin{tikzcd}
                                    & {\underline{\mathrm{Map}}_{\cS}(\cX,\cY)\times_{\cS}\cX} \arrow[ld,"\pi"] \arrow[rd,"ev"] &   \\
{\underline{\mathrm{Map}}_{\cS}(\cX,\cY)} &                                                                      & \cY
\end{tikzcd}$$
Here, $\pi$ is the projection map and $ev$ is the tautological evaluation map. Then $\mathbb{L}_{\underline{\mathrm{Map}}_{\cS}(\cX,\cY)/\cS}\simeq \pi_{\#}ev^{*}\mathbb{L}_{\cY}.$ Here $\pi_{\#}$ is the left adjoint of $\pi^*$ (which exist since $1_{\cX}$ is $f$-\textit{suave}). 
\end{theorem}
 \begin{proof}
     Take $\cA$ (over $\cS$) a bounded affinoid ring, and $M\in D_{\sol}(\cA)$ a connective module. Denote $H\coloneq\underline{\mathrm{Map}}_{S}(X,Y) $, take $\eta\in H(\cA)$. Then the fiber product 
      $$\begin{tikzcd}
            & \eta \arrow[d] \\
H(\cA\oplus M) \arrow[r] & H(\cA)\times_{\cS(\cA)}\cS(\cA\oplus M).          
\end{tikzcd}$$
is equivalent to the fiber product
$$\begin{tikzcd}
            & f\arrow[d] \\
\mathrm{Map}_{\cS}(\cX_{\cA\oplus M},\cY) \arrow[r] & \mathrm{Map}_{\cS}(\cX_{\cA},\cY).          
\end{tikzcd}$$
Since $\cY/\cS$ admits a relative cotangent complex $\mathbb{L}_{\cY/\cS}$, this is equivalent to $\mathrm{Hom}_{D_{\sol}(\cX_{\cA})}(f^{*}\mathbb{L}_{\cY/\cS},p_{\cA}^{*}M)$, where $p_{\cA}\colon \cX_{\cA}\rightarrow \Anspec(\cA)$ is the projection.
Notice that we have the following commutative diagram\upshape: 
$$\begin{tikzcd}
\cX_{\cA} \arrow[r, "\eta\times id"] \arrow[d, "p_{\cA}"] \arrow[rr, "f", bend left] & {\underline{\mathrm{Map}}_{S}(X,Y)\times_{S}\cX} \arrow[d, "\pi"] \arrow[r, "ev"] & \cY \\
\Anspec(\cA) \arrow[r, "\eta"]                                                       & {\underline{\mathrm{Map}}_{S}(X,Y)} .                                              &    
\end{tikzcd}$$
On the left hand side it forms a Cartesian diagram. Thus we have that 
$$\mathrm{Hom}_{D_{\sol}(\cX_{\cA})}(f^{*}\mathbb{L}_{\cY/\cS},p_{\cR}^{*}M)\simeq \mathrm{Hom}_{D_{\sol}(\cA)}(p_{\cA,\#}f^{*}\mathbb{L}_{\cY/\cS},M)\simeq \mathrm{Hom}_{D_{\sol}(\cA)}(\eta^{*}(\pi_{\#}ev^{*}\mathbb{L}_{\cY/\cS}),M) $$
This shows that $\pi_{\#}ev^{*}\mathbb{L}_{\cY/\cS}$ represents the (relative) cotangent complex of $\underline{\mathrm{Map}}_{\cS}(\cX,\cY)$.
 \end{proof}
\begin{prop}
    If $Y$ is a \textit{Fredholm} stack admiting a cotangent complex, then $Y^{\mathrm{an}}\in \mathrm{GelfStk}_{\cK}$ admits a cotangent complex. In particular if $Y$ is a \textit{perfectly Tannakian stack of geometric nature}, then $Y^{\mathrm{an}}$ admits a cotangent complex.
    
\end{prop}
\begin{proof}
    By one sufficient to check that, there is an object $\mathbb{L}_{Y^{\mathrm{an}}}$ such that for any $\cA\in \mathrm{Nilperfd}$, $\eta\in Y(\cA)$ and  any connective $\cA$-module $M$, we have that the functor sending $M$ to the fiber product 
       $$\begin{tikzcd}
            & \eta \arrow[d] \\
Y^{\mathrm{an}}(\cA\oplus M) \arrow[r] & Y^{\mathrm{an}}(\cA)          
\end{tikzcd}$$
Both $\cA\oplus M$ and $\cA$ are \textit{Fredholm},thus by \Cref{map from fredholm} $Y^{\mathrm{an}}(\cA\oplus M)\simeq Y(\underline{A}(*)\oplus M(*))$ and $Y^{\mathrm{an}}(\cA)\simeq Y(\underline{\cA}(*))$, in particular $\eta$ correspond to some $\eta(*)\in Y(\underline{\cA}(*))$. Thus the fiber product above is equivalent to $\mathrm{Hom}_{\underline{\cA}(*)}((\eta(*))^{*}\mathbb{L}_{Y}, M(*))\simeq \mathrm{Hom}_{\cA}((\eta(*))^{*}\mathbb{L}_{Y}\tensor_{\underline{\cA}(*)}\cA, M)$. But $(\eta(*))^{*}\mathbb{L}_{Y}\tensor_{\underline{\cA}(*)}\cA\simeq (j_Y\circ \eta)^{*}\mathbb{L}_Y$ via the obvious commutative diagram.
$$\begin{tikzcd}
\Anspec(\cA) \arrow[d] \arrow[r, "\eta"] & Y^{\mathrm{an}} \arrow[d, "j_{Y}"] \\
\Spec (\cA(*)) \arrow[r, "\eta(*)"]      & Y                        
\end{tikzcd}$$
This shows that $\mathbb{L}_{Y^{\mathrm{an}}}\coloneq j_Y^{*}\mathbb{L}_Y$ gives the cotangent complex on $Y^{\mathrm{an}}$.
\end{proof}

\begin{proof}[Proof of \Cref{analytification of cotangent complex}]

Consider the following diagram\upshape: 
$$
\begin{tikzcd}
(U\times_X U)^{\mathrm{an}} \arrow[dd, "p"] \arrow[rr, "p_1"] \arrow[rd, "j_{U\times_X U}"] &                                                      & U^{\mathrm{an}} \arrow[dd, "f"'] \arrow[rd] &                   \\
                                                                                     & U\times_X U \arrow[dd, "\pi_1"] \arrow[rr, "\pi_2"'] &                                    & U \arrow[dd, "g"] \\
U^{\mathrm{an}} \arrow[rr, "f"] \arrow[rd, "j_U"]                                             &                                                      & X^{\mathrm{an}} \arrow[rd, "j_X"]           &                   \\
                                                                                     & U \arrow[rr, "g"]                                    &                                    & X                
\end{tikzcd}$$
since $f$ is a $!$ cover, sufficient to prove natural map $p^*f^{*}j_{X}^{*}\mathbb{L}_X\rightarrow p^*f^{*}\mathbb{L}_{X^{\mathrm{an}}}$ induces an equivalence. 
This follows from the following diagram:

$$\begin{tikzcd}
p^{*}f^*\mathbb{L}_{X^{\mathrm{an}}} \arrow[r]\arrow[dd] & p^{*}\mathbb{L}_{U^{\mathrm{an}}} \arrow[r] \arrow[d, "by\ \Cref{idempotent cotangent}"] & \mathbb{L}_{(U\times_X U)^{\mathrm{an}}/U^{\mathrm{an}}} \arrow[d, "\simeq"] \\
 & p^*j_{U}^{*}\mathbb{L}_{U} \arrow[r] \arrow[d,"\simeq"]               & j_{U\times_{X}U}^{*}\mathbb{L}_{U\times_X U/U}\arrow[d, "\simeq"]      \\ 
 p^*f^{*}j_{X}^{*}\mathbb{L}_X\arrow[r] & j_{U\times_X U}^{*}\pi_{1}^{*}\mathbb{L}_{U}\arrow[r] & j_{U\times_{X}U}^{*}\mathbb{L}_{U\times_X U/U}
\end{tikzcd}
$$
Where each row forms a fiber sequence (the last row forms a fiber sequence since we have that $p^*f^{*}j_{X}^{*}\mathbb{L}_X\simeq j_{U\times_X U}^{*}\pi_{1}^{*}g^{*}\mathbb{L}_{X}$). 
\end{proof}

 \begin{prop}
     Let $Y/K$ be a \textit{perfectly Tannakian stack with geometric nature} and $X$ be a proper scheme over $K$, then $\underline{\mathrm{Map}}(X^{\mathrm{an}},Y^{\mathrm{an}})$ admits a cotangent complex and we have that $j^* \mathbb{L}_{\underline{\mathrm{Map}}(X,Y)}\simeq \mathbb{L}_{\underline{\mathrm{Map}}(X^{\mathrm{an}},Y^{\mathrm{an}})}$. Here, $j\colon \underline{\mathrm{Map}}(X^{\mathrm{an}},Y^{\mathrm{an}})\rightarrow\underline{\mathrm{Map}}(X,Y)$ is the counit map of analytification functor.
 \end{prop}
 \begin{proof}
     This is a direct consequence of \Cref{analytification of cotangent complex} and \Cref{analytification}.
 \end{proof}
 \begin{eg}
     Let $C$ be a smooth projective curve over $\Spec K$, $G$ be a smooth affine group. By  \Cref{cot of mapp} $\mathrm{Bun}_{G}(C^{\mathrm{an}})\coloneq\underline{\mathrm{Map}}(C^{\mathrm{an}},BG^{\mathrm{an}})$ admits a cotangent complex $\mathbb{L}_{\mathrm{Bun}_G(C^{\mathrm{an}})}$, and that $\mathbb{T}^{an,*}\mathrm{Bun}_G(C^{\mathrm{an}})\simeq (\mathbb{T}^{*}\mathrm{Bun}_G(C))^{\mathrm{an}}\simeq (\mathrm{Higgs}_G(C))^{\mathrm{an}}\simeq \mathrm{Higgs}_{G}(C^{\mathrm{an}})$.
 \end{eg}

\csub{Virtual cotangent complex}
At the end, for further applications, we discuss a variant of analytic cotangent complex. 
\begin{definition}[Virtual cotangent complex]
    Let $\cF\colon \AnRing^{\mathrm{op}}\rightarrow \Ani$ be a presheaf. We denote the \emph{virtual cotangent complex} $\mathbb{L}^{vir}_{\cF}$  as the following object in $D_{\sol}(\cF)$
    $$\mathbb{L}^{vir}_{\cF}\coloneq \underset{\cA, \eta\in \cF(\cA)}{\lim} \eta_{*}\mathbb{L}_{\cA}$$
    If $f\colon \cF\rightarrow\cF'$ is a natural transform of presheaves, we denote the \emph{virtual relative cotangent complex} $\mathbb{L}^{vir}_{\cF/\cF'}$ as follow
    $$\mathbb{L}^{vir}_{\cF/\cF'}\coloneq\underset{\cA, \eta\in \cF'(\cA)}{\lim}\eta_{f,*}\mathbb{L}_{\cF\stackrel{\eta}{\times}_{\cF'}\Anspec \cA/\Anspec \cA}$$
    where $\eta_f$ denote the base change of $\eta$ along $f$.\\
    For $f\colon \cF\rightarrow \cF'$, there is a natural map $f^{*}\mathbb{L}^{vir}_{\cF'}\rightarrow \mathbb{L}^{vir}_{\cF}$ via the following composition:
    $$f^{*}\mathbb{L}^{vir}_{\cF'}=f^{*}\underset{\cA, \eta\in \cF'(\cA)}{\lim} \eta_{*}\mathbb{L}_{\cA}\rightarrow\underset{\cA, \eta\in \cF'(\cA)}{\lim} f^{*}\eta_{*}\mathbb{L}_{\cA}\rightarrow \underset{\cB, \eta'\in \cF\stackrel{\eta}{\times}_{\cF}\Anspec \cA(\cB)}{\lim} (\eta_f\circ\eta')_{*}\mathbb{L}_{\cB}\rightarrow \underset{\cB, \eta''\in \cF(\cB)}{\lim} \eta_{*}\mathbb{L}_{\cB}$$
    Where for the last term we have $\underset{\cB, \eta''\in \cF(\cB)}{\lim} \eta_{*}\mathbb{L}_{\cB}=\mathbb{L}^{vir}_{\cF}$, and $\eta_f$ denote the base change of $\eta$ along $f$.
\brem It is clear from definition that for analytic rings $\cA\in \AnRing$, $\mathbb{L}^{vir}_{\Anspec \cA}\simeq \mathbb{L}_{\cA}$. 
\erem
\end{definition}
\begin{prop}\label{computing vir cot}
\benuma
 
\item If $X$ is a derived Berkovich space, then $\mathbb{L}_X^{\mathrm{vir}}\simeq \mathbb{L}_X$.
   
\item If $X$ is a derived Berkovich space and $f\colon X\rightarrow Y$ is morphism of Gelfand stacks such that $X\times_Y X$ is representable by a smooth derived Berkovich space over $X$, then we have the following fiber sequence 
    $f^{*}\mathbb{L}^{vir}_Y\rightarrow \mathbb{L}_X\rightarrow \mathbb{L}^{vir}_{X/Y}.$ Moreover, we have a natural equivalence $\mathbb{L}_{Y^{\mathrm{an}}}^{vir}\simeq \underset{[n]\in \Delta^{\mathrm{op}}}{\lim} \mathbb{L}_{X^{\times n/Y}}$.
    \eenum
\end{prop}
\begin{proof} (1) follows from the proof of \Cref{berk has cot}, We now prove (2).
    Notice that by \Cref{descent of cotangent complex}, the assignment $\mathbb{L}\coloneq\cA\mapsto (D_{\sol}(\cA), \mathbb{L}_{\cA})$ defines a $!$-sheaf values in pointed $\infty$-category, thus by construction we have that $(D_{\sol}(X),\mathbb{L}_X^{\mathrm{vir}})\simeq \mathbb{L}(X)$ for any Gelfand stack, moreover $\mathbb{L}$ sends colimit to limit. Thus to compute $\mathbb{L}_X^{\mathrm{vir}}$ for a Gelfand stack, one sufficient to find one diagram of affinoid $\{\mathrm{GSpec}(A_i)\}_{i\in I}$ such that $X\simeq \colim_{i\in I} \mathrm{GSpec}(A_i)$, then taking the limit of corresponding pushforwards of $\mathbb{L}_{A_i}.$ Thus by \Cref{idempotent cotangent} the first one follows, since one can always write $X$ as union of affinoid. For the second one, we first find an affinoid cover $\eta_{i\in I}\colon \{\mathrm{GSpec}(B_i)\}_{i\in I}\surjects X$, and denote by $\{\mathrm{GSpec(B)^{[n]}}\}_{[n]\in\Delta^{op}}$ its Čech nerve, then by the discussion above, we see that 
    $$ f^{*}\mathbb{L}^{vir}_{Y}=f^{*}\underset{[n]\in \Delta^{\mathrm{op}}}{\lim} f\circ\eta_{*}^{[n]}\mathbb{L}_{\mathrm{GSpec}(B)^{[n]}}\simeq\underset{[n]\in \Delta^{\mathrm{op}}}{\lim}f^* f\circ\eta_{*}^{[n]}\mathbb{L}_{\mathrm{GSpec}(B)^{[n]}}\simeq \underset{[n]\in \Delta^{\mathrm{op}}}{\lim} h_{*}^{[n]}g^*\mathbb{L}_{\mathrm{GSpec}(B)^{[n]}}. $$
      Here, $h^{[n]}$ and $g$ denote the arrows in the following Cartesian diagram. Here $X_{B}^{[n]}=X\times_{Y} \mathrm{GSpec}(B)^{[n]}$ and $X_{B}^{[n]}$ are derived Berkovich spaces.
  \begin{equation}
     \begin{tikzcd}
X_{B}^{[n]} \arrow[d, "h^{[n]}"] \arrow[r, "g"] & \mathrm{GSpec}(B)^{[n]} \arrow[d, "f\circ\eta^{[n]}"] \\
X \arrow[r, "f"]                     & Y                            
\end{tikzcd}
  \end{equation}
  We now claim that $\mathbb{L}_X^{\mathrm{vir}}\simeq\mathbb{L}_X\simeq \underset{[n]\in \Delta^{\mathrm{op}}}{\lim} h_{*}^{[n]}\mathbb{L}_{X_{B}^{[n]}}.$ Indeed, we notice that by construction $h^{[n]}$ factors through $X\times_Y X$ and $p_1\colon X\times_Y X\rightarrow X$ is a rigid smooth cover of Berkovich spaces, on which the cotangent complex descends according to the proof of \cite{anschütz2025analyticrhamstacksfarguesfontaine}*{proposition 5.2.1}. Thus the claim follows from both $p_1$ and the base change of $\eta^{[n]}$ descends the cotangent complex (which is a affinoid cover).
At the end, we can also compute $\mathbb{L}_{X/Y}^{\mathrm{vir}}$ via the chosen cover $\eta_{i\in I}$, simlimar argument shows that $\mathbb{L}_{X/Y}^{\mathrm{vir}}\simeq \underset{[n]\in\Delta^{\mathrm{op}}}{\lim}h_{*}^{[n]}\mathbb{L}_{X_{B}^{[n]}/\mathrm{GSpec}(B)^{[n]}}.$ where the desired fiber sequence then follows from the fiber sequence 
$$\underset{[n]\in\Delta^{\mathrm{op}}}{\lim}h_{*}^{[n]}g^*\mathbb{L}_{\mathrm{GSpec}(B)^{[n]}}\rightarrow \underset{[n]\in \Delta^{\mathrm{op}}}{\lim} h_{*}^{[n]}\mathbb{L}_{X_{B}^{[n]}}\rightarrow \underset{[n]\in\Delta^{\mathrm{op}}}{\lim}h_{*}^{[n]}\mathbb{L}_{X_{B}^{[n]}/\mathrm{GSpec}(B)^{[n]}}$$
induced by the fiber sequences $g^*\mathbb{L}_{\mathrm{GSpec}(B)^{[n]}}\rightarrow\mathbb{L}_{X_{B}^{[n]}}\rightarrow \mathbb{L}_{X_{B}^{[n]}/\mathrm{GSpec}(B)^{[n]}}.$
\end{proof}

\begin{lemma}\label{descent of cotangent complex}
    If $f\colon \mathrm{GSpec}\cB\rightarrow \mathrm{GSpec} \cA$ is a descendable cover of Gelfand rings, then we have the following
    $$\mathbb{L}_{\cA}\simeq \underset{[n]\in \Delta^{\mathrm{op}}}{\lim} f_{n,*}\mathbb{L}_{\cB^{\otimes n/\cA}}$$
    As a consequence, we have that $\mathbb{L}_{\cB/\cA}\simeq\underset{[n]\in \Delta^{\mathrm{op}}}{\lim} \mathbb{L}_{\cB^{\otimes n+1/\cA}/\cB^{\otimes n/\cA}}.$ The last assersion also follows from $\mathbb{L}$ being a sheaf.
\end{lemma}
\begin{proof}
   
  From \Cref{fiber seq of cotan} and \cite{camargo2024analyticrhamstackrigid}*{Remark 3.4.6} we have the following fiber sequence.
    $$\mathbb{L}_{\cA}\tensor_{\cA}\cB\rightarrow\mathbb{L}_{\cB}\rightarrow \mathbb{L}_{\cB/\cA}$$
    By descent, this reduce us to show that $\underset{[n]\in \Delta^{\mathrm{op}}}{\lim} \mathbb{L}_{\cB^{\tensor n/\cA}/\cA}\simeq 0$. We first claim that $$\mathbb{L}_{\cB^{\tensor n/\cA}/\cA}\simeq \underset{i=1,2,...n}{\bigoplus}(\mathbb{L}_{\cB/\cA}\tensor_\cB \cB^{\tensor n/\cA})$$ where each component correspond to base change along difference projections. Indeed, this follows from universal property of cotangent complex: for any $\cB^{\tensor n/\cA}$-algebra $\eta\colon \cB\rightarrow\cC$, and an almost connective module $M\in \mathrm{Mod}_{\cC}(D_{\sol}(\mathbb{Q}_p))$, we have that 
    $$\mathrm{Map}_{\cA}(\cB^{\tensor n/\cA}, \cC\oplus M)\times_{\mathrm{Map}_{\cA}(\cB^{\tensor n/\cA}, \cC)}\eta \simeq \mathrm{Map}_{\cC}(\cC\tensor_{\cB^{\tensor n/\cA}}\mathbb{L}_{\cB^{\tensor n/\cA}/\cA}, M)$$
    But for the left hand side we have that 
    \begin{align*}
        \mathrm{Map}_{\cA}(\cB^{\tensor n/\cA}, \cC\oplus M)\times_{\mathrm{Map}_{\cA}(\cB^{\tensor n/\cA}, \cC)}\eta\simeq &\mathrm{Map}_{\cA}(\cB, \cC\oplus M)^{n}\times_{\mathrm{Map}_{\cA}(\cB, \cC)^n} (\eta_1\times\eta_2\times...\eta_n)\\
        \simeq& \underset{i=1,2,...n}{\prod} (\mathrm{Map}_{\cA}(\cB, \cC\oplus M)\times_{\mathrm{Map}_{\cA}(\cB, \cC)}\eta_i)\\
        \simeq & \underset{i=1,2,...n}{\prod}\mathrm{Map}_{\cC}(\eta_i^{*}\mathbb{L}_{\cB/\cA}, M)\\
        \simeq &\ \ \mathrm{Map}_{\cC}(\cC\tensor_{\cB^{\tensor n/\cA}}\mathbb{L}_{\cB^{\tensor n/\cA}/\cA}, M)
        \end{align*}
    Where $\eta_i$ correspond to the different compositions of $\cB\rightarrow \cB^{\tensor n/\cA}\rightarrow \cC$.
    Consider the cosimplicial algebra $\cB(\bullet)\simeq \mathrm{Cech}(\cA\rightarrow \cB)\in c\mathrm{CAlg}(D_{\sol}(\mathbb{Q}_p))$, denote by $F: \mathrm{Mod}_{\cB}(D_{\sol}(\mathbb{Q}_p))\rightarrow c\mathrm{Mod}_{\cB(\bullet)}(D_{\sol}(\mathbb{Q}_p))$ the left adjoint of forget full functor. The above shows that, the cosimplicial object $\mathbb{L}_{\cB(\bullet)/\cA}\in c\mathrm{Mod}_{\cA}(D_{\sol}(\mathbb{Q}_p))$ is obtained as applying $F$ to $\mathbb{L}_{\cB/\cA}$. Now by \cite{bhatt2012completionsderivedrhamcohomology}*{Lemma 2.5}, this cosimplitial object is homotopy equivalent to $0$, thus its totalizaton, $\underset{[n]\in \Delta^{\mathrm{op}}}{\lim} \mathbb{L}_{\cB^{\tensor n/\cA}/\cA}\simeq 0.$ 
\end{proof}
\begin{thm}
    Let $X$ be an Artin stack, then choosing any representable \up{smooth} chart $U\surjects X$, we have a natural equivalence $\mathbb{L}_{X^{\mathrm{an}}}^{vir}\simeq \underset{[n]\in \Delta^{\mathrm{op}}}{\lim} \mathbb{L}_{({U^{\times n/X}})^{\mathrm{an}}}$.
    Moreover, we have that $j_{X}^{*} \mathbb{L}_{X}\simeq \mathbb{L}_{X^{\mathrm{an}}}^{vir}$, where $j_X$ is defined as in \Cref{analytification def}.
\end{thm}
\begin{proof}
    Essentially the same proof as in  \Cref{analytification of cotangent complex} except that we use the \Cref{computing vir cot}.
\end{proof}
\begin{appendix}

    \section{Discrete adic space}\label{appendic A}
     In this appendix we record a proof of a well known schematic description of discrete adic space which was promised in \Cref{adic spcae and valuation rings}.
     
    Let $R$ be a commutative ring and $A$ a $R$-algebra. Following \cite{cite-key}, recall that a valuation on $A$ is a map $|\cdot|\colon A\rightarrow \Gamma\cup 0$ where $\Gamma$ is a totally ordered group such that  
    \benuma
    \item $|a+b|\leq \mathrm{Max}\{|a|,|b|\}$ for all $a,b\in A$.
    \item $|ab|=|a||b|$ for all $a,b\in A$.
    \item $|0|=0$ and $|1|=1$.
    \eenum
    The set of $|\cdot|^{-1}(0)$ is call the \textit{support} of $A$, denote by $\mathrm{supp}_{|\cdot|}$. Notice that for a valuation $|\cdot|$ on $A$, $A/\mathrm{supp}_{|\cdot|}$ is an integral ring by (2) and $\mathrm{Frac}(A/\mathrm{supp}_{|\cdot|})$ is a valuation field.
    
Let $|\mathrm{Spa}(A,R)|=\{|\cdot|\colon A\rightarrow \Gamma\cup 0| v(R)\leq 1\}/\sim$ to be the set of valuations on $A$ with values in some totally ordered abelian group $\Gamma$, where two valuations are call equivalence if they have the same support and the same sets of elements in $\mathrm{Frac}(A/\mathrm{supp}_{|\cdot|})$ whose valuation $\leq 1$ (see also \cite{wedhorn2019adicspaces}*{Definition 1.27}).

\begin{theorem}We have a bijection of sets\textup{:}
    $$|\mathrm{Spa}(A,R)|=\{\text{valuation ring V over R}, \text{R-homomorphisms }f\colon A\rightarrow \mathrm{Frac}(V) \}/\sim$$
    Here the equivalence relation is\textup{:} two valuation ring $V$,$W$ over $R$ together with homomorphisms $f\colon A\rightarrow \mathrm{Frac}(V)$, $g\colon A\rightarrow \mathrm{Frac}(W)$ are equivalent if there exist a faithful flat map $\pi\colon V\rightarrow W$ such that we have commuatative diagram 
    $$\begin{tikzcd}
A \arrow[r, "f"'] \arrow[rr, "g", bend left] & \mathrm{Frac}(V) \arrow[r, "\mathrm{Frac}(\pi)"'] & \mathrm{Frac}(W) \\
R \arrow[u] \arrow[r]                        & V \arrow[r, "\pi"] \arrow[u]                      & W. \arrow[u]     
\end{tikzcd}$$
\end{theorem}
\begin{proof} We denote by $\mathrm{Val}(A,R)$ the set $\{\text{valuation ring V over R}, \text{R-homomorphisms }f\colon A\rightarrow \mathrm{Frac}(V) \}$.
    We define the map 
    \begin{align*}
        \eta\colon \mathrm{Val}(A,R)&\rightarrow |\mathrm{Spa}(A,R)|\\
        (V,f)&\mapsto |\cdot|_{(V,f)}\colon A\rightarrow \mathrm{Frac}(V)\underset{|\cdot|_V}{\rightarrow}\Gamma_V
    \end{align*}
    where $\Gamma_V\simeq \mathrm{Frac}(V)^*/V^*$ is a valuation group of $V$ and $|\cdot|_V$ is the valuation induced from $V$.
    This is surjective since for any valuation $v\colon A\rightarrow \Gamma\in |\mathrm{Spa}(A,R)|$ , the fraction field of $A/\mathrm{supp}(v)$ is a valuation field, and the composite $R\rightarrow A\rightarrow \mathrm{Frac}(A/\mathrm{supp}(v))$ factors through its valuation ring since $v(R)\leq 1$.
    Sufficient to prove that any two homomorphisms $f\colon A\rightarrow \mathrm{Frac(V)}$, $g\colon A\rightarrow \mathrm{Frac}(W)$ induces the same valuation on $A$ via $\eta$ if only if they are equivalent.
    
    We first show the if part, we need to check that for a faithful flat map $\pi\colon V\rightarrow W$, and two maps $f\colon A\rightarrow \mathrm{Frac}(V)$, $\mathrm{Frac}(\pi)\circ f\colon A\rightarrow \mathrm{Frac}(W)$, $\eta$ will give rise to two valuation on $A$ that is equivalent. It is clear that  they have the same support. By \Cref{ff is local and inj}, the map $V\rightarrow W$ is injective and local hence $W\cap \mathrm{Frac}(V)=V$. Hence, those elements with $|\cdot|_{(W,g)}\leq 1$, i.e., images in $\mathrm{Frac}(W)$ are contained in $W$, also satisfies $|\cdot|_V\leq 1$.

    Now we prove the only if part, which will follows from that if we have a pair $(V,f)\in \mathrm{Val}(A,R)$, then the valuation ring $V'$ of $\mathrm{Frac}(A/\mathrm{supp}(v_{V,f}))$ maps to $V$ injectively and is local. By construction, $f$ factors as 
    $$A\rightarrow A/\mathrm{supp(v)}\rightarrow \mathrm{Frac}(A/\mathrm{supp(v)})\rightarrow \mathrm{Frac(V)}$$. Indeed, since kernel of $f$ is $\mathrm{supp}(v_{(V,f)})$ and $A/\mathrm{supp}(v_{(V,f)})$ is integral. Hence the valuation ring $V'$ of $\mathrm{Frac}(A/\mathrm{supp(v)}$, which is the subring with valuation $\leq 1$ maps to $V=(\mathrm{Frac}(V))^{|\cdot|_V\leq 1}$ injectively, moreover local since $V\cap \mathrm{Frac}(A/\mathrm{supp(v)}=\mathrm{Frac}(A/\mathrm{supp(v)}^{|\cdot|_V\leq 1}=V'$.
\end{proof}
\begin{lemma}\label{ff is local and inj}
        A map $V\rightarrow W $ between valuation rings is faithful flat if and only if it is injective and local.
    \end{lemma}
    \begin{proof}
        See \cite[\href{https://stacks.math.columbia.edu/tag/0539}{Tag 0539}]{stacks-project} and \cite[\href{https://stacks.math.columbia.edu/tag/00HP}{Tag 00HP}]{stacks-project}.
    \end{proof}
\end{appendix}
\bibliographystyle{plain}

\bibliography{ref.bib}
\end{document}